\documentclass{amsart}
\usepackage[shortcuts]{extdash}
\usepackage{hyperref}
\usepackage[nosort]{cite}

\usepackage{csquotes}

\usepackage{color}
\usepackage{amssymb,amsmath,amsfonts,amsthm, latexsym,amscd}
\usepackage[all]{xy}
\usepackage{fullpage}
\usepackage{epic,eepic,float}
\usepackage[mathscr]{eucal}
\usepackage{stmaryrd}
\usepackage{enumitem}

\usepackage{rotating,indentfirst,array,varioref}
\usepackage{appendix,marginnote,tikz,pgf,mathtools}
\usepackage{setspace}

 \usepackage[left=2cm,right=2cm,top=2cm,bottom=2cm,bindingoffset=0cm]{geometry}

\usepackage{tikz-cd,pgf}

\numberwithin{equation}{section} 

\def\res{\mathop{\mathrm{Res}}\limits}

\def\be{\begin{equation}}
\def\ee{\end{equation}}
\def\barr{\begin{array}}
\def\earr{\end{array}}

\newtheorem{dummy}{dummy}[section]

\newtheorem{Proposition}[dummy]{Proposition}
\newtheorem{Theorem}[dummy]{Theorem}
\newtheorem{Corollary}[dummy]{Corollary}
\newtheorem{lemma}[dummy]{Lemma}
\newtheorem{remark}[dummy]{Remark}
\newtheorem{definition}[dummy]{Definition}
\newtheorem{assumption}[dummy]{Assumption} 
\newtheorem{situation}[dummy]{Situation} 

\newtheorem{notation}[dummy]{Notation}

\newcommand{\ba}{\begin{equation}\begin{aligned}}
\newcommand{\ea}{\end{aligned}\end{equation}}

\newcommand{\bml}{\begin{multline}}
\newcommand{\eml}{\end{multline}}

\newcommand{\Si}{\Sigma}
\newcommand{\si}{\sigma}

\newcommand{\ga}{\mathfrak{g}}

\newcommand{\fB}{\mathfrak{B}}

\newcommand{\fC}{\mathfrak{C}}
\newcommand{\fD}{\mathfrak{D}}

\newcommand{\fM}{\mathfrak{M}}
\newcommand{\fP}{\mathfrak{P}}
\newcommand{\fS}{\mathfrak{S}}

\newcommand{\fX}{\mathfrak{X}}
\newcommand{\fY}{\mathfrak{Y}} 
\newcommand{\fZ}{\mathfrak{Z}}

\newcommand{\fg}{\mathfrak{g}}
\newcommand{\fh}{\mathfrak{h}}
\newcommand{\fp}{\mathfrak{p}}

\newcommand{\fz}{\mathfrak{z}}

\newcommand{\CC}{\mathbb{C}}

\newcommand{\ZZ}{\mathbb{Z}}

\newcommand{\bC}{\mathbb{C}}
\newcommand{\bE}{\mathbb{E}}
\newcommand{\bF}{\mathbb{F}}

\newcommand{\bH}{\mathbb{H}}
\newcommand{\bK}{\mathbb{K}}
\newcommand{\bL}{\mathbb{L}}

\newcommand{\bP}{\mathbb{P}}
\newcommand{\bQ}{\mathbb{Q}}
\newcommand{\bR}{\mathbb{R}}

\newcommand{\bZ}{\mathbb{Z}}

\newcommand{\bfL}{\mathbf{L}}
\newcommand{\one}{\mathbf{1}}

\newcommand{\bw}{\mathbf{w}}
\newcommand{\bSi}{\mathbf{\Si}}

\newcommand{\bfA}{\mathbf{A}}
\newcommand{\bfC}{\mathbf{C}}
\newcommand{\bfU}{\mathbf{U}}
\newcommand{\bfX}{\mathbf{X}}
\newcommand{\bfZ}{\mathbf{Z}}

\newcommand{\cA}{\mathcal{A}}
\newcommand{\cB}{\mathcal{B}}
\newcommand{\cC}{\mathcal{C}}

\newcommand{\cE}{\mathcal{E}}
\newcommand{\cF}{\mathcal{F}}
\newcommand{\cH}{\mathcal{H}}
\newcommand{\cI}{\mathcal{I}}
\newcommand{\cL}{\mathcal{L}}

\newcommand{\cO}{\mathcal{O}}

\newcommand{\cT}{\mathcal{T}}
\newcommand{\cV}{\mathcal{V}}
\newcommand{\cX}{\mathcal{X}}

\newcommand{\cOb}{\cO b}

\newcommand{\Q}{\mathrm{Q}}

\newcommand{\rank}{\mathrm{rank}}

\newcommand{\dd}{\mathrm{d}}
\newcommand{\mempty}{\varnothing}

\newcommand{\ord}{\mathrm{ord}}

\newcommand{\Ker}{\mathrm{Ker}}

\newcommand{\Sg}{\mathrm{Sg}}

\newcommand{\Hom}{\mathrm{Hom}}
\newcommand{\Aut}{\mathrm{Aut}} 
\newcommand{\Ext}{\mathrm{Ext}}
\newcommand{\Pic}{\mathrm{Pic}}
\newcommand{\Spec}{\mathrm{Spec}}

\newcommand{\ev}{ {\mathrm{ev}} }

\newcommand{\MF}{\mathrm{MF}}

\newcommand{\GL}{\mathrm{GL}}
\newcommand{\vir}{\mathrm{vir}}
\newcommand{\ch}{\mathrm{ch}}
\newcommand{\tw}{ {\mathrm{tw}} }
\newcommand{\pr}{ {\mathrm{pr}} }
\newcommand{\eff}{ {\mathrm{eff}} }

\newcommand{\pre}{ {\mathrm{pre}} } 
\newcommand{\td}{\mathrm{td}}

\newcommand{\im}{\mathrm{Im}}

\newcommand{\inv}{\mathrm{inv}}
\newcommand{\age}{\mathrm{age}}
\newcommand{\tot}{\mathrm{tot}}

\newcommand{\ct}{ {\mathrm{ct} } }

\newcommand{\Perf}{\mathrm{Perf}}

\newcommand{\Ob}{\mathrm{Ob}}
\newcommand{\Sym}{\mathrm{Sym}}
\newcommand{\Crit}{\mathrm{Crit}} 

\newcommand{\ess}{ {\mathrm{ess}} }
\newcommand{\noness}{ {\mathrm{noness}} }

\newcommand{\DD}{\mathscr{D}}

\newcommand{\sF}{\mathscr{F}}
\newcommand{\sV}{\mathscr{V}}

\newcommand{\sX}{\mathscr{X}}

\newcommand{\sZ}{\mathscr{Z}}

\newcommand{\D}{\mathsf{D}}
\newcommand{\St}{\mathsf{t}} 

\newcommand{\tD}{\widetilde{D}}

\newcommand{\tM}{\widetilde{M}}
\newcommand{\tN}{\widetilde{N}}
\newcommand{\tT}{\widetilde{T}}

\newcommand{\te}{\widetilde{e}}
\newcommand{\tit}{\widetilde{t}}

\newcommand{\tch}{\widetilde{\ch}}
\newcommand{\tbL}{\widetilde{\bL}} 

\newcommand{\tGa}{\widetilde{\Gamma}}

\newcommand{\tzeta}{\tilde{\zeta}} 

\newcommand{\hc}{\hat{c}}
\newcommand{\hG}{\hat{G}}
\newcommand{\hq}{\hat{q}}

\newcommand{\ufX}{\underline{\fX}}

\newcommand{\pd}{\partial}

\newcommand{\lra}{\longrightarrow}
 
\newcommand{\Boxs}{\mathrm{Box}(\si)}

\begin{document}

\title{Higgs-Coulomb Correspondence and wall-crossing  in Abelian GLSMs}

\author{Konstantin Aleshkin}
\address{Konstantin Aleshkin, Department of Mathematics, Columbia University, 2990 Broadway, New York, NY 10027, USA}
\email{aleshkin@math.columbia.edu}

\author{Chiu-Chu Melissa Liu} 
\address{Chiu-Chu Melissa Liu, Department of Mathematics, Columbia University, 2990 Broadway, New York, NY 10027, USA}
\email{ccliu@math.columbia.edu}

\dedicatory{Dedicated to the memory of Professor Bumsig Kim}

\begin{abstract}
We compute $I$-functions and central charges for abelian GLSMs using virtual matrix factorizations of Favero and Kim. In the Calabi-Yau case we provide analytic continuation for the central charges by explicit integral formulas. The integrals in question are called hemisphere partition functions and we call the integral representation Higgs-Coulomb correspondence. We then use it to prove GIT stability wall-crossing for central charges.

\end{abstract} 

\maketitle

\flushbottom

\setcounter{tocdepth}{1} 

\tableofcontents


\section{Introduction}
2d gauged linear sigma models (GLSMs) were introduced
by Witten in 1993.  Following \cite{FJR}, the input data of a GLSM is a 5-tuple $(V,G,\bC_R^*, W, \zeta)$,
where $V$ is a finite dimensional complex vector space, $G\subset \GL(V)$ is a reductive linear group known as the gauged group,
$\bC_R^*\cong \bC^*$ acts linearly on $V$ and the action commutes with the $G$-action, $W: V\to \bC$ is a $G$-invariant polynomial which is quasi-homogeneous with respect to the $\bC_R^*$-action, and $\zeta$ is a $G$-character with the property $V_G^{ss}(\zeta)=V_G^s(\zeta)$, i.e.,  every $\zeta$-semistable point is $\zeta$-stable, so that the GIT quotient stack $\sX_\zeta = [V/\!/\!_\zeta G] = [V^{ss}_G(\zeta)/G]$ is an orbifold (i.e. smooth DM stack with trivial generic stabilizer). The space of stability conditions is  $\hG \otimes_{\bZ}\bR \cong \bR^{\dim_\bC Z(G)}$, where $\hG = \Hom(G,\bC^*)$ is the group of 
$G$ characters and $Z(G)$ is the center of $G$; it is decomposed into chambers called phases. 
The $G$-invariant polynomial $W$ descends to $\bw_\zeta:\sX_\zeta\to \bC$; the pair $(\sX_\zeta, \bw_\zeta)$ is a Landau-Ginzburg (LG) model, where 
$\bw_\zeta$ is known as the superpotential.   GLSM invariants of $(V,G,\bC_R^*,W,\zeta)$ are, roughly speaking, virtual counts of curves in the critical locus 
$\sZ_\zeta :=\Crit(\bw_\zeta) = [\left(\Crit(W)\cap V^{ss}_G(\zeta)\right)/G]$ which is often assumed to be compact/proper but can be singular.

\smallskip

Different phases of a GLSM (that is GLSMs which differ only by the choice of
a stability parameter) are closely related to each other. For example, let $G=\bC^*$ act on $V=\bC^6$ by weights $(1,1,1,1,1,-5)$ with $W=p W_5$ where $W_5= x_1^5+\cdots + x_5^5$ is the Fermat quintic polynomial in 5 variables. In the CY/geometric phase $\zeta>0$,   $\sX_\zeta = K_{\bP^4}$, 
$\sZ_\zeta  =  X_5 := \{ W_5=0\}  \subset \bP^4$ is the Fermat quintic threefold, and GLSM invariants are (up to sign) Gromov-Witten (GW) invariants of $X_5$. 
In the LG phase $\zeta<0$,  $\cX_{\zeta} \cong [\bC^5/\mu_5]$, where $\mu_5$ is the group of 5-th roots of unity acts diagonally on $\bC^5$,  $\sZ_\zeta$ is supported at the origin, and GLSM invariants are Fan-Jarvis-Ruan-Witten (FJRW) invariants of the affine LG model $([\bC^5/\mu_5], W_5)$.  
Chiodo-Ruan \cite{CR} proved genus-zero LG/CY correspondence for quintic threefolds  relating GW invariants of $X_5$ and FJRW invariants of $([\bC^5/\mu_5], W_5)$.  Their proof can be summarized into two steps.
 \begin{enumerate} 
 \item ($\epsilon$-wall-crossing) In the CY (resp. LG) phase, the Givental-style mirror theorem says  the $J$-function which governs the genus-zero GW (resp. FJRW) is related to the $I$-function, which can be expressed in terms of explicit
 hypergeometric series, by explicit change of variables known as the mirror map.
 \item ($\zeta$-wall-crossing) The $I$-function admits a Mellin-Barnes integral representation. $I$-functions in the two phases are  related by analytic continuation given by deforming the contour of integration in $\bC$.
 \end{enumerate}
The interpretation of Step (1) as $\epsilon$-wall-crossing appeared in later work. For each $\epsilon\in \bQ_{>0}$, Ciocan-Fontanine--Kim--Maulik \cite{CKM}  introduced
$\epsilon$-stable quasimaps to certain GIT quotient $W/\!/G$. Ciocan-Fontanine and Kim \cite{CK} introduced $J^\epsilon$ which is a generating
function of invariants defined by genus-zero $\epsilon$-stable quasimaps; $J^\epsilon$ specializes to the $I$-function and the $J$-function as
$\epsilon\to 0^+$ and $\epsilon\to +\infty$, respectively. In the presence of a good torus action, they proved
$\epsilon$-wall-crossing which relates $J^\epsilon$ to the $J$-function $J=J^\infty$  for any $\epsilon\in \bQ_{>0}$, by change of variables. They 
computed the $I$-function $I=J^{0+}$ explicitly. In particular, they recover the mirror theorem in the geometric phase first proved
by Givental  \cite{Gi96} and Lian-Liu-Yau \cite{LLY}. The mirror theorem in the LG phase was first proved by Chiodo-Ruan \cite{CR} and later reproved by Ross-Ruan \cite{RR}  via $\epsilon$-wall-crossing.

\smallskip

For a general GLSM, GLSM invariants are defined by integrating against virtual cycles on moduli of $\epsilon$-stable LG quasimaps. The virtual cycle is constructed for narrow sectors by Fan-Jarvis-Ruan \cite{FJR} via cosection localization,  and for both narrow and broad sectors by Favero-Kim \cite{FK} via matrix factorization; Favero-Kim's construction generalizes previous constructions for affine LG models \cite{PV16} and for convex hybrid models \cite{CFGKS}.
When $\epsilon>0$, the definition relies on a good lift $\tilde{\zeta}$ which is a character of the group $\Gamma\subset \GL(V)$ generated by $G$ and $\bC_R^*$, such that $V^{ss}_\Gamma(\tilde{\zeta}) = V^{ss}_G(\zeta)$. Such a good lift does not always exist. At $\epsilon=0^+$, a good lift is not needed.
 In this paper we focus on the $\epsilon \to 0^{+}$ stability condition and study genus-zero GLSM invariants and $\zeta$-wall-crossing for
 abelian GLSMs where $G=(\bC^*)^{\kappa}$. In this case, $\sX_\zeta$ is  a smooth toric DM stack. Let $\tT$-be the diagonal subgroup of $\GL(V)$.
 Using the work of Favero-Kim \cite{FK}, we define and compute K-theoretic GLSM $I$-function $I^K_{\bw}$ which takes values in the K-theory of category 
 of matrix factorizations on the inertial stack $I\sX_\zeta$, and the (cohomological) GLSM $I$-function $I_{\bw}$ which takes values in the GLSM state space 
 $H_{\bw}$.  We also define and compute K-theorectic $\tT$-equivariant $I$-function $I^K_{\tT}$ of $(V,G, \bC_R^*, 0,\zeta)$ which takes values in $K_{\tT}(I\sX_\zeta)$, 
 and $\tT$-equivariant $I$-function $I_{\tT}$ of $(V,G,\bC_R^*,0,\zeta)$ which takes values in $H_{\tT}(I\sX_\zeta)$. 

\smallskip

 In Gromov-Witten theory $I$-functions fail to capture integral structure on cohomology~\cite{Iritani}. Hosono 
 defined an object called a central charge that sees the Gamma integral structure. Integral structures are crucial for integral representations. Central charges are power series which are constructed
 from both $J$-function and objects of the derived category of coherent sheaves of the target manifold. Hori and Romo~\cite{HR} constructed explicit analytic functions
  called hemisphere partition functions and conjectured that their power series expansions are equal to the central charges in appropriate cases.
We define the GLSM central charge of a matrix factorization $\fB$ of $(\sX_\zeta,\bw_\zeta)$ as
$Z_{\bw}(\fB) = \langle I_{\bw}, \Gamma_\bw \ch_{\bw}( [\fB])\rangle$ where  $\Gamma_{\bw}$ is an appropriate version of Iritani's $\Gamma$-class
and $[\fB]$ is the K-theory class of $\fB$. 
We also define the $\tT$-equivariant central charge of a $\tT$-equivariant perfect complex $\fB$ on $\sX_\zeta$
as $Z^{\tT}(\fB) = \langle I_{\tT}, \Gamma_{\tT}\ch_{\tT}([\fB])\rangle$.   
We show that our central charges indeed have integral representations of the hemisphere partition function form (Theorem \ref{the:expansion}). We call this representation
  Higgs-Coulomb correspondence because in physics GLSM central charges can be obtained by a version of the Higgs branch localization and hemisphere partition functions are computed by the Coulomb branch localization (c.f.~\cite{BeniniCremonesi}). The integral
  representations we obtain depend continuously on the complexified stability parameter $\theta = \zeta +2\pi\sqrt{-1}B$ and do not have any restrictions on $\zeta$.

\smallskip

  In this philosophy $\zeta$-wall-crossing follows immediately by analytic
  continuation in $\zeta$ (Theorem~\ref{th:wallCrossing}). Remarkably, matrix factorizations of the central
  charges in question are related by the Fourier-Mukai transform~\cite{BP10}.
  Let $\zeta_{\pm}$ represent stability conditions in two
  adjacent chambers and $\cB_{+}$ be a matrix factorizations in the phase corresponding
  to $\zeta_{+}$. Then, the analytic continuation of the central charge of $\cB_{+}$
  is a central charge of $\cB_{-} = \mathrm{FM}(\cB_{+})$. Particular Fourier-Mukai
  kernel is choosen by $B = \im(\theta)/2\pi$ and convergence of the integral
  representation is equivalent to the so-called Grade Restriction Rule~\cite{HL, BFK, CIJS}. 

 \begin{equation}
    \begin{tikzcd}
      & \cB \arrow[ld,left,"\pi_{+}"'] \arrow[rd,"\pi_{-}"] &   \\
      \cB_{+} \arrow[rr, dashed, "FM"] &    & \cB_{-}
    \end{tikzcd}
  \end{equation}

\subsection*{Acknowledgements}  
We wish to thank Daniel Halpern-Leistner, Kentaro Hori, Hiroshi Iritani, Andrei Okounkov, Tudor P\u{a}durariu, Renata Picciotto, Alexander Polishchuk,  Che Shen, Yefeng Shen, Mark Shoemaker, and Yang Zhou for helpful communications. 
We thank the hospitality and support of the Simons Center for Geometry and Physics (SCGP) during the program {\em Integrability, Enumerative Geometry and Quantization} (August 22-September 23, 2022) where part of the paper was completed. The authors are partially supported by NSF grant DMS-1564497.

\section{Geometry  of gauged linear sigma models and $A$-model state spaces}  \label{sec:geometry} 
In this paper, all the schemes and algebraic stacks are defined over $\Spec \, \bC$, where $\bC$ is the field of complex numbers. 

\subsection{Gauged linear sigma models} \label{sec:GLSM}
We start with the setup of a general gauged linear sigma model (GLSM) following
\cite{FJR, FK}. Part of our formulation in Section \ref{sec:Higgs} is closer to that in the more general setting in \cite{CJR}. 

The input data of a GLSM is a 5-tuple $(V, G, \bC_R^*, W, \zeta)$, where 
\begin{enumerate}
\item (linear space)  $V =\Spec \bC[x_1,\ldots, x_{n+{\boldsymbol{\kappa}}}]\simeq  \bC^{n+{\boldsymbol{\kappa}}}$ is a finite dimensional complex vector space, where ${\boldsymbol{\kappa}} =\dim G$. 
\item (gauge group) $G\subset \GL(V)$ is a reductive algebraic group. 
\item (vector R-symmetry) $\bC^*_R\cong \bC^*$ acts linearly and faithfully on $V$, so we may view $\bC^*_R$ as a subgroup of $\GL(V)$. 
We assume that 
\begin{enumerate}
\item the intersection $G\cap \bC^*_R$ is finite, and
\item the $\bC^*_R$-action commutes with the $G$-action. 
\end{enumerate}
The finite group $G\cap \bC^*_R$ must be cyclic, generated by an element $J$ of finite order $r\in \bZ_{>0}$, given explicitly in Equation \eqref{eqn:J} below.  The  surjective group homomorphism
$\bC^*_R \to \bC^*_\omega := \bC^*/\langle J\rangle$ is a degree $r$ covering map. Let $\Gamma := G\bC_R^* \subset \GL(V)$ be the subgroup generated
by $G$ and $\bC_R^*$. Then we have a short exact sequence  of groups:
\begin{equation}
1\to G\to \Gamma \stackrel{\chi}{\to} \bC^*_\omega \to 1.
\end{equation}  
By (b), the $\Gamma$-action on $V$ induces
a $\bC^*_\omega$-action on the smooth Artin stack $[V/G]$ in the sense of \cite{Ro}. 
The  $R$-charges are
$$
(q_1,\ldots,q_{n+{\boldsymbol{\kappa}}}) =\Big(\frac{2c_1}{r},\ldots, \frac{2c_{n+{\boldsymbol{\kappa}}}}{r}\Big),
$$
where $c_1,\ldots,c_{n+{\boldsymbol{\kappa}}} \in \bZ$ are the weights of the $\bC^*_R$-action on $V\simeq \bC^{n+{\boldsymbol{\kappa}}}$.  (Note that $\mathrm{gcd}(c_1,\ldots, c_{n+{\boldsymbol{\kappa}}})=1$ since
$\bC_R^*$ acts faithfully on $V$.)  Let
\begin{equation}\label{eqn:J}
J= (e^{2\pi\sqrt{-1}c_1/r},\ldots, e^{2\pi\sqrt{-1}c_{n+{\boldsymbol{\kappa}}}/r}).
\end{equation}
\item (superpotential) $W: V\to \bC$  is a $G$-invariant regular function which is a quasi-homogeneous polynomial of degree $r$ with respect to the $\bC^*_R$-action on $V$; in other words,
$W\in \bC[x_1,\ldots, x_{n+{\boldsymbol{\kappa}}}]^G$ and 
$$
W(t^{c_1} x_1,\ldots, t^{c_{n+{\boldsymbol{\kappa}}}}x_{n+{\boldsymbol{\kappa}}}) = t^r W(x_1,\ldots,x_{n+{\boldsymbol{\kappa}}}),\quad
t\in \bC_R^*,\quad (x_1,\ldots,x_{n+{\boldsymbol{\kappa}}})\in V. 
$$
It descends to a function 
$\bw: [V/G]\to \bC$ of degree 1 with respect to the $\bC^*_\omega$-action on $[V/G]$.
\item (stability condition) $\zeta \in  \Hom(G,\bC^*) =\Hom(G^{\text{ab}},\bC^*)$, where $G^{\text{ab}} = G/[G,G]$ is the abelianization of $G$.
We view $\Hom(G,\bC^*)$ as an additive group and let $\chi^\zeta: G\to \bC^*$ denote the associated $G$-character. 
Let $V^{ss}_G(\zeta)$ (respectively $V^s_G(\zeta)$) be the set of semistable (respectively stable)
points in $V$ determined by the $G$-linearization on the trivial line bundle $V\times \bC\to V$ given by $g\cdot (v,t) = (g\cdot p, \chi^\zeta(g)t)$.   We assume that
$V^s_G(\zeta) = V^{ss}_G(\zeta)$. Then the quotient stack
$\sX_\zeta := [V^{ss}_G(\zeta) /G]$  is an orbifold (i.e. a smooth Deligne-Mumford stack with trivial generic stabilizer) of dimension $n$. The GIT quotient 
$X_\zeta:= V\sslash_\zeta G = V^{ss}_G(\zeta) /G$ is the coarse moduli space of $\sX_\zeta$. 
\end{enumerate} 
Let $\mathscr{Z}_\zeta := [ (\mathrm{Crit}(W)\cap V^{ss}_G(\zeta))/G]$ be the critical locus of 
$\bw_\zeta:= \bw|_{\sX_\zeta} :\sX_\zeta\to \bC$. We say $\zeta$ is in a geometric phase if $\mathrm{Crit}(W)\cap V^{ss}(\zeta)$ is non-singular,
which implies $\mathscr{Z}_\zeta$ is an orbifold.   

The {\em central charge} of the GLSM is  (cf. \cite[Definition 3.2.3]{FJR})  
\begin{equation}
\hc := \dim V -\dim G -  2\hq = n- 2\hq,\quad \text{where } \hq= \frac{1}{2} \sum_{j=1}^{n+{\boldsymbol{\kappa}}} q_j 
=\frac{1}{r}\sum_{j=1}^{n+\boldsymbol{\kappa}} c_j. 
\end{equation}
\begin{remark}
By the condition (4) above, the superpotential $W$ is  semi-invariant with respect to $\Gamma$ in the sense of \cite[Section 2]{PV11}, i.e., 
\begin{equation}\label{eqn:semi-invariant}
W(\gamma\cdot x) = \chi(\gamma)W(x) \quad \text{for any }\gamma\in \Gamma \text{ and } x\in V.
\end{equation} 
\end{remark} 

\begin{remark} In this paper the stability condition  $\zeta$ is an element in $\Hom(G,\bC^*)$  and corresponds to the symbol $\theta$ in \cite{FJR, FK}. 
The 1-dimensional torus $\bC^*_\omega$ in this paper corresponds to $\mathbf{C}^*_\omega$ in \cite{CJR}.  
\end{remark} 

\begin{remark} At this point, we do not assume the critical locus $\mathscr{Z}_\zeta$ is proper.    This allows us to include the
case without superpotential as  a special case where the superpotential and the R-charges are zero, i.e. $W=0$ and $q_j=0$ for $j=1,\ldots, n+{\boldsymbol{\kappa}}$;
note that in this special case the $\bC_R^*$-action on $V$ is trivial, and in particular, not faithful. 
\end{remark}

\subsection{Abelian GLSMs} \label{sec:abelian} 
In the rest of this paper, we consider abelian GLSMs  where the gauge group $G =  (\bC^*)^{\boldsymbol{\kappa}}$ is a complex algebraic torus.  In this case
$[G,G]=\{1\}$ and $G^{\mathrm{ab}}= G$. 

Up to an inner automorphism of $\GL(V)$, we may assume the image of
$\rho_V: G\to \GL(V)$  is contained in the diagonal torus  $\tT\simeq (\CC^*)^{n+{\boldsymbol{\kappa}}} \subset \GL(V)\cong \GL_{n+{\boldsymbol{\kappa}}}(\bC)$. 
Then $\sX_\zeta$ is an $n$-dimensional  toric orbifold, and its coarse moduli $X_\zeta$ is a semi-projective simplicial toric variety which contains
$T:=\tT/G \cong (\bC^*)^n$ as a Zariski dense open subset.  We have a short exact sequence of abelian groups
\begin{equation}\label{eqn:GtT}
1\to G \stackrel{\rho_V}{\lra} \tT \lra T \to 1.
\end{equation}

\begin{remark}
The notation in this subsection is similar to but slightly different from that in \cite{CIJ}: $G$, $\tT$, and $T$ in this paper correspond to $K$, $T$, and
$\mathcal{Q}$ in \cite[Section 4.3]{CIJ}, respectively; ${\boldsymbol{\kappa}}$ and $n$ in this paper correspond  to $r$ and $m-r$ in \cite[Section 4.1]{CIJ}, respectively. 
\end{remark}

\begin{itemize}
\item Applying $\Hom(\bC^*, -)$ to \eqref{eqn:GtT}, we obtain the following short exact sequence of cocharater lattices
\begin{equation}\label{eqn:LtN2} 
0 \to \bL :=\Hom(\bC^*,G) \lra \tN := \Hom(\bC^*, \tT)\lra N :=\Hom(\bC^*,T) \to 0,
\end{equation} 
where $\bL\cong \bZ^{{\boldsymbol{\kappa}}}$, $\tN\cong \bZ^{n+{\boldsymbol{\kappa}}}$, and $N\cong \bZ^n$. 
\item  Applying $\Hom(-,\bC^*)$ to \eqref{eqn:GtT}, or equivalently dualizing \eqref{eqn:LtN2}, we obtain the following short exact sequence of character lattices:
\begin{equation}\label{eqn:MML}
0 \to M :=\Hom(T,\bC^*) \lra \tM := \Hom(\tT,\bC^*) \lra \bL^\vee := \Hom(G,\bC^*) \to 0
\end{equation}
\end{itemize}

The map $\bL \to \tN \cong \bZ^{n+{\boldsymbol{\kappa}}}$ is given by $(\D_1,\ldots,\D_{n+{\boldsymbol{\kappa}}})$ where $\D_i \in \bL^\vee$. The stability condition $\zeta$ is an element in $\Hom(G,\bC^*)=\bL^\vee$. 
If we choose a $\bZ$-basis $\{\xi_1,\ldots,\xi_{\boldsymbol{\kappa}}\}$ of $\bL$ (which is equivalent to a choice of an isomorphism $G\simeq (\bC^*)^{\boldsymbol{\kappa}}$) and let $\{\xi_1^*,\ldots,\xi_{\boldsymbol{\kappa}}^*\}$ be the dual $\bZ$-basis of $\bL^\vee$, then
$$
\D_i =\sum_{a=1}^{\boldsymbol{\kappa}} Q_i^a  \xi_a^*
$$
for some $Q_i^a \in \bZ$. Given any $\St =\sum_{a=1}^{\boldsymbol{\kappa}} t_a \xi_a^* \in \bL^\vee$, where $t_1,\dots, t_{\boldsymbol{\kappa}}\in \bZ$, let
$\chi^{\St}:G\to \bC^*$ be the corresponding $G$ character given by 
\begin{equation}
\chi^{\St}(s_1,\ldots, s_{\boldsymbol{\kappa}})= \prod_{a=1}^{\boldsymbol{\kappa}} s_a^{t_a}.
\end{equation}
Then the map $G\simeq (\bC^*)^{\boldsymbol{\kappa}} \to \tT \cong (\bC^*)^{n+{\boldsymbol{\kappa}}}$ is given by 
\begin{equation}\label{eqn:chi}
s  =(s_1,\ldots,s_{\boldsymbol{\kappa}}) \in G  \mapsto \left( \chi^{\D_1} (s), \ldots, \chi^{\D_{n+{\boldsymbol{\kappa}}}}(s) \right) \in \tT, \quad \text{where } \chi^{\D_i}(s) =\prod_{a=1}^{\boldsymbol{\kappa}} s_a^{Q_i^a}.
\end{equation} 
Given a lattice  $\Lambda \cong \bZ^r$ and a field $\bF$, we define
$\Lambda_\bF := \Lambda\otimes_{\bZ} \bF \cong \bF^r$; in this paper, $\bF =\bQ$, $\bR$, or $\bC$. 

\begin{remark}
The map $G\to \tT$ is injective iff $\D_1,\ldots, \D_{n+{\boldsymbol{\kappa}}}$ generate the lattice $\bL^\vee$ over $\bZ$. In Section \ref{sec:Coulomb} (The Coulomb Branch), we work with the weaker assumption that $\D_1,\ldots, \D_{n+{\boldsymbol{\kappa}}}$  span the vector space $\bL^\vee_\bQ$ over $\bQ$, or equivalently, the kernel $K$ of  the group homomorphism $G\to \tT$ is finite. Then $\sX_\zeta$ is a smooth toric DM stack with a generic stabilizer $K$; it is a toric orbifold iff $K$ is trivial. It is also possible to work in this generality in Section \ref{sec:Higgs} (The Higgs branch). We assume $K$ is trivial in Section \ref{sec:Higgs}  mainly because \cite{FJR} and \cite{FK} assume so, but $K$ can be non-trivial in orbifold quasimap theory \cite{CCK} which can be viewed as a mathematical theory of  GLSM without superpotential.
\end{remark} 

Let  $G_{\bR}  \simeq  U(1)^{\boldsymbol{\kappa}}$ be the maximal compact subgroup  of $G \simeq (\bC^*)^{\boldsymbol{\kappa}}$. Then $\bL_\bR$ is 
canonically isomorphic to the Lie algebra $\fg_\bR$ of $G_\bR$. 
The $G$-action on $V=\Spec\bC[x_1,\ldots,x_{n+{\boldsymbol{\kappa}}}]$ restricts to a Hamiltonian $G_\bR$-action  on the K\"{a}hler manifold
$(V, \frac{\sqrt{-1}}{2} \sum_{i=1}^{n+{\boldsymbol{\kappa}}} dx_i\wedge d\bar{x}_i)$ with a moment map
$$
\mu: V \longrightarrow  \fg_\bR^\vee= \bL_\bR^\vee, \quad (x_1,\ldots,x_{n+{\boldsymbol{\kappa}}}) = \frac{1}{2} \sum_{a=1}^{\boldsymbol{\kappa}} Q_i^a |x_i|^2 \xi^*_a.
$$
Then 
$$
\sX_\zeta = [ V^{ss}_G(\zeta)/G] = [\mu^{-1}(\zeta)/G_\bR].
$$
From this perspective, the stability condition $\zeta$ is a regular value of the moment map $\mu$, and can be an element in $\bL_\bR^\vee \cong \bR^{{\boldsymbol{\kappa}}}$. 

\subsection{Anticones and the extended stacky fan}  \label{sec:anticones} 
The triple $(V,G,\zeta)$, which is part of the input data $(V,G,\bC_\bR^*,W,\zeta)$ of the given abelian GLSM, determines
a set $\cA_\zeta$ of anticones and an extended stacky fan $\bSi_\zeta= (N,\Sigma_\zeta,\beta,S_\zeta)$. We describe them
in this subsection, and describe $V_G^{ss}(\zeta)\subset V$ in terms of anticones.

We fix an isomorphism $V=\Spec\bC[x_1,\ldots,x_{n+{\boldsymbol{\kappa}}}] \cong \bC^{n+{\boldsymbol{\kappa}}}$, which determines an ordered $\bZ$-basis
$(\te_1,\ldots, \te_{n+{\boldsymbol{\kappa}}})$ of $\tN$. In particular,
$$
\tN =\bigoplus_{i=1}^{n+{\boldsymbol{\kappa}}}\bZ \te_i. 
$$
Let $v_i\in N$ be the image of $e_i$ under $\tN\to N$. Define $\beta=(v_1,\ldots,v_{n+{\boldsymbol{\kappa}}})$. 

Given a subset $I$ of $\{1,\ldots,n+{\boldsymbol{\kappa}}\}$, let $I' =\{1,\ldots,n+{\boldsymbol{\kappa}}\}\setminus I$ be its complement, and define
$$
\angle_I =\{ \sum_{i\in I} a_i \D_i: a_i \in \bR, a_i >0\} \subset \bL^\vee_\bR, \quad 
\sigma_I =  \{ \sum_{i\in I} a_i v_i: a_i\in \bR, a_i\geq 0\} \subset N_\bR. 
$$
If $I=\emptyset$ is the empty set, define $\sigma_{\emptyset} =\{0\}$.  

For a fixed $G$-action on $V$, a stability condition $\zeta\in \bL_\bR^\vee$ determines the following three sets. 
\begin{eqnarray*}
\cA_\zeta &=& \{ I\subset \{1,\ldots,n+{\boldsymbol{\kappa}}\}: \zeta\in \angle_I\}, \\
\Sigma_\zeta &=& \{ \sigma_I : I'  \in \cA_\zeta\},\\
S_\zeta &=& \{ i\in \{1,\ldots, n+{\boldsymbol{\kappa}}\}: \sigma_{\{i\}} \notin \Sigma_\zeta  \} = \{ i\in \{1,\ldots,n+{\boldsymbol{\kappa}}\}: \{i\}'  \notin \cA_
\zeta\}.
\end{eqnarray*}
Note that $S_\zeta \subset I$ for any $I\in \cA_\zeta$.  

\begin{assumption} \label{simplical} 
We choose the stability condition $\zeta\in \bL_\bR^\vee$ such that  the following three equivalent conditions
are satisfied.
\begin{enumerate}
\item[(i)] For any  $I \in\cA_\zeta$, $\{ \D_i: i\in I\}$ spans $L_\bQ^\vee$ as a vector space over $\bQ$. 
\item[(ii)] For any $I\in \cA_\zeta$, $\{v_i : i \in I' \}$ is a set of linearly independent vectors in $N_\bQ$, or equivalently,
$\sigma_{I'}$ is a simplicial cone in $N_\bR$.
\item[(iii)] $\Sigma_\zeta$ is a simplicial fan in $N_\bR$. 
\end{enumerate}
\end{assumption} 

Elements in $\Sigma_\zeta$ are called {\em cones},  while elements in $\cA_\zeta$ are called {\em anticones}. 
By (i), the cardinality $|I|$ of any anticone $I\in \cA_\zeta$ is greater or equal to ${\boldsymbol{\kappa}}$.  Let
$\Sigma_\zeta(d)$ be the set of $d$-dimensional cones in $\Sigma_\zeta$. Then $\sigma \in \Sigma_\zeta(d)$ iff
$\sigma =\sigma_I$ where $|I|=d$ and $I' \in \cA_\zeta$. Let
$$
\cA^{\min}_\zeta =\{ I\in \cA_\zeta: |I|={\boldsymbol{\kappa}}\} = \{ I\in \cA_\zeta: \sigma_{I'} \in \Sigma_\zeta(n)\} 
$$
be the set of minimal anti-cones. 

The {\em irrelevant ideal} of $\zeta$ is the ideal $B_\zeta$ in $\bC[x] :=\bC[x_1,\ldots, x_{n+{\boldsymbol{\kappa}}}]$ generated
by $\{ x_I := \prod_{i\in I} x_i : I\in \cA_\zeta\}$.  Let $Z_\zeta = Z(B_\zeta) $ be the closed subvariety of $V =\Spec\bC[x]$  defined by 
the irrelevant ideal $B_\zeta\subset \bC[x]$, and let $U_\zeta = V\setminus Z_\zeta$.  If $\zeta \in \bL^\vee$ is a $G$ character,  then 
$U_\zeta = V^{ss}_G(\zeta) $, and  $Z_\zeta = V^{un}_G(\zeta)$ is the set of unstable points defined by $\zeta$.

For any $I\in\cA_\zeta$, define 
\begin{equation}\label{eqn:VI}
V_I =  V \setminus Z(x_I)=  \{(x_1,\ldots,x_{n+{\boldsymbol{\kappa}}})\in V: x_i\neq 0 \text{ if } i\in I\}  = (\bC^*)^I\times \bC^{I'}.
\end{equation}
Then
$$
U_\zeta  =\bigcup_{I\in \cA_\zeta} V_I. 
$$
Note that if $I, J\in \cA_\zeta$ and $I\subset J$ then $V_J\subset V_I$. Therefore,
$$
U_\zeta = \bigcup_{I\in \cA_\zeta^{\min}} V_I. 
$$

We now give an alternative description of $Z_\zeta$.  If $I\subset \{1,\ldots,n+{\boldsymbol{\kappa}}\}$,  and $I\notin \cA_\zeta$, or equivalently $\sigma_{I'} \notin \Sigma_\zeta$, define
$$
Z_I = \bC^{I}\times \{0\} ^{I'} = \{ (x_1,\ldots, x_{n+{\boldsymbol{\kappa}}}) \in V: x_i =0 \text{ if } i \in I'\}. 
$$
Then
$$
Z_\zeta = \bigcup_{ I\notin\cA_\zeta } Z_I.
$$

Define 
$$
C_\zeta := \bigcap_{I\in \cA_\zeta} \angle_I = \bigcap_{I\in\cA^{\min}_\zeta} \angle_I \subset L_\bR^\vee.
$$
The open cone $C_\zeta$ is  called the {\em extended ample cone} in \cite{CIJ}.  It is  a chamber in the space of stability conditions: if $\zeta'\in C_\zeta$ then
$\Sigma_{\zeta'} =\Sigma_\zeta$ and $U_{\zeta'} = U_\zeta$.  We recall the following facts.
\begin{enumerate}
\item The quotient stack  
$$
 \sX_\zeta =[U_\zeta /G] 
 $$
is the smooth toric Deligne-Mumford (DM) stack defined by  the stacky fan $(N,\Sigma_\zeta,\beta)$. See Borisov-Chen-Smith \cite{BCS} for definition of toric Deligne-Mumford stacks in terms of stacky fans. 
\item  The coarse moduli space of $\sX_\zeta$ is the categorical (and geometric) quotient
$$ 
X_\zeta = U_\zeta /G  
$$
which is the toric variety defined by the simplicial fan $\Sigma_\zeta \subset N_\bR$.  
See \cite{Fu93, CLS} for  an introduction of toric varieties, and in particular the definition  
of general normal toric varieties in terms of fans.
\item  If $\zeta\in \bL^\vee$ is a $G$-character then 
$$
\sX_\zeta =[V^{ss}_G(\zeta)/G]
$$
is the GIT quotient stack, and 
$$
X_\zeta = V^{ss}_G(\zeta)/G  = V\sslash_\zeta G
$$
is the GIT quotient. 

\item  The triple $(V,G,\zeta)$ determines a particular presentation of $\sX_\zeta$ as a quotient stack $[U_\zeta/G]$  and an extended stacky fan  
$$
\bSi_\zeta= (N, \Sigma_\zeta, \beta,S_\zeta),
$$
 a notion  introduced by Jiang \cite{Ji08}. 
 \end{enumerate}

\subsection{Closed toric substacks and their generic stabilizers}   \label{sec:closed-substacks}
The $\tT$-divisor $\tD_j= \{x_j=0\} \subset V=\Spec\bC[x]$  restricts to a $\tT$-divisor $\tD_j \cap U_\zeta  \subset  U_\zeta$ which descends to 
a $T$-divisor $\DD_j =[ (\tD_j\cap U_\zeta )/G]$ in the toric stack $\sX_\zeta = [U_\zeta/G]$
and a $T$-divisor $D_j$ in the toric variety $X_\zeta$.  Note that $\DD_j$ and $D_j$ are empty if $j\in S_\zeta$.

Given  any $\sigma \in \Sigma_\zeta(d)$, where $0\leq d\leq n$,  we have $\sigma=\sigma_{I'}$ for some $I\in \cA_\zeta$ with $|I|={\boldsymbol{\kappa}} + n-d$. Let
$$
\sV(\sigma) =  \bigcap_{i\in I'} \DD_i   \subset \sX_\zeta,\quad
V(\sigma) = \bigcap_{i\in I'} D_i \subset X_\zeta.
$$
Then $\sV(\sigma)$ (resp. $V(\sigma)$)  is an $(n-d)$-dimensional closed toric substack (resp. subvariety) of $\sX_\zeta$
(resp. $X_\zeta$). 
The generic stabilizer of the toric stack $\cV(\sigma)$ is  the finite group 
$$
G_\sigma =\bigcap_{i\in I} \Ker(\chi^{\D_i})   \subset G
$$
where $\chi^{\D_i}: G\to \bC^*$ is defined as in \eqref{eqn:chi}.  If $\tau, \sigma\in \Sigma_\zeta$ and $\tau \subset \sigma$ then $\sV(\tau)\supset \sV(\sigma)$, so $G_\tau \subset G_\sigma$. In particular, $\sV(\{0\}) = \sX_\zeta$ and $G_{\{0\} }$ is trivial. 
If $I \in \cA_\zeta^{\min}$,  then $\sigma_{I'} \in \Sigma_\zeta(n)$ and $\fp_I := \sV(\sigma_{I'}) \simeq [\bullet/G_{\sigma_{I'}}] = B G_{\sigma_{I'}}$  is 
the unique $T$-fixed point in 
$$
\sX_I  := [V_I/G] \simeq  \Big[ \big( (\bC^*)^I \times \bC^{I'} \big)/ G \Big] \simeq [\bC^{I'}/G_{\sigma_{I'}}]  \simeq [\bC^n/G_{\sigma_{I'}}]. 
$$
Here $\bullet =\Spec\bC$ is a point, $BG_{\sigma_{I'} }$ is the classifying space
of $G_{\sigma_{I'}}$, and  $V_I$ is  defined by Equation \eqref{eqn:VI}.

\subsection{Line bundles} \label{sec:line-bundles} 
Let  
$$
U_j^T:= \cO_{\sX_\zeta}(-\DD_j) \in \Pic_T(\sX_\zeta), \quad  u_j^T := - (c_1)_T(U_j^T) \in H^2_T(X_\zeta;\bQ).
$$
Then $u_j^T$ is the $T$-equivariant Poincar\'{e} dual of $\DD_j$. Note that $u_j^T=0$ if $j\in S_\zeta$.
The $T$-equivariant Chern character of $U_j^T$ is  $\ch_T(U_j^T) = e^{-u_j^T}$. 

The $T$-equivariant line bundles $U_j^T$ generate $K_T(\sX_\zeta)$ as an algebra over $\bZ$.  Any group homomorphism $A\to T$ induces a map 
$[\sX_\zeta/A] \to [\sX_\zeta/T] =[V^{ss}_G(\zeta)/\tT]$ and ring homomorphisms
$$
K_T(\sX_\zeta) \to K_A(\sX_\zeta),\quad  \phi^*: H_T^*(X_\zeta;\bQ)\to H_A^*(X_\zeta;\bQ).
$$

Let $U_j \in \Pic(\sX_\zeta)$ (resp. $u_j\in H^2(X_\zeta;\bQ)$)  be the image of $U_j^T\in \Pic_T(\sX_\zeta)$ (resp. $u_j^T\in H^2_T(X_\zeta;\bQ)$)
under the surjective group homomorphism $\Pic_T(\sX_\zeta)\to \Pic(\sX_\zeta)$ (resp. $H^2_T(X_\zeta;\bQ) \to H^2(X_\zeta;\bQ)$) induced by
the group homomorphism $\{1\}\to T$. Then $c_1(U_j) =-u_j$ and $\ch(U_j) = e^{-u_j}$. 

Let $U_j^{\tT} \in \Pic_{\tT}(\sX_\zeta)$ (resp. $u^{\tT}_j\in H^2_{\tT}(X_\zeta;\bQ)$)  be the image of $U_j^T\in \Pic_T(\sX_\zeta)$ (resp. $u_j^T\in H^2_T(X_\zeta;\bQ)$)
under the group homomorphism $\Pic_T(\sX_\zeta)\to \Pic_{\tT}(\sX_\zeta)$ (resp. $H^2_T(X_\zeta;\bQ) \to H^2_{\tT}(X_\zeta;\bQ)$) induced by
the group homomorphism $\tT\to T=\tT/G$. Then $(c_1)_{\tT}(U_j^{\tT}) =-u_j^{\tT} $ and $\ch_{\tT}(U^{\tT}_j) = e^{-u^{\tT}_j}$. 

For $i=1,\ldots, n+{\boldsymbol{\kappa}}$, let $\chi^{\D_i}:G\to \bC^*$ be defined as in \eqref{eqn:chi}. 
For $a=1,\ldots,{\boldsymbol{\kappa}}$, let $\chi^{\xi_a^*}: G\to \bC^*$ be the character associated to $\xi_a^*\in \bL^\vee$, i.e.,  $\chi^{\xi_a^*}(s_1,\ldots,s_{\boldsymbol{\kappa}}) = s_a$. 
Let $G$ act on $U_\zeta\times \bC$ by 
$$
s\cdot (x_1,\ldots, x_{n+{\boldsymbol{\kappa}}},y) = (\chi^{\D_1}(s) x_1,\ldots, \chi^{\D_{n+{\boldsymbol{\kappa}}}} x_{\boldsymbol{\kappa}}, \chi^{\xi_a^*}(s^{-1}) y),
$$
This defines a $G$-equivariant line bundle on $U_\zeta$, or equivalently  a line bundle $P_a$ on $\sX_\zeta =[U_\zeta/G]$. 
Let $p_a = -c_1(P_a) \in H^2(X_\zeta;\bQ)$. Then  $\ch(P_a) = e^{-p_a}$. For $j=1,\ldots, n+{\boldsymbol{\kappa}}$, we have
\begin{equation}\label{eqn:u-p}
U_j =\prod_{a=1}^{\boldsymbol{\kappa}} P_a^{Q_j^a} \in \Pic(\sX_\zeta), \quad u_j = \sum_{a=1}^{\boldsymbol{\kappa}} Q_j^a p_a \in H^2(X_\zeta;\bQ). 
\end{equation} 

Let $\Lambda_j \in \Pic_{\tT}(\bullet) =\Pic(B\tT)$ be the $\tT$-equivariant line bundle over a point $\bullet$ defined by
the $\tT$-character $\tit_j$, and let $\lambda_j = -(c_1)_{\tT}(\Lambda_j) \in H^2_{\tT}(\bullet;\bZ) = H^2(B\tT;\bZ)$. 
Then
$$
K_{\tT}(\bullet) =\bZ[\Lambda_1^{\pm 1},\ldots, \Lambda_{n+{\boldsymbol{\kappa}}}^{\pm 1}],\quad
H^*(B\tT;\bZ) =\bZ[\lambda_1,\ldots, \lambda_{n+{\boldsymbol{\kappa}}}]. 
$$
For later convenience, we introduce  the following definition.
\begin{definition} \label{D-dual} 
Let $I\in \cA^{\min}_\zeta$ be a minimal anticone. Then  $\{\D_i: i\in I\}$ is a $\bQ$-basis of $\bL^\vee_\bQ$. Let
$\{ \D_i^{*, I}: i\in I\}$ be the dual $\bQ$-basis of $\bL_\bQ$,  i.e., for any $i,j\in I$,  
$$
\langle \D_i, \D_j^{*,I}\rangle =\delta_{ij}
$$
where $\langle  \;\, , \; \, \rangle :\bL_\bQ^\vee\times \bL_\bQ\to \bQ$ is the natural pairing between dual vector spaces over $\bQ$. 
\end{definition}

Given any $I\in \cA_\zeta^{\min}$, the inclusion $\iota_I: \fp_I \hookrightarrow \sX_\zeta$ of the torus fixed point $\fp_I$ induces
a ring homomorphism
$$
\iota_I^*: H^*_{\tT}(X_\zeta;\bQ) \lra H^*_{\tT}(\fp_I;\bQ) \simeq  H^*(B\tT;\bQ) =\bQ[\lambda_1,\ldots,\lambda_{n+{\boldsymbol{\kappa}}}].
$$
For $i=1,\ldots, n+{\boldsymbol{\kappa}}$, 
\begin{equation} \label{eqn:u-at-I}
\iota_I^* u_i^{\tT} = \sum_{j\in I}\langle \D_i,   \D_j^{*, I}\rangle\lambda_j  -\lambda_i \quad \forall I\in \cA_\zeta^{\min}. 
\end{equation} 
In particular, if $i\in I$ then  $\iota_I^* u_i^{\tT} =0$, which is consistent with the fact that the $T$-fixed point $\fp_I$ is not contained  in the divisor $\DD_i$.

Let $\tT\times G$ act on $U_\zeta \times \bC$ by 
$$
(\tit, s) \cdot (x_1,\ldots, x_{n+{\boldsymbol{\kappa}}},y) = ( \tit_1 \chi^{\D_1}(s) x_1,\ldots, \tit_{n+{\boldsymbol{\kappa}}} \chi^{\D_{n+{\boldsymbol{\kappa}}}} x_{\boldsymbol{\kappa}}, \chi^{\xi_a^*}(s^{-1}) y),
$$
where $\tit =(\tit_1,\ldots,\tit_{n+{\boldsymbol{\kappa}}}) \in \tT$ and $s\in G$. 
This defines a $\tT\times G$-equivariant line bundle on $U_\zeta$, or equivalently  a $\tT$-equivariant line bundle $P^{\tT}_a$ on $\sX_\zeta =[U_\zeta/G]$. 
Let $p^{\tT} _a = -c_1(P^{\tT}_a) \in H^2_{\tT}(X_\zeta;\bQ)$. Then  $\ch_{\tT}(P^{\tT}_a) = e^{-p^{\tT}_a}$. 
For $a=1\ldots,{\boldsymbol{\kappa}}$, 
\begin{equation} \label{eqn-p-at-I}
\iota_I^* p_a^{\tT} =  \sum_{j\in I} \langle \xi_a^*, \D_j^{*,I}\rangle \lambda_j \quad \forall I\in \cA_\zeta^{\min}. 
\end{equation} 
We have
\begin{equation}\label{eqn:u-p-lambda} 
U_j^{\tT} =\prod_{a=1}^{\boldsymbol{\kappa}}(P_a^{\tT})^{Q_j^a}  \cdot \Lambda_j^{-1} \in \Pic_{\tT}(\sX_\zeta), \quad 
u_j^{\tT} = \sum_{a=1}^{\boldsymbol{\kappa}} Q_j^a p_a^{\tT}-\lambda_j  \in H^2_{\tT}(X_\zeta;\bQ). 
\end{equation}
 
 Under $K_{\tT}(\sX_\zeta) \lra K(\sX_\zeta)$, $U_j^{\tT}$, $P_a^{\tT}$, and $\Lambda_j$ are mapped
 to $U_j$, $P_a$, and $1$, respectively; under $H^2_{\tT}(X_\zeta;\bQ)\to H^2(X_\zeta;\bQ)$, 
 $u^{\tT}_j$, $p_a^{\tT} $, and $\lambda_j$ are mapped to $u_j$, $p_a$, and $0$, respectively. Our definitions of  $u^{\tT}_j$, $U^{\tT}_j$, $p^{\tT}_a$, $P^{\tT}_a$, $\lambda_j$, 
$\Lambda_j$  are consistent with the convention in Givental's papers \cite{GiV, GiVI} on permutation-equivariant quantum K-theory of toric manifolds.

\subsection{The inertia stack} 

\subsubsection{The inertia stack of a general algebraic stack}  
Given a general algebraic stack $\cX$, the {\em inertia  stack}  $\cI\cX$ of $\cX$ is the fiber product
$$
\xymatrix{
\cI\cX \ar[r]\ar[d] & \cX \ar[d]^\Delta \\
\cX  \ar[r]^{\Delta} \ar[r]^\Delta & \cX\times \cX
}
$$
where $\Delta:\cX\to \cX\times \cX$ is the diagonal morphism. $\cI\cX$  is an algebraic stack, and in particular a groupoid. An object in the groupoid 
$\cI\cX$ is a pair $(x,g)$ where $x$ is an object in the groupoid $\cX$ and $k$ is an element in the automorphism 
group $\Aut_{\cX}(x) =\Hom_{\cX}(x,x)$ of $x$.  Morphisms between two objects in $\cI\cX$ are
$$
\Hom_{\cI\cX} ((x_1,g_1), (x_2,g_2)) = \{ h \in \Hom_{\cX}(x_1,x_2) : h\circ g_1 = g_2 \circ h\}. 
$$
The map $(x,g) \mapsto (x,g^{-1})$, where $(x,g)$ is an object in $\cI\cX$, defines an involution 
$$
\inv: \cI\cX \to \cI\cX. 
$$

\subsubsection{The inertia stack of a quotient stack} 
If $\cX=[U/G]$ is a quotient stack, where $U$ is a scheme and $G$ is an algebraic group, then the inertia stack
is also a quotient stack:
$$
\cI\cX = [ IU/G]
$$
where $IU :=\{ (x,g)\in U\times G \mid g \cdot x = x\}$ is a closed subscheme of $U\times G$, 
and the $G$-action on $IU$ is given by 
$$
h\cdot (x,g) =  (h\cdot x, h g h^{-1}),\quad \text{where }h\in G \text{ and } (x,g) \in IU. 
$$ 
In particular, if $G$ is abelian then the action is given by $h\cdot (x,g)= (h\cdot x, g)$. 

The involution $\inv: \cI\cX \to \cI\cX$ is induced by the $G$-equivariant
involution $IU\to IU$ given by $(x,g)\mapsto (x,g^{-1})$. 

\subsubsection{The inertial stack of the toric orbifold $\sX_\zeta$} 
Let $\sX_\zeta = [U_\zeta/G]$ be  as in  Section \ref{sec:anticones}. To describe its inertia stack $\cI\sX_\zeta$ more explicitly, we introduce some definitions. 
Given $\sigma =\sigma_I \in \Sigma_\zeta$, where $I\subset \{1,\ldots, n+{\boldsymbol{\kappa}}\}$,  define
\begin{equation}
\Boxs := \Big\{  v\in N  : v=\sum_{i\in I}  c_i v_i, \  0\leq c_i <1 \Big\}
\end{equation} 
and
\begin{equation}
\mathrm{Box}'(\sigma) := \Big\{  v\in N  : v=\sum_{i\in I} c_i v_i, \  0 < c_i <1 \Big\}.
\end{equation} 
Define 
\begin{equation}
\mathrm{Box}(\bSi_\zeta) :=  \bigcup_{\sigma\in \Sigma_\zeta}\Boxs =\bigcup_{\sigma\in \Sigma_\zeta(n)} \Boxs.  
\end{equation}
which is a finite set.  For any $v\in \mathrm{Box}(\bSi_\zeta)$ there exists a unique $\sigma \in \Sigma_\zeta$ such that $v\in \mathrm{Box}'(\sigma)$. Therefore,
\begin{equation} \label{eqn:union}
\mathrm{Box}(\bSi_\zeta) = \bigcup_{\si\in \Sigma_\zeta} \mathrm{Box}'(\sigma),
\end{equation} 
where the right hand side of \eqref{eqn:union} is a disjoint union.  Given any $\sigma= \sigma_I \in \Sigma_\zeta$, where $I'\in \cA_\zeta$, define
\begin{equation}\label{eqn:age}
\age(v) =\sum_{i\in I} c_i \in \bQ \quad\text{ if } v =\sum_{i\in I} c_i v_i  \in \mathrm{Box}'(\sigma).
\end{equation} 

Suppose that $\sigma  =\sigma_I \in \Sigma_\zeta$ where $I'\in \cA_\zeta$. There is a bijection
$\Boxs \lra G_\sigma$ given by 
\begin{equation}\label{eqn:kv}
v= \sum_{i\in I} c_i v_i \mapsto   g(v) =(a_1, \cdots, a_{n+{\boldsymbol{\kappa}}})\in G_\sigma \subset G\subset \tT\simeq (\bC^*)^{n+{\boldsymbol{\kappa}}} 
\; \text{ where } \;
a_i = \begin{cases}
e^{2\pi\sqrt{-1} c_i}, & i\in I,\\
1, & i\in I'.
\end{cases} 
\end{equation}
The map $v\mapsto g(v)$ defines a bijection 
$$
\mathrm{Box}(\bSi_\zeta)  =\bigcup_{\sigma\in \Sigma_\zeta} \Boxs 
\lra  \bigcup_{\sigma\in \Sigma_\zeta} G_\sigma =\{ g \in G: (U_\zeta)^g \text{ is non-empty.}\},
$$ 
where $(U_\zeta)^g =\{ x \in U_\zeta: g\cdot x = x\}$. 
We have 
\begin{equation} \label{eqn:IU} 
IU_\zeta = \{(x,g) \in U_\zeta \times G: g\cdot x =x \} = \bigcup_{v\in \mathrm{Box}(\bSi_\zeta) }  (U_\zeta)^{g(v)} \times \{ g(v)\}. 
\end{equation} 
The above union is a disjoint union of connected components. 
$$
I\sX_\zeta  =  [IU_\zeta/G] = \bigcup_{v\in \mathrm{Box}(\bSi_\zeta)}  \sX_{\zeta,v} \quad \text{where } 
 \sX_{\zeta, v}  \simeq [ (U_\zeta)^{g(v)}  /G]. 
$$
In particular, $g(0) =1$ is the identity element of $G$ and $\sX_{\zeta,0}\simeq  \sX_\zeta$. 
There is an involution $\inv: \mathrm{Box}(\bSi_\zeta) \to \mathrm{Box}(\bSi_\zeta)$ characterized by 
$g(\inv(v)) = g(v)^{-1} \in G$.  The involution $\inv: \cI\sX_\zeta\lra \cI\sX_\zeta$ maps  $\sX_{\zeta,v}$ isomorphically to $\sX_{\zeta,\inv(v)}$. Observe that
$$
\age(v)  +\age(\inv(v))  +\dim \sX_{\zeta,v} = n =\dim \sX_\zeta. 
$$

Let 
\begin{equation}\label{eqn:zeta}
\eta = (e^{\pi\sqrt{-1} q_1},\ldots, e^{\pi\sqrt{-1} q_{n+{\boldsymbol{\kappa}}}}) \in \bC_R^*.
\end{equation} 
Then  $\eta^2 = J = (e^{2\pi\sqrt{-1} q_1},\ldots, e^{2\pi\sqrt{-1} q_{n+{\boldsymbol{\kappa}}}})$ and $W(\eta \cdot x) = -W(x)$. 
Let $\inv_R: \cI\sX_\zeta\to \cI\sX_\zeta$ be the map induced by the map $IU_\zeta\to IU_\zeta$ given by 
$(x,g)\mapsto (\eta\cdot x, g^{-1})$. Then 
$\inv_R$ maps $\sX_{\zeta,v}$ isomorphically to $\sX_{\zeta, \inv(v)}$. Note that in general
$\inv_R\circ \inv_R$ is not the identity map, i.e.,  $\inv_R$ is not an involution.

\subsection{A-model GLSM state spaces} \label{sec:state} 
Let $M$ be a large positive number such that the real part of any critical values of $W$ is less than $M$, and define
$\bw_\zeta^\infty := (\mathfrak{Re} W)^{-1}(M, \infty) \subset \sX_\zeta$.  As a graded vector space over $\bQ$, the rational
GLSM state space of $(V,G, \bC^*_R, W, \zeta)$ is
\begin{equation}
\cH_{\bw,\bQ}  = \bigoplus_{v\in \mathrm{Box}(\bSi_\zeta) } H^*(\sX_{\zeta,v}, \bw_{\zeta,v}^\infty;\bQ)[2 \left(\age(v)-\hq\right)].
\end{equation}
where $\bw_{\zeta,v} = \bw_\zeta \big|_{\sX_{\zeta,v}}$.  In particular,
$\cH_{\bw,\bQ} \cong H^*(\cI\sX_\zeta;\bQ)$ as a vector space over $\bQ$. 
Note that
$\inv_R$ maps $(\sX_{\zeta,v}, \bw_{\zeta,v})$ diffeomorphically to $(\sX_{\zeta,\inv(v)}, -\bw_{\zeta,\inv(v)})$  and induces an isomorphism
\begin{equation}\label{eqn:inv-w}
\inv_R^*: H^*(\sX_{\zeta,\inv(v)}, -\bw_{\zeta,\inv(v)}^\infty;\bQ)\stackrel{\cong}{\lra} H^*(\sX_{\zeta,v}, \bw_{\zeta,v}^\infty;\bQ). 
\end{equation}
 By \cite[Corollary 4.3]{FK}, $\mathrm{Crit}(\bw_{\zeta,v}) \subset \sZ_\zeta= \mathrm{Crit}(\bw_\zeta)$. If $\sZ_\zeta$ is proper then
there is a nondegenerate pairing 
\begin{equation} \label{eqn:pairing-w}
H^*(\sX_{\zeta,v},\bw_{\zeta,v}^\infty;\bQ) \times H^*(\sX_{\zeta,v}, -\bw_{\zeta,v}^\infty;\bQ)\to  \bQ. 
\end{equation}
for all $v\in \mathrm{Box}(\bSi_\zeta)$. Combining \eqref{eqn:inv-w} and \eqref{eqn:pairing-w}, we obtain a nondegenerate pairing
$$
H^*(\sX_{\zeta,v},\bw_{\zeta,v}^\infty ;\bQ)\times H^*(\sX_{\zeta,\inv(v)},  \bw_{\zeta,\inv(v)}^\infty;\bQ) \to \bQ
$$
for all $v\in \mathrm{Box}(\bSi_\zeta)$, thus a non-degenerate pairing $( \  , \   )_{\bw}: \cH_{\bw,\bQ}\times \cH_{\bw,\bQ} \to \bQ$. 

As a graded vector space over $\bC$, the GLSM state space of $(V,G,\bC^*_R, W,\zeta)$ is
$$
\cH_{\bw} = \bigoplus_{v\in \mathrm{Box}(\bSi_\zeta)} \bH^*\left(\sX_{\zeta,v}, (\Omega^\bullet_{\sX_{\zeta,v}}, d\bw_{\zeta,v})\right)[2(\age(v)-\hq)]
$$
where 
$$
\bH^*\left(\sX_{\zeta,v}, (\Omega^\bullet_{\sX_{\zeta,v}}, d\bw_{\zeta,v})\right) \simeq
H^*(\sX_{\zeta,v}, \bw_{\zeta,v}^\infty;\bC).
$$
Therefore, $\cH_{\bw} \simeq \cH_{\bw,\bQ}\otimes_{\bQ}\bC$ as a graded vector space over $\bC$.

We say $v\in \mathrm{Box}(\bSi_\zeta)$ is {\em narrow} if $\sX_{\zeta,v}$ is compact. If $v$ is narrow then $\bw_{\zeta,v}$ is constant  and
$\bw_{\zeta,v}^\infty$ is empty, so 
$$
H^*(\sX_{\zeta,v},\bw^\infty_{\zeta,v};\bC) = H^*(\sX_{\zeta,v};\bC).
$$
In this case, we call the above vector space a narrow sector.  We say $v$ is broad if $v$ is not narrow, and call
$$
H^*(\sX_{\zeta,v},\bw^\infty_{\zeta,v};\bC)
$$
a broad sector.  

We also consider a closely related GLSM $(V,G,\bC_R^*, 0, \zeta)$ obtained by replacing the superpotental $W$ by zero.  
As a graded vector space over $\bC$, the  GLSM state space of $(V,G, \bC^*_R, 0, \zeta)$ is
\begin{equation}
\cH  = \bigoplus_{v\in \mathrm{Box}(\bSi_\zeta) } H^*(\sX_{\zeta,v};\bC)[2 \left(\age(v)- \hq \right) ].
\end{equation}
Note that $\inv$ and $\inv_R$ are homotopic as maps from $\sX_{\zeta,v}$ to $\sX_{\zeta, \inv(v)}$, so 
$\inv_R^*=\inv^*: H^*(\sX_{\zeta,\inv(v)};\bC)\to H^*(\sX_{\zeta,v};\bC)$.

The action of $\tT$ on $V$ commutes with the action of its subgroup $G$ and $\bC^*_R$, and preserves the zero superpotential $0$
and the semistable locus $V_G^{ss}(\zeta)$. 
We define the  $\tT$-equivariant state space $\cH_{\tT}$ of $(V,G,\bC^*_R,0,\zeta)$ as follows. As a graded vector space over $\bC$, 
\begin{equation}\label{eqn:sHT} 
\cH_{\tT} =\bigoplus_{v\in \mathrm{Box}(\bSi_\zeta)} H^*_{\tT}(\sX_{\zeta,v};\bC)[2 \left(\age(v)-\hq \right) ]
\end{equation} 
where each $H^*_{\tT}(\sX_{\zeta,v};\bC)$ is a module over $H^*_{\tT}(\bullet;\bC) = H^*(B\tT;\bC) =\bC[\lambda_1,\ldots, \lambda_{n+\boldsymbol{\kappa}}]$.
Let $\bC(\lambda):=\bC(\lambda_1,\ldots, \lambda)$ be the fractional field of $\bC[\lambda]:=\bC[\lambda_1,\ldots, \lambda_{n+\boldsymbol{\kappa}}]$.
There is a non-degenerate pairing
$$
\cH_{\tT}\otimes_{\bC[\lambda]} \bC(\lambda)  \times \cH_{\tT}\otimes_{\bC[\lambda]} \bC(\lambda)\to \bC(\lambda). 
$$

\section{Categories of B-branes and K-theories}

Given an abelian GLSM, we consider several versions
of the category of B-branes and the A-model state space.
\medskip 

\begin{tabular}{|c|c|c|c|} 
\hline 
dg category  & category of B-branes & K-theory & A-model state space \\ [2pt]
  \hline
$\MF(\sX_\zeta,\bw_\zeta)$ & $D\MF(\sX_\zeta,\bw_\zeta) \simeq D_\Sg(\sX_{\zeta,0})$  & $K(D\MF(\sX_\zeta,\bw_\zeta))$ & $\cH_{\bw}   =  \displaystyle{ \bigoplus_{v\in \mathrm{Box}(\bSi_\zeta)}H^*(\sX_{\zeta,v},\bw_{\zeta,v}^{\infty}) }$ 
  \\[3pt]
  \hline
$\Perf_{\tT}(\sX_\zeta)$ & $D^b_{\tT}(\sX_\zeta)$ & $K_{\tT}(\sX_\zeta)$  &  
$\cH_{\tT}= \displaystyle{ \bigoplus_{v\in \mathrm{Box}(\bSi_\zeta)} H^*_{\tT}(\sX_{\zeta,v}) }$
  \\ [3pt]
\hline 
$\Perf(\sX_\zeta)$ & $D^b(\sX_\zeta)$ & $K(\sX_\zeta)$  &  
$\cH= \displaystyle{ \bigoplus_{v\in \mathrm{Box}(\bSi_\zeta)} H^*(\sX_{\zeta,v}) }$  \\ [3pt]
\hline
& $D^b_c(\sX_\zeta)$ & $K_c(\sX_\zeta)$ &  $\cH_c= \displaystyle{ \bigoplus_{v\in \mathrm{Box}(\bSi_\zeta)} H_c^*(\sX_{\zeta,v}) }$ \\ [3pt]%
 \hline
\end{tabular} 

\smallskip

In each version, the  Chern character sends the K-theory class of a B-brane to an element in the A-model state space. 
$$
\begin{array}{llll}
K_{\bw} = K(D\MF(\sX_\zeta,\bw_\zeta)) & \stackrel{\ch_{\bw}}{\lra} \quad &  \displaystyle{ \cH_{\bw} := \bigoplus_{v\in \mathrm{Box}(\bSi_\zeta)} H_{\bw,v} }, & H_{\bw,v} = H^*(\sX_{\zeta,v}, w_{\zeta,v};\bC). \\
K_{\tT} = K_{\tT}(\sX_\zeta)     &\stackrel{\ch_{\tT}}{\lra}  \quad & \cH_{\tT} = \displaystyle{ \bigoplus_{v\in \mathrm{Box}(\bSi_\zeta)} H_{\tT,v} }, &  H_{\tT,v}= H^*_{\tT}(\sX_{\zeta,v};\bC). \\
K= K(\sX_\zeta)                & \stackrel{\ch}{\lra}   \quad & \cH =  \displaystyle{ \bigoplus_{v\in \mathrm{Box}(\bSi_\zeta)} H_v} , & H_v:= H^*(\sX_{\zeta,v};\bC).   \\
 K_c = K_c(\sX_\zeta)         & \stackrel{\ch_c}{\lra} \quad  & \cH_c =  \displaystyle{ \bigoplus_{v\in \mathrm{Box}(\bSi_\zeta)} H_{c,v} }, & H_{c,v} = H^*_c(\sX_{\zeta,v};\bC).  
 \end{array}
$$

$K_{\tT}$ is a commutative ring with unity and an algebra over $K_{\tT}(\bullet) = K(B\tT) = \bZ[\Lambda_1^{\pm},\ldots, \Lambda_{n+{\boldsymbol{\kappa}}}^{\pm}]$, and
$K$ is a commutative ring with unity and an algebra over $K(\bullet)=\bZ$.  There is a surjective ring homomorphism $K_{\tT}\to K$.  $K_c$ and $K_{\bw}$ are 
modules over the ring $K$. There is a map $K_c\to K$ which is a morphism of $K$-modules; the image $K_{\ct}$ is an ideal in the ring $K$.
Taking Euler characteristic defines non-degenerate pairings:
$$
K_{\tT} \times K_{\tT}\to \bQ(\Lambda_1,\ldots,\Lambda_{n+{\boldsymbol{\kappa}}}),\quad K\times K_c \to \bZ, \quad \quad K_{\bw}\times K_{\bw}\to \bZ. 
$$

Fix $v\in \mathrm{Box}(\bSi_\zeta)$. $H_{\tT,v}$ is a commutative ring with unity and an algebra over $H^*_{\tT}(\bullet) = H^*(B\tT) 
= \bC[\lambda_1,\ldots, \lambda_{n+{\boldsymbol{\kappa}}}]$, and
$H_v$ is a commutative ring with unity and an algebra over $H^*(\bullet)=\bC$.  There is a surjective ring homomorphism $H_{\tT,v}\to H_v$.  $H_{c,v}$ and $H_{w,v}$ are 
modules over the ring $H_v$. There is a map $H_{c,v}  \to H_v$ which is a morphism of $H_v$-modules; the image $H_{\ct,v}$ is an ideal in the ring $H_v$.  We have. non-degenerate pairings:
$$
H_{\tT,v} \times H_{\tT, \inv(v)} \to \bC(\lambda_1,\ldots,\lambda_{n+{\boldsymbol{\kappa}}}),\quad
H_v \times H_{c,\inv(v)} \to \bC, \quad H_{\bw,v}\times H_{\bw,\inv(v)} \to \bC.  
$$

\section{The Higgs Branch} \label{sec:Higgs} 

\subsection{A mathematical theory of GLSM: an overview}  
Let  $(V,G, \bC_R^*, W, \zeta)$ be the input date of a GLSM. 
The first four components $(V,G,\bC_R^*,W)$ give rise to the following diagram:
$$
\xymatrix{
V\ar[dr]^W \ar[d] & \\
[V/G] \ar[r]^{\bw} \ar[d] \ar@{}[dr]|-{\square}  & \bC \ar[d] \\
[V/\Gamma] \ar[r]^{\hat{\bw}} \ar[d]  & [\bC/\bC^*_\omega] \ar[d]  \\
B\Gamma=[\bullet/\Gamma] \ar[r]^{B\chi} & B\bC^*_\omega =[\bullet/\bC^*_\omega]
}
$$
where the middle square is Cartesian, the upper triangle and the lower square are commutative, and the bottom right arrow $B\chi: B\Gamma\to B\bC^*_\omega$ is induced by the group homomorphism $\chi: \Gamma \to \bC^*_\omega$. A Landau-Ginzburg (LG) quasimap 
to $(V,G, \bC^*_R,W,\zeta)$  is a birational map  from an orbicurve $\cC$ to  $[V^{ss}_G(\zeta)/\Gamma]$ which extends to a representable morphism $f: \cC\to [V/\Gamma]$  
of smooth Artin stacks and satisfies certain stability conditions   such that the following diagram commutes: 
$$
\xymatrix{
& [V/\Gamma] \ar[d]^\pi\\
\cC \ar[ur]^f \ar[r]^P \ar[dr]_{\omega^{\log}_\cC} & B\Gamma  \ar[d]^{B\chi} \\
& B\bC^*_\omega
}
$$

Recall that $B\Gamma$ is the classifying space of principal $\Gamma$-bundles, and $[V/\Gamma]$ is the classifying
space of a principal $\Gamma$-bundle $P\to \cC$ together with a section $u: \cC \to P\times_\Gamma V$, where 
$P\times_\Gamma V\to \cC$ is the rank $n+{\boldsymbol{\kappa}}$ vector bundle associated to the representation $\Gamma \to \GL(V)$. 
A section $u: \cC\to P\times_\Gamma V$ is equivalent to a $\Gamma$-equivariant map $\tilde{f}: P\to V$. 
More explicitly, we have the following cartesian diagram
$$
\xymatrix{
P \ar[r] \ar[d] & \bullet \ar[d]\\
\cC = P/\Gamma  \ar[r]^{\pi\circ f} & B\Gamma =[\bullet/\Gamma] 
}
$$
Let $\pr_1: P\times V\lra P$ and $\pr_2: P\times V\to V$ be the projection to the first and second factors, respectively. 
We have a commutative diagram
$$
\xymatrix{
P\times V \ar[d]^{\pr_1} \ar[r]^{\pr_2} & V\ar[d] \\
P\ar[r]  \ar@/^/[u]^{(id_P,\tilde{f)}} \ar[ur]^{ \quad\quad \tilde{f}} & \bullet
}
$$
where all the arrows are $\Gamma$-equivariant. Taking the quotient of the above diagram by the $\Gamma$-action, we obtain the
following commutative diagram
$$ 
\xymatrix{
P\times_\Gamma V \ar[d] \ar[r] &[V/\Gamma]\ar[d]^\pi \\
\cC \ar[r]^{\pi\circ f}  \ar@/^/[u]^{u} \ar[ur]^f  & B\Gamma = [\bullet/\Gamma] 
}
$$

\subsection{Twisted curves and their moduli} \label{sec:twisted}
We follow the presentation of  \cite{AbVi, AGV} on twisted curves. 
A genus-$g$, $\ell$-pointed twisted prestable curve is a connected proper one-dimensional DM stack $\cC$ 
together with $\ell$ disjoint zero-dimensional integral closed substacks $\fz_1,\ldots,\fz_\ell \subset \cC$, such that
\begin{enumerate}
\item[(i)]  $\cC$ is \'{e}tale locally a nodal curve; 
\item[(ii)] formally locally near a node, $\cC$ is isomorphic
to the quotient stack
$$
[\Spec(\bC[x,y]/(xy))/\mu_r],\quad\text{where}\ \eta\cdot (x,y)= (\eta x, \eta^{-1}y),\, \eta\in \mu_r;
$$ 
\item[(iii)] each marking $\fz_i\subset \cC$ is 
contained in the smooth locus of $\cC$;
\item[(iv)] $\cC$ is a scheme away from the markings and the singular points of $\cC$;
the coarse moduli space $C$ of $\cC$ is a nodal curve of arithmetic genus $g$.
\end{enumerate}

Let $\pi:\cC\to C$ be the coarse moduli morphism; let $z_i=\pi(\fz_i)$.
The resulting $(C,z_1,\ldots, z_\ell)$ is a genus-$g$, $\ell$-pointed  prestable curve.
We say $(\cC,\fz_1,\ldots, \fz_\ell)$ is stable if $(C,z_1,\ldots,z_\ell)$ is stable. 
The moduli $\fM^\tw_{g,\ell}$ of genus-$g$, $\ell$-pointed twisted prestable curves is a smooth Artin stack of 
dimension $3g-3+\ell$.  For $i=1,\ldots,\ell$, let $\bL_i$ be the line bundle on $\fM_{g,\ell}^{\tw}$ whose fiber
over $(\cC,\fz_1,\ldots, \fz_\ell)$ is the $T^*_{z_j}C$, the cotangent line to the coarse curve $C$ at $z_i$.

The space of  infinitesimal automorphisms of $(\cC,\fz_1,\ldots,\fz_\ell)$ is 
$$
\Ext^0_{\cO_{\cC}}(\Omega_{\cC}(\fz_1+\cdots +\fz_\ell), \cO_{\cC}),
$$
while the space of infinitesimal deformations of $(\cC,\fz_1,\ldots,\fz_\ell)$ is
$$
\Ext^1_{\cO_{\cC}}(\Omega_{\cC}(\fz_1+\cdots+ \fz_\ell), \cO_{\cC}). 
$$
$(\cC,\fz_1,\ldots \fz_\ell)$ is stable if and only if $\Ext^0_{\cO_{\cC}}(\Omega_{\cC}(\fz_1 + \cdots + \fz_\ell), \cO_{\cC})=0$.

\subsection{Line bundles over a twisted curve}
Let $(\cC,\fz_1,\ldots, \fz_\ell)$ be a genus-$g$, $\ell$-pointed twisted prestable curve, where $\fz_i= B \mu_{r_i}$. A line bundle $\cL$ on $\cC$ defines a  morphism
$\cC \longrightarrow B \bC^*$ such that we have the following Cartesian diagram
$$
\xymatrix{
\cL\ar[r]\ar[d]\ar@{}[dr]|-{\square} &  [\bC/\bC^*]  \ar[d]\\
\cC \ar[r]  & [\bullet/\bC]=B \bC^*
} 
$$
where  $[\bC/\bC^*] \to B\bC^*$ is the universal line bundle over the classifying space $B\bC^*$ of principal $\bC^*$-bundles. 

Given any line bundle $\cL$ over $\cC$,  there exists a positive integer $m$ such that $\cL^{\otimes m} = \pi^*M$ for some line bundle $M$ over the coarse moduli space $C$. Define
$$
\deg \cL = \frac{1}{m}\deg M \in \bQ
$$
\begin{itemize}
\item If $\fz_i$ is a scheme point, we define $\age_{\fz_i}(\cL)=0$.  
\item  If $\fz_i= B\mu_{r_i}$, where $r_i>1$, then the restriction $\cL_{\fz_i}$ of $\cL$ to $\fz_i$ is an element in $\Pic(B\mu_{r_i})  =  \Hom(\mu_{r_i},\bC^*)\cong \bZ/r_i \bZ$.  There is a unique $a_i \in \{0,1,\ldots, r_i-1 \}$ such that
$$
\cL_{\fz_i}\cong (T_{\fz_i}\cC)^{\otimes a_i}.
$$
Define $\age_{\fz_i}(\cL) = \displaystyle{ \frac{a_i}{r_i} }\in (0,1)\cap \bQ$. 
\end{itemize}

There is a unique line bundle $L$ on the coarse moduli $C$ such that
$$
\cL \simeq \pi^*L \otimes \cO_{\cC}(\sum_{i=1}^\ell a_i \fz_i). 
$$
where
$$
\deg\Big(\cO_{\cC}\big(\sum_{i=1}^\ell  a_i \fz_i\big) \Big) =\sum_{i=1}^\ell \age_{\fz_i}(\cL), \quad \deg(\pi^*L) = \deg L \in \bZ.
$$
So
$$
\deg \cL -\sum_{i=1}^\ell  \age_{\fz_i}(\cL) \in \bZ. 
$$
For $i=0,1$,  $h^i(\cC,\cL) = h^i(C,L)$, so
\begin{equation}\label{eqn:KRR}
\chi (\cC,\cL) = \chi (C,L) = \deg\cL +1-g -\sum_{j=1}^\ell  \age_{\fz_j}(\cL)
\end{equation} 
which is a special case of  Kawasaki's orbifold version of Riemann-Roch theorem \cite{Ka79}.

The space of infinitesimal automorphisms of $\cL$ on a fixed twisted prestable curve $\cC$ is
$$
\Ext^0_{\cO_\cC}(\cL, \cL) \simeq H^0(C,\cO_C).
$$
The space of infinitesimal deformations of $\cL$ on a fixed twisted prestable curved $\cC$ is
$$
\Ext^1_{\cO_\cC}(\cL, \cL) \simeq H^1(C,\cO_C).
$$

\subsection{Universal moduli of principal  $\Gamma$-bundles} 

Let $\fM_{g,\ell}(B \Gamma)$ denote the moduli of pairs $( (\cC,\fz_1,\ldots,\fz_\ell), P)$, where
\begin{itemize}
\item  $(\cC, \fz_1,\ldots, \fz_\ell)$ is a genus-$g$, $\ell$-pointed twisted prestable curve;
\item  $P$ is a principal $\Gamma$-bundle over $\cC$ which corresponds
to a representable morphism $\cC\lra B\Gamma$.
\end{itemize}
Then $\fM_{g,\ell}(B\Gamma)$ is a smooth Artin stack.  Note that $\fM_{g,\ell}(B\Gamma)$ can be identified with the Hom-stack  
$$
\Hom_{\fM} (\fC_{\fM}, \cB\Gamma \times \fM)
$$
where $\fM=\fM_{g,\ell}^\tw$.  Forgetting the  principal bundle $P$ defines a (non-representable) morphism of smooth Artin stacks
$$
\pi_{\fD/\fM}: \fD := \fM_{g,\ell}(\cB \Gamma)\lra  \fM
$$
which is smooth of relative dimension $\dim \Gamma (g-1) = ({\boldsymbol{\kappa}}+1)(g-1)$. So $\fM_{g,\ell}(\cB \Gamma)$  is a smooth Artin stack of dimension
$$
3g-3 + \ell + ({\boldsymbol{\kappa}}+1)(g-1) = (4+{\boldsymbol{\kappa}})(g-1) + \ell = (\dim \cB \Gamma -3)(1-g) + \ell
$$
where $\dim  B \Gamma = -\dim \Gamma = -{\boldsymbol{\kappa}} -1$.   Let $\pi_{\fD}: \fC_{\fD} \to \fD := \fM_{g,\ell}(\cB\Gamma)$ be the universal curve, let
$f_{\fD}: \fC_{\fD}  \to B\Gamma$ be the universal map, and let $\fP_{\fD}\to  \fC_{\fD}$ be the universal principal $\Gamma$-bundle.  We have the following
cartesian diagrams:
$$
\xymatrix{
\fC_{\fD} \ar[r] \ar[d]_{\pi_{\fD}}  \ar@{}[dr]|-{\square}& \fC_{\fM} \ar[d]^{\pi_{\fM}} \\
\fD \ar[r]^{\pi_{\fD/\fM} } & \fM
}
\hspace{1cm} 
\xymatrix{ 
\fP_{\fD} \ar[r] \ar[d]   \ar@{}[dr]|-{\square}& \bullet  \ar[d]\\
\fC_{\fD} \ar[r]^{f_{\fD}\quad}  & B\Gamma
}
$$ 

Let $\tbL:= \Hom(\bC^*, \Gamma) \cong \bZ^{{\boldsymbol{\kappa}}+1}$ be the cocharacter lattice of $\Gamma$. Its dual lattice
$\tbL^\vee =\Hom(\Gamma,\bC^*)$ is the character lattice of $\Gamma$.  A principal $\Gamma$-bundle $P\lra \cC$ determines a map 
$\cC = P/\Gamma  \to  B\Gamma  = [\bullet/\Gamma]$
whose degree is an element 
$$
\beta_\Gamma \in H_2(B\Gamma;\bQ) = \tbL_\bQ
$$
characterized by the following property. For any  $\Gamma$-character $\lambda \in \Hom(\Gamma,\bC^*) =\tbL^\vee = H^2(B\Gamma ;\bZ)$, 
let $P\times_\lambda \bC \to \bC$ be the line bundle associated to the representation $\lambda:\Gamma\to \bC^*=\GL(1)$. Then
$$
\int_{\beta_\Gamma} \lambda=\deg(P\times_{\lambda}\bC) = \int_{[\cC]} c_1(P\times_\lambda \bC)   \in \bQ, 
$$
where
\begin{itemize}
\item $\displaystyle{ \int_{\beta_\Gamma} \lambda }$ denotes the natural pairing between $\beta_\Gamma \in \tbL_\bQ = H_2(B\Gamma;\bQ)$ and $\lambda \in \tbL^\vee \subset \tbL^\vee_\bQ = H^2(B\Gamma;\bQ)$, and 
\item $\displaystyle{ \int_{[\cC]} c_1(P\times_\lambda \bC)}$ denotes the natural pairing between $[\cC] \in H_2(\cC;\bQ)$ and
$c_1(P\times_\lambda\bC) \in H^2(\cC;\bQ)$.
\end{itemize} 
In other words, $\beta_\Gamma \in H_2(B\Gamma;\bQ)$  is the image of $[\cC]\in H_2(\cC;\bQ)$ under
$P_*: H_2(\cC;\bQ) \lra H_2( B\Gamma;\bQ)$.

The monodromy of $P$ at $\fz_j = \cB \mu_{r_j}$ is an element $\gamma_j \in \Gamma$ of order $r_j$. The subset
$$
\tbL_\bQ/\tbL \cong (\bQ/\bZ)^{{\boldsymbol{\kappa}}+1} \subset \tbL_\bC/\tbL  \cong  (\bC/\bZ)^{{\boldsymbol{\kappa}}+1} = (\bC^*)^{{\boldsymbol{\kappa}}+1}
$$
can be identified with 
$$
\{ \gamma \in \Gamma :  \ord(\gamma) \textup{ is finite.} \}.
$$
The monodromies $(\gamma_1,\ldots,\gamma_\ell) \in (\tbL_{\bQ}/\tbL)^\ell$ and the degree $\beta_\Gamma \in \tbL_\bQ$ satisfy the following compatibility condition:
$$
\sum_{j=1}^\ell  \gamma_j = \beta_\Gamma + \tbL \in \tbL_{\bQ}/\tbL. 
$$
Let $\fM_{g, \vec{\gamma}}(B\Gamma,\beta_\Gamma) \subset \fM_{g,\ell}(B\Gamma)$ be the open and closed substack
of pairs $((\cC,\fz_1,\ldots,\fz_\ell),P)$ with degree $\beta_\Gamma \in \tbL_\bQ$ and monodromies
$\vec{\gamma} =(\gamma_1,\ldots,\gamma_\ell) \in (\tbL_\bQ/\tbL)^\ell$. Then
$$
\fM_{g,\ell}(B\Gamma) =\bigcup_{ \substack{ \vec{\gamma}= (\gamma_1,\ldots,\gamma_\ell) \in  (\tbL_\bQ/\tbL)^\ell  \\ \beta_\Gamma \in \tbL_\bQ ,\;
\sum_{i=1}^\ell \gamma_i =\beta_\Gamma + \tbL} }\fM_{g,\vec{\gamma}}(B\Gamma,\beta_\Gamma). 
$$

\subsection{Universal moduli of $\Gamma$-structures} \label{sec:Gamma-structures} 
Polishchuk-Vaintrob introduced $\Gamma$-structures \cite{PV16} which is an alternative formulation
for  $W$-structures in \cite{FJR11}. 

Given an object  $(\cC,\fz_1,\ldots,\fz_\ell)$ in $\fM_{g,\ell}(\bullet)$, a $\Gamma$-structure on $(\cC,\fz_1,\ldots,\fz_\ell)$ is a pair $(P,\rho)$ where $((\cC,\fz_1,\ldots, \fz_\ell), P)$ is an object in $\fM_{g,\ell}(B\Gamma)$ and
$$
\rho: P\times_{\chi}\bC \stackrel{\cong}{\lra} \omega_\cC^{\log}
$$
is an isomorphism of line bundles on $\cC$. 

Let $\fB_{g,\ell}$ be the moduli space
of triples $((\cC,\fz_1,\ldots, \fz_\ell), P, \rho)$, where   $(\cC,\fz_1,\ldots,\fz_\ell)$ is an object
in $\fM^\tw_{g,\ell}(\bullet)$ and $(P,\rho)$ is a $\Gamma$-structure on $(\cC,\fz_1,\ldots,\fz_\ell)$. 
We have a commutative diagram
$$
\xymatrix{
\fB=\fB_{g,\ell} \ar[r]^{\pi_{\fB/\fD}\quad } \ar[dr]_{\pi_{\fB/\fM}} & \fD := \fM_{g,\ell}(B\Gamma) \ar[d]^{\pi_{\fD/\fM} } \\
&  \fM=\fM_{g,\ell}^{\tw}
}
$$
where $\pi_{\fB/\fD} :\fB\lra \fD$ is given by forgetting $\rho$.  The map $\pi_{\fB/\fM}:\fB \to \fM$ is
smooth of relative dimension $\dim G (g-1) ={\boldsymbol{\kappa}}(g-1)$, so $\fB_{g,\ell}$ is a smooth Artin stack of dimension
$$
3g-3+\ell + {\boldsymbol{\kappa}}(g-1)= (3+{\boldsymbol{\kappa}})(g-1) +\ell = (\dim BG-3)(1-g)+\ell. 
$$

The map  
$$
(\rho_V,\rho_R) : G \times \bC^*_R \cong  (\bC^*)^{{\boldsymbol{\kappa}}+1}  \lra \Gamma \cong (\bC^*)^{{\boldsymbol{\kappa}}+1} \quad (h, t)\mapsto  ht
$$
is a surjective group homomorphism and a covering map of  degree $r = \ord(J)$.  It induces
$$
(\rho_V, \rho_R)_*: H_2(BG;\bZ)\times H_2(B\bC^*_R;\bZ) = \bL \times \left(\bZ[\bP^1]\right) \lra H_2(B\Gamma;\bZ) =\tbL 
$$
which  is an inclusion of lattices of finite index $r$. (We recall that $B\bC^*_R = \bP^\infty$, and we let $[\bP^1] \in H_2(\bP^\infty;\bZ)$ be the class of $\bP^1\subset \bP^\infty$.) Therefore, we obtain an isomorphism
$$
(\rho_V, \rho_R)_*: H_2(BG;\bQ)\times H_2(B\bC^*_R;\bQ) = \bL_\bQ \times \left( \bQ[\bP^1]\right) \stackrel{\cong}{\lra} H_2(B\Gamma;\bQ) =\tbL_\bQ.
$$
Let $(\beta_G, \beta_R) = (\rho_V,\rho_R)_*^{-1}(\beta_\Gamma)$.  Then $\beta_R = \displaystyle{ \frac{2g-2+\ell}{r} [\bP^1] } \in \bQ[\bP^1]= H_2(B\bC_R^*;\bQ)
= H_2(\bP^\infty;\bQ)$.

Let $\fB_{g,\ell}(\beta_G)  \subset \fB_{g,\ell}$ be the open and closed substack
of triples $((\cC,\fz_1,\ldots,\fz_\ell),P,\rho)$ with degree 
$$
\beta_\Gamma=\left(\beta_G, \frac{2g-2+\ell}{r}[\bP^1] \right) \in \bL_\bQ \times \left(\bQ[\bP^1]\right)  \cong H_2(B\Gamma;\bQ).
$$
Then 
$$
\fB_{g,\ell} =\bigcup_{\beta_G\in \bL_\bQ} \fB_{g,\ell}(\beta_G). 
$$
For each $\beta_G$, let $\fB_{g,\vec{\gamma}}(\beta_G) \subset \fB_{g,\ell}(\beta_G)$ be the open and closed substack
of triples $((\cC,\fz_1,\ldots, \fz_\ell),P,\rho)$ with monodromies  $\vec{\gamma} =(\gamma_1,\ldots,\gamma_\ell) \in (\bL_\bQ/\bL)^\ell$. Then
$$
\fB_{g,\ell}(\beta_G) =\bigcup_{ \substack{ \vec{\gamma}= (\gamma_1,\ldots,\gamma_\ell) \in  (\bL_\bQ/\bL)^\ell  \\ 
\sum_{i=1}^\ell \gamma_i =\beta_G + \bL} }\fB_{g, \vec{\gamma} }(\beta_G). 
$$

\subsection{Moduli of sections} \label{sec:sections} 
In this subsection we fix $g,\ell, \beta_G$ and $\vec{\gamma} =(\gamma_1,\ldots,\gamma_\ell)$ and let $\fB=\fB_{g,\vec{\gamma} }(\beta_G)$.

Following \cite[Section 1]{BF97}, we introduce the following definition. Given a coherent sheaf $\cF$ of $\cO_X$-modules on 
an algebraic stack $X$, let $C(\cF):=\Spec_X(\Sym\cF^\vee)$ be the abelian cone associated to $\cF$. In particular, 
when $\cF$ is locally free, $C(\cF) = \tot(\cF)$ is the vector bundle associated to $\cF$. 

Let $\fP_{\fB}\to \fC_{\fB}$ be the universal principal $\Gamma$-bundle over the universal curve $\pi_{\fB}:\fC_{\fB} \to \fB$. 
Recall that $\Gamma$ is a subgroup of the diagonal torus  $\tT\subset \GL(V)$, so 
$$
\cV_{\fB}:= \fP_{\fB}\times_{\Gamma} V =\bigoplus_{i=1}^{n+{\boldsymbol{\kappa}}} \cL_{i,\fB}
$$
where each $\cL_{i,\fB}$ is a line bundle over $\fC_{\fB}$.  Consider the moduli of sections
$$
\fB_{g,\vec{\gamma} }(V,\beta_G)  := C\left(\pi_{\fB *} \cV_{\fB} \right) 
$$
which parametrizes $4$-tuples 
$\left( (\cC,\fz_1,\ldots, \fz_\ell), P,\rho, u \right)$  
where the triple $( (\cC,\fz_1,\ldots, \fz_\ell), P,\rho)$ is an object in 
$\fB(\bullet)$  and 
$$
u =(u_1,\ldots, u_{n+{\boldsymbol{\kappa}}}) \in H^0( \cC,P\times_{\Gamma} V) = H^0(\cC, \bigoplus_{j=1}^{n+{\boldsymbol{\kappa}}} \cL_j)
$$
where  $u_j \in H^0(\cC,\cL_j)$.   Let
$$
\fB_{g,\ell}(V,\beta_G) := C \left(\pi_{\fB_{g,\ell}(\beta_G) *}  \cV_{\fB_{g,\ell}(\beta_G) }\right).
$$
Then
$$
\fB_{g,\ell}(V,\beta_G) =\bigcup_{ \substack{ \vec{\gamma}= (\gamma_1,\ldots,\gamma_\ell) \in  (\bL_\bQ/\bL)^\ell  \\ 
\sum_{i=1}^\ell \gamma_i =\beta_G + \bL} }\fB_{g, \vec{\gamma} }(V, \beta_G). 
$$

Given a fixed triple $((\cC,\fz_1,\ldots, \fz_\ell), P,\rho)$, the space of infinitesimal deformations of the section $u_j$ of the line bundle $\cL_j$ is $H^0(\cC,\cL_j)$ and the space
of obstructions to deforming $u_j$ is $H^1(\cC,\cL_j)$.  We now compute
$$
\chi(\cC,\cL_j) = h^0(\cC,\cL_j) - h^1(\cC,\cL_j). 
$$
The line bundle
$\cL_j$ is of degree 
$$
\langle \D_j,\beta_G\rangle +\frac{q_j}{2}(2g-2+\ell)
$$
and has monodromy $e^{2\pi\sqrt{-1} \age_{\gamma_i}(\D_j)}$ around $\fz_i$, where
$\age_{\gamma_i}(\D_j) \in \bQ\cap [0,1)$ is the unique representative  in $[0,1)$
of the pairing $\langle \D_j, \beta_G\rangle \in \bQ/\bZ$.  Note that
$\age_{\gamma_i}(\D_j) =\age_{\fz_i}\cL_j$. 
We have
$$
\langle \D_j, \beta_G\rangle +\frac{q_j}{2}(2g-2+\ell) -\sum_{i=1}^\ell \age_{\gamma_i}(\D_j) \in \bZ. 
$$
By Kawasaki's orbifold version of the Riemann-Roch theorem \cite{Ka79}, 
$$
\chi(\cC,\cL_j)= \langle \D_j, \beta_G\rangle + \frac{q_j}{2}(2g-2+\ell)  + (1-g) -\sum_{i=1}^\ell \age_{\gamma_i}(\D_j). 
$$
Let
\begin{equation} \label{eqn:c1} 
(c_1)_G(V) =\sum_{j=1}^{n+{\boldsymbol{\kappa}}}\D_j\in \bL^\vee, \quad \age_{\gamma}(V) =\sum_{j=1}^{n+{\boldsymbol{\kappa}}} \age_{\gamma_i}(\D_j) \in \bQ,
\end{equation} 
and recall that $\hq= \frac{1}{2} \sum_{j=1}^{n+{\boldsymbol{\kappa}}}q_j$. Then
$$ 
\chi(\cC, P\times_\Gamma V) =\sum_{j=1}^{n+{\boldsymbol{\kappa}}} \chi(\cC,\cL_j)  
=   \langle (c_1)_G(V), \beta_G\rangle + (n+{\boldsymbol{\kappa}}-2\hq)(1-g)  - \sum_{i=1}^\ell  \left( \age_{\gamma_i}(V)- \hq \right). 
$$ 
There is a map $\pi_{\fS/\fB}: \fS =\fB_{g,\vec{\gamma} }(V,\beta_G) \lra \fB = \fB_{g,\vec{\gamma} }(\beta_G)$ given by 
$((\cC,\fz_1,\ldots,\fz_\ell), P, \rho,u)\mapsto ((\cC,\fz_1,\ldots,\fz_\ell), P, \rho)$, i.e, forgetting the section $u$.  The map 
$\pi_{\fS/\fB}$ is virtually smooth: there is a relative perfect obstruction theory $\bE_{\fS/\fB}$, where
$$
\bE_{\fS/\fB}^\vee =  \pi_{\fS/\fB}^*\bR\pi_{\fB*}\cV_{\fB}.
$$
The  relative virtual dimension of $\pi_{\fS/\fB}$ is 
\begin{equation}
d^\vir_{\fS/\fB} = \langle (c_1)_G(V),\beta_G\rangle +  (n+{\boldsymbol{\kappa}}-2\hq) (1-g) - \sum_{i=1}^\ell \left( \age_{\gamma_i}(V) - \hq \right). 
\end{equation}
Therefore, the moduli of sections $\fS$ is a possibly singular, but virtually smooth Artin stack; it is equipped with a perfect obstruction theory $\bE_{\fS}$
of virtual dimension 
$$
d^\vir_{\fS}  = d^\vir_{\fS/\fB} + \dim \fB =  \langle (c_1)_G(V),\beta_G\rangle +  (n+{\boldsymbol{\kappa}}- 2\hq) (1-g)  - \sum_{i=1}^\ell \left( \age_{\gamma_i}(V) -\hq \right)
+ ({\boldsymbol{\kappa}} + 3)(g-1) +\ell  
$$
which can be rewritten as
\begin{equation}\label{eqn:virdimS}
d^\vir_{\fS} =    \langle (c_1)_G(V), \beta_G\rangle +  (\hat{c}-3)(1-g) +\ell -  \sum_{i=1}^\ell \left( \age_{\gamma_i}(V) - \hq \right)
\end{equation} 
where $\hat{c} = n-2\hq$ is the central charge of the GLSM.  Note that if $\gamma \in \bL_\bQ/\bL$ corresponds to  $v\in \mathrm{Box}(\bSi_\zeta)$ then 
$$
\age_{\gamma}(V) = \age(v). 
$$
The definition of the central charge $\hat{c}$ and the degree shift $2 \left(\age(v) - \hq\right)$  in the definition of the A-model state spaces are motivated by the formula \eqref{eqn:virdimS} of the virtual dimension, which is consistent with \cite[Lemma 6.1.7]{FJR}.

Let $\fP_{\fS} \to \fC_{\fS}$ be the universal principal $\Gamma$-bundle
over the universal curve $\pi_{\fS}:\fC_{\fS} \to \fS$, and let $u_{\fS}: \fC_{\fS} \to  \cV_{\fS}:= \fP_{\fS}\times_{\Gamma} V$ be the universal section. 
We have the following commutative diagram. 
$$
\xymatrix{
\cV_{\fS} \ar[r]\ar[d] \ar@{}[dr]|-{\square}& \cV_{\fB} \ar[d]& \\
\fC_{\fS} \ar[r]\ar[d]^{\pi_{\fS}}   \ar@/^/[u]^{u_{\fS}} \ar@{}[dr]|-{\square}& \fC_{\fB} \ar[r]\ar[d]^{\pi_{\fB}}   \ar@{}[dr]|-{\square} & \fC_{\fM} \ar[d]^{\pi_{\fM}}\\
\fS \ar[r]^{\pi_{\fS/\fB}}  & \fB \ar[r]^{\pi_{\fB/\fM}}   &\fM
}
$$

\subsection{Landau-Ginzburg quasimaps and their moduli}
\begin{definition}[prestable LG quasimaps] 
A prestable genus-$g$, $\ell$-pointed, degree $\beta_G$ Landau-Ginzburg (LG)  quasimap 
to the 5-tuple $\ufX= (V,G,\bC_R^*, W, \zeta)$  is a 4-tuple $\Q = ( (\cC,\fz_1,\ldots, \fz_\ell), P, \rho, u)$, 
which is an object in $\fB_{g,\ell}(V,\beta_G)(\bullet)$ such that the base locus
$$
B(\Q):= u^{-1}(P\times_{\Gamma} V^{us}_G(\zeta))  \subset \cC
$$
of $\Q$ is purely zero-dimensional  and is disjoint  from the marked points and the nodes in $\cC$.  
\end{definition} 
\begin{remark} Note that the $\tT$-action preserves $V^{ss}_G(\zeta)$ and $V^{us}_G(\zeta)$. In particular, $\Gamma$ acts on $V^{us}_G(\zeta)$. 
\end{remark}

Let $LG_{g,\ell}^{\pre}(\ufX,\beta_G)$ be the moduli  of prestable genus-$g$, $\ell$-pointed, degree $\beta_G$ LG quasimaps to $\ufX$. It is an open substack 
of $\fB_{g,\ell}(V,\beta_G)$. There are evaluation maps
$$
\ev_i: LG_{g,\ell}^{\pre}(\ufX,\beta_G) \to I\sX_\zeta =\bigcup_{v\in \mathrm{Box}(\bSi)}\sX_{\zeta,v} , \quad i=1,\ldots, \ell. 
$$
Given $\vec{v}=(v_1,\ldots, v_\ell) \in \mathrm{Box}(\bSi_\zeta)^\ell$, let
$$
LG_{g,\vec{v}} (\ufX,\beta_G) := \bigcap_{i=1}^\ell \ev_i^{-1}\left(\cX_{\zeta,v_i}\right).
$$
Then $LG_{g,\vec{v}}(\ufX,\beta_G)$ is an open substack of $\fB_{g,\vec{\gamma} =(\gamma_1,\ldots,\gamma_\ell)}(V,\beta_G)$, 
where $\gamma_i \in \bL_\bQ/\bL$ corresponds to $v_i\in \mathrm{Box}(\bSi)$. Therefore,  the virtual dimension of
$LG_{g,\vec{v}}(\ufX,\beta_G)$ is 
\begin{equation} \label{eqn:LGvirtual} 
\langle (c_1)_G(V), \beta_G\rangle + (\hat{c}-3)(1-g) +\ell - \sum_{i=1}^\ell \left(\age(v)- \hq\right).
\end{equation}

\begin{definition}[good lift]
 $\tzeta \in \tbL^\vee$ is a  {\em lift} of $\zeta \in \bL^\vee$ if $\zeta$ is the image of $\tzeta$ under 
$$
\tbL^\vee =\Hom(\Gamma,\bC^*) \to \bL^\vee=\Hom(G,\bC^*).
$$
$\tzeta$ is a {\em good lift} of $\zeta$ if $V_\Gamma^{ss}(\tzeta) = V_G^{ss}(\zeta)$. 
\end{definition} 

For any $\tzeta \in \Hom(\Gamma,\bC^*) = \tbL^\vee$, let $\chi^{\tzeta}:\Gamma\to \bC^*$ denote the corresponding $\Gamma$-character, let
$\bfL_{\tzeta} \in \Pic_{\Gamma}(V)$ denote the $\Gamma$-equivariant line bundle on $V$ determined by $\chi^{\tzeta}$
$$
\chi^{\tzeta_1+\tzeta_2} = \chi^{\tzeta_1} \chi^{\tzeta_2},\quad \bfL_{\tzeta_1 + \tzeta_2} = \bfL_{\tzeta_1}\otimes \bfL_{\tzeta_2}. 
$$
A section $s \in H^0(V, \bfL_{\tzeta})^\Gamma$ defines a $\Gamma$-equivariant map $V\to \bC$ which induces a morphism
$$
s: P\times_{\Gamma}V \to P\times_{\chi^{\tzeta}} \bC. 
$$
Given any $u \in H^0(\cC,P\times_\Gamma V)$,  let $u^*s := s\circ u \in H^0(\cC, P\times_{\chi^{\zeta}} \bC)$.   The following definition 
is  \cite[Definition 4.2.10]{FJR} (which is essentially \cite[Definition 7.1.1]{CKM})  in slightly different notation. 

\begin{definition}[length] 
Let $\Q= ((\cC,\fz_1,\ldots, \fz_\ell), P,\rho,u)$ be a prestable LG quasimap to $\ufX=(V,G, \bC_R^*, W, \zeta)$,  let $\tzeta \in \tbL^\vee$ be a good lift of $\zeta\in \bL^\vee$.
The length of a point $y$ in $\cC$ with respect to  $\Q$ and $\tzeta$ is defined to be
\begin{equation}
\ell_y(\Q,\tzeta):= \min\Big\{ \frac{\ord_y(u^*s)}{m} \; \Big| \; s\in H^0(V,\bfL_{m\tzeta} =\bfL_{\tzeta}^ {\otimes m})^\Gamma, m\in \bZ_{>0} \Big\}
\end{equation} 
where $\ord_y(u^*s)$ is the order of vanishing of the section $u^*s \in H^0(\cC, P\times_{\chi^{m\tzeta}}\bC)$ at $y$.
\end{definition}

\begin{definition}[$\epsilon$-stable LG quasimaps]  \label{df:LG} 
Let $\tzeta  \in \Hom(\Gamma,\bC^*)$ be a good lift of $\zeta\in \Hom(G,\bC^*)$, and let $\epsilon$ be a positive rational number. 
A prestable LG quasimap $\Q= ((\cC,\fz_1,\ldots, \fz_\ell), P,\rho,u)$ is {\em $\epsilon$-stable 
with respect to $\tzeta$}  if 
\begin{enumerate}
 \item  $\omega_{\cC}^{\log} \otimes  (P\times_{\tzeta}\bC)^\epsilon \in \Pic(\cC)\otimes_{\bZ}\bQ$ is an ample $\bQ$ line bundle on $\cC$, and
 \item  $\epsilon \ell_y(\Q,\tzeta) \leq 1$ for every $y \in \cC$. 
\end{enumerate} 
\end{definition} 

\begin{remark}
Let $\cC_v$  (respectively $C_v$) be the connected component of the normalization of $\cC$  (respectively $C$) associated to a vertex $v$
in the dual graph of  the coarse moduli  $C$ of $\cC$. Let
$g_v$ be the genus of $C_v$ and let $\ell_v$ be the number of points on $C_v$ mapped to a marked point
or a node under the normalization map, and let $\beta_\Gamma(v) \in H_2(B\Gamma;\bQ) =\tbL_\bQ$ be the degree of  $\cC_v\to \cC \stackrel{P}{\to} B\Gamma$. 
Condition (1) in Definition \ref{df:LG}  is equivalent to the following condition:
$$
2g_v -2 +\ell_v +\epsilon \int_{\beta_\Gamma(v)} \tzeta >0
\quad \text{for all vertex $v$ in the dual graph of $C$}. 
$$
\end{remark}

\begin{remark}
Let $m\in \bZ_{>0}$, $\epsilon \in \bQ_{>0}$, and $\tzeta\in \tbL^\vee$.  Then $\Q$ is $\epsilon$-stable with respect to $\tzeta$ iff it is $(\epsilon/m)$-stable with respect to $m\tzeta$;
see \cite[Remark 7.1.4]{CKM} for the analogous statement in quasimap theory.  Given $\nu\in \tbL^\vee_{\bQ}$, choose $m\in \bZ_{>0}$ such that $m\nu \in \tbL^\vee$ and define $\Q$ to be $\nu$-stable if it is $(1/m)$-stable with respect to $m\nu$; the definition is independent of 
the choice of $m$ and agrees with \cite[Definition 2.6]{FK}.
\end{remark} 

Let $LG_{g,\ell}^{\pre} (\underline{\fX}, \beta_G)$  (respectively
$LG_{g,\ell}^{\epsilon, \tzeta}(\underline{\fX},\beta_G)$) be the moduli of genus-$g$, $\ell$-pointed, degree $\beta_G$ prestable (respectively $\epsilon $-stable with respect to $\tzeta$) quasimaps  
to  $\underline{\fX}:= (V,G,\bC_\bR^*, W,\zeta)$.  More generally, let $Y$ be a $\Gamma$-invariant closed subscheme of $V$ such that $Y\cap V^{s}_G(\zeta) = Y\cap V^{ss}_G(\zeta)$ is non-empty (e.g. $Y=\mathrm{Crit}(W)$), and let
$\underline{\fY}=(Y, G, \bC_\bR^*, W, \zeta)$. Let $LG_{g,\ell}^\pre(\underline{\fY}, \beta_G)$ be the closed substack of $LG_{g,\ell}^\pre(\underline{\fX},\beta_G)$ 
parametrizing $\Q= ((\cC,\fz_1,\ldots, \fz_\ell), P,\rho,u)$ such that $u: \cC\to [Y/\Gamma] \subset [V/\Gamma]$,  and define $LG_{g,\ell}^{\epsilon,\tzeta}(\underline{\fY}, \beta_G)$ similarly. 
Fan-Jarvis-Ruan proved the following result: 

\begin{Theorem}[{\cite{FJR}}] \label{thm:proper} 
$LG^{\epsilon,\tzeta}_{g,\ell}(\underline{\fY},  \beta_G)$  is a separated Deligne-Mumford stack of finite type.  It is proper over $\Spec\bC$ if $[ (Y\cap V^{ss}_G(\zeta) )/G]$ is. 
\end{Theorem} 

Given $\vec{v}=(v_1,\ldots, v_\ell) \in \mathrm{Box}(\bSi_\zeta)^\ell$, let
$$
LG_{g,\vec{v}}^{\epsilon,\tzeta}(\underline{\fX}, \beta_G) : = LG_{g,\ell}^{\epsilon,\tzeta}(\underline{\fX},\beta_G)\cap LG_{g,\vec{v}}(\underline{\fX}, \beta_G).
$$
Then
$$
LG_{g,\ell}^{\epsilon,\tzeta}(\underline{\fX},\beta_G) = \bigcup_{\vec{v} \in \mathrm{Box}(\bSi_\zeta)^\ell} 
LG_{g,\vec{v}}^{\epsilon,\tzeta}(\underline{\fX},\beta_G). 
$$
 In Section \ref{sec:cosection} and Section \ref{sec:factorization} below, we fix $g, \vec{v}, \epsilon, \beta_G$, and let
$$
\bfX   = LG_{g,\vec{v}}^{\epsilon,\tzeta}(\underline{\fX},\beta_G), \quad
\bfZ  = LG_{g,\vec{v}}^{\epsilon,\tzeta}(\underline{\fZ}, \beta_G)
$$
where $\underline{\fZ} = (\Crit(W), G, \bC_R^*, W,\zeta)$.  By Theorem \ref{thm:proper}, if 
$$
\sZ_\zeta = [\left(\Crit(W)\cap V_G^{ss}(\zeta)\right) /G] = \mathrm{Crit}(\bw_\zeta) 
$$
is proper, then $\bfZ$ is proper. 

\subsection{Cosection localized virtual cycle and cosection localized virtual structure sheaf}  \label{sec:cosection} 
Recall that $v\in \mathrm{Box}(\bSi_\zeta)$ is narrow if $\sX_{\zeta,v}$ is compact.  We assume $v_1,\ldots, v_\ell$ are narrow in
this subsection. Under this assumption, Fan-Jarvis-Ruan \cite{FJR} constructed a cosection $\delta: \cO b_{\bfX} \to \cO_{\bfX}$ whose zero locus is $\bfZ$.  Applying \cite{KL13}, they obtain a {\em cosection localized virtual cycle}
$$
[\bfX]^\vir_\text{loc} \in A_*(\bfZ;\bQ)
$$
such that 
$$
\iota_*[\bfX]^\vir_\text{loc}  = [\bfX]^\vir \in A_*(\bfX;\bQ)
$$
where $\iota: \bfZ\to \bfX$ is the inclusion, and $[\bfX]^\vir$ is the Behrend-Fantechi {\em virtual fundamental class} \cite{BF97} defined by the perfect obstruction theory described in previous subsections.  When $\bfZ$ is proper, $[\bfX]^\vir_\text{loc}$ can be used to define (cohomological) GLSM invariants in the narrow sector\footnote{In \cite[Section 6]{FJR}, GLSM correlators are defined for compact type insertions \cite[Definition 4.1.4]{FJR} which are more general than narrow insertions. See \cite{Sh} for subtleties of defining GLSM invariants involving compact type insertions which are not narrow, as well as an alternative construction of genus-zero compact type GLSM invariants under additional assumptions.}.  The construction of cosection localized virtual cycle in the narrow sector in \cite{FJR} can  be viewed as generalization of Chang-Li-Li's construction of Witten's top Chern class via cosection localization \cite{CLL}.

We now consider the following particularly nice case, as in \cite[Proposition  5.6]{BF97}. 
\begin{situation}\label{nice}
$\bfX$ is smooth of dimension $r_0$ and $\cO b_{\bfX}$ is locally free of rank $r_1$. 
\end{situation} 

In Situation \ref{nice}, $\Ob_{\bfX} : = \tot(\cO b_{\bfX})$ is a vector bundle over $\bfX$ of rank $r_1$, called the obstruction bundle, and the virtual dimension is
$r=r_0-r_1$. Let $\delta^\vee: \cO_{\bfX} \to \cO b^\vee_{\bf X}$ be the dual of the cosection, which is a section of $\cO b^\vee_{\bfX}$.  
\begin{eqnarray}
[\bfX]^\vir &=& c_{r_1}(\Ob_{\bfX}) \cap [\bfX]  = (-1)^{r_1} c_{r_1}(\Ob_{\bfX}^\vee)\cap [\bfX] \in A_r (\bfX;\bQ),  \\  \label{eqn:virtual-cycle}
{[\bfX]^\vir_{\text{loc}} } &=& (-1)^{r_1} c_{r_1}(\Ob_{\bfX}^\vee,\delta^\vee) \cap [\bfX] \in A_r(\bfZ;\bQ). 
\end{eqnarray}
where
\begin{itemize}
\item $[\bfX] \in A_{r_0}(\bfX;\bQ)$ is the fundamental class of the smooth DM stack $\bfX$, 
\item $c_{r_1}(\Ob_{\bfX}^\vee)\cap -: A_k(\bfX;\bQ)\to A_{k-r_1}(\bfX;\bQ)$ is the top Chern class, and 
\item $c_{r_1}(\Ob_{\bfX}^\vee,\delta^\vee)\cap - : A_k(\bfX;\bQ)\to A_{k-r_1}(\bfZ;\bQ)$ is the localized top Chern class
 \cite[Chapter 14]{Fu98}. 
\end{itemize} 

Let $K_0(\bfX)$ (resp. $K^0(\bfX)$) denote the Grothendieck group generated by coherent sheaves (resp. locally free sheaves) on $\bfX$
with relations $[F]=[F']+[F'']$ whenever there is a short exact sequence $0\to F'\to F\to F'' \to 0$. 
Applying \cite{KL18}, one obtains a {\em cosection localized virtual structure sheaf}
$$
\cO^\vir_{\bfX, \text{loc}}\in K_0(\bfZ)
$$
such that
$$
\iota_* \cO^\vir_{\bfX,\text{loc}} = \cO^\vir_{\bfX} \in K_0(\bfX)
$$
where $\cO^\vir_{\bfX}$ is the {\em virtual structure sheaf} defined in \cite{BF97} (see \cite[Remark 5.4]{BF97}) and \cite{Lee}. 
When $\bfZ$ is proper, $\cO^\vir_\text{loc}$ an be used to define K-theoretic GLSM invariants in the narrow sector. 

In Situation \ref{nice}, 
$$
\cO^{\vir}_{\bfX} =\sum_{i=0}^{r_1} (-1)^i \wedge^i \cO b_{\bfX}^\vee = (-1)^{r_1} \det(\cO b^\vee_{\bfX}) \sum_{i=0}^{r_1}
(-1)^i \wedge^i \cO b_{\bfX} \in K_0(\bfX). 
$$
Note that
$$
\td(\Ob_{\bfX}) \ch(\cO^\vir_{\bfX}) = c_{r_1}(\Ob_{\bfX}), 
$$
so 
$$
[\bfX]^\vir = \td(\Ob_{\bfX}) \ch(\cO^\vir_{\bfX})\cap [\bfX]. 
$$

The section $\delta^\vee:  \cO_{\bfX} \to \cO b^\vee$ defines a Koszul complex 
\begin{equation} \label{eqn:Koszul} 
\bK(\delta^\vee):= \Sym^{r_1}\left(\cO b_{\bfX} \stackrel{i_{\delta^\vee}}{\to}  \cO_{\bfX}\right)
= \big[ 0\to \wedge^{r_1} \cO b_{\bfX} \stackrel{i_{\delta^\vee} }{\to} \wedge^{r_1-1} \cO b_{\bfX} \to
\cdots \to  \wedge^1 \cO b_{\bfX} \stackrel{i_{\delta^\vee}}{\to} \cO_{\bfX} \to 0 \big]
\end{equation} 
which is exact on $\bfX-\bfZ$. If $e_1,\ldots, e_k$ are local sections of $\cO b_{\bfX}$ then
$$
i_{\delta^\vee} (e_1\wedge \cdots \wedge e_k) =  \sum_{i=1}^k (-1)^{i-1}  \langle \delta^\vee, e_i \rangle  e_1\wedge \cdots \wedge e_{i-1} \wedge e_{i+1} \wedge \cdots \wedge e_k. 
$$
Note that
$$
\td(\Ob_{\bfX}^\vee)\ch^{\bfX}_{\bfZ}(\bK(\delta^\vee)) = c_{r_1}(\Ob_{\bfX}^\vee, \delta^\vee), 
$$
where  $\ch^{\bfX}_{\bfZ}(\bK(\delta^\vee))\cap : A_*(\bfX;\bQ)\to A_*(\bfZ;\bQ)$
is the localized Chern character \cite[Chapter 18]{Fu98}.

\subsection{Virtual factorization} \label{sec:factorization} 

Favero-Kim \cite{FK} construct GLSM invariants for general choice of stability in both narrow and broad sectors via matrix factorization, generalizing constructions in \cite{PV16, CFGKS}.  In this subsection we briefly describe the construction (in slightly different notation).  

\subsubsection{The Artin stack $\bfC$}
Let $\vec{\gamma} =(\gamma_1,\ldots,\gamma_\ell)$, where $\gamma_i\in \bL_\bQ/\bL$ corresponds to $v_i\in \mathrm{Box}(\bSi_\zeta)$.
Let $\fB:=\fB_{g,\vec{\gamma}}(\beta_G)$ be the universal moduli of $\Gamma$-structures defined in Section \ref{sec:Gamma-structures}.
Let $\fP_{\fB}\to \fC_{\fB}$ be the universal principal $\Gamma$-bundle over the universal curve $\pi_{\fB}:\fC_{\fB}\to \fB$, let
$\cV_{\fB}=\fP_{\fB}\times_\Gamma V$,  and let $C(\pi_{\fB *}\cV_{\fB})$ be the moduli of sections, which is an abelian cone over $\fB$, 
as in Section \ref{sec:sections}.  Then $\bfX$ is an open substack of $C(\pi_{\fB*}\cV_{\fB})$,  and the image of $\bfX$ under
the projection $C(\pi_{\fB*}\cV_{\fB})\to \fB$ (forgetting the section) is contained in a finite type open substack $\fB^\circ \subset \fB$. 
Therefore, $\bfX$ is an open substack of 
$$
\bfC:= C\left(\pi_{\fB^\circ *} \cV_{\fB^\circ}  \right)= \fB^\circ \times_{\fB} C\left(\pi_{\fB*}\cV_{\fB}\right). 
$$
For $i=1,\ldots, \ell$, there are evaluations maps 
\begin{equation}
\ev^{\bfC}_i:\bfC \lra \fX_{v_i}:= [V^{g(v_i)}/G]
\end{equation}
which restricts to 
\begin{equation}
\ev_i: \bfX \lra \sX_{\zeta, v_i} = \left[ \big(V^{g(v_i)}\cap V^{ss}_G(\zeta)\big)/G\right]  \subset \fX_{v_i}.
\end{equation} 

\subsubsection{The smooth Artin stack $\bfA$ and the smooth DM stack $\bfU$} 
Over the finite type smooth Artin stack $\fB^\circ$,  $\bR \pi_{\fB^\circ *}\cV_{\fB^\circ}$ admits a global resolution 
$$
\bR \pi_{\fB^\circ *}\cV_{\fB^\circ} = [A \stackrel{d_A}{\lra} B]
$$
where $A, B$ are locally free sheaves of $\cO_{\fB^\circ}$-modules.
Let $\pi_{\bfA/\fB^\circ}: \bfA:= \tot(A)\to \fB^\circ$ be the projection, let $t_A \in \Gamma(\bfA ,  \pi_{\bfA/\fB^\circ}^*A)$ be the tautological section, and let 
$\beta_{\bfA} := \left(\pi_{\bfA/\fB^\circ}^*d_A \right)\circ t_A \in \Gamma(\bfA, \pi_{\bfA/\fB^\circ}^*B)$. The zero locus of $t_A$ is the zero section
in $\bfA=\tot(A)$, and the zero locus of $\beta_{\bfA}$ is  $\bfC$. 
There exists an open substack $\bfU\subset \bfA$ such that $\bfU$ is a DM stack of finite type and the following diagram 
is a Cartesian square. 
$$
\xymatrix{
\bfX = Z(\beta_{\bfU}) \ar[r]^{\quad \iota_{\bfX}}    \ar[d]^{j_{\bfX} }  & \bfU \ar[d]^{j_{\bfU} } \\
\bfC = Z(\beta_{\bfA}) \ar[r] ^{\quad \iota_{\bfC}} &  \bfA
}
$$
In the above Cartesian diagram, the two vertical arrows are open embeddings, the two horizontal arrows are closed embeddings, and
$\beta_{\bfU}= j_{\bfU}^*\beta_{\bfA} \in \Gamma(\bfU,  B_{\bfU})$, where $B_{\bfU}:= j_{\bfU}^*\pi_{\bfA/\fB^\circ}^*B$.  The virtual dimension $r$ of $\bfX$ is
$$
r = \dim \fB + \rank A -\rank B   = \dim \bfU - \rank B.
$$
We have
$$
[\bfX]^\vir =  c_{\rank B}\left( B_{\bfU}, \beta_{\bfU} \right)\cap [\bfU],
$$
where
\begin{itemize}
\item $[\bfU]\in A_{r+\rank B}(\bfU;\bQ)$ is the fundamental class of the smooth DM stack $\bfU$, 
\item $c_{\rank B}\left(B_{\bfU}, \beta_{\bfU}\right) \cap - : A_{r+\rank B}(\bfU)\to A_r(\bfX)$
is the localized top Chern class, and 
\item $[\bfX]^\vir \in A_r(\bfX;\bQ)$ is the Behrend-Fantechi virtual fundamental class of $\bfX$.
\end{itemize}

\subsubsection{Evaluation maps}
For $i=1,\ldots, \ell$, the evaluation map $\ev_i^{\bfC}:\bfC\to \fX_{v_i}$ extends to $\ev_i^{\bfA}: \bfA \to  \fX_{v_i}$ which restricts to 
$\ev_i^{\bfU}:\bfU\to \sX_{\zeta,v_i}$,  so that we have the following commutative diagram
$$
\xymatrix{
\bfX \ar[r]^{\iota_{\bfX}} \ar[d]  & \bfU\ar[r]^{\ev_i^{\bfU}} \ar[r]^{\ev_i^{\bfU}} \ar[d] & \sX_{\zeta, v_i}    \ar[d]\\
\bfC \ar[r]^{\iota_{\bfC}} & \bfA \ar[r]^{\ev_i^{\bfA}\quad}  &  \fX_{v_i} 
}
$$
where $\ev_i^{\bfU}\circ \iota_{\bfX} =\ev_i$ and $\ev_i^{\bfA}\circ \iota_{\bfC} = \ev_i^{\bfC}$, and all the vertical arrows are open embeddings. 
Let 
$$
\vec{v}=(v_1,\ldots, v_\ell),\quad \fX_{\vec{v}} := \prod_{i=1}^\ell \fX_{v_i},\quad \sX_{\zeta, \vec{v}} := \prod_{i=1}^\ell \sX_{\zeta,v_i}.
$$
$\bfA$ and $\ev_i^{\bfA}$ are chosen such that 
\begin{equation}\label{eqn:evA}
\ev^{\bfA}:= \prod_{i=1}^\ell \ev_i^{\bfA}: \bfA \to  \fX_{\vec{v}}
\end{equation} 
is a surjective smooth map between smooth Artin stacks. 
More explicitly, given any object  $\xi = ((\cC,\fz_1,\ldots, \fz_\ell), P, \rho)$ 
in $\fB^\circ(\bullet)$, which corresponds to a morphism $\Spec\bC \to \fB^\circ$, let
$\bfA_\xi := \Spec\bC\times_{\fB^\circ} \bfA$ be the fiber of $\bfA =\tot(A)$ over $\xi$. We have the following linear maps between complex vector spaces:
\begin{equation}
H^0(\cC,P\times_{\Gamma}V) \stackrel{\iota_{\bfC,\xi}}{\lra}   \bfA_\xi \stackrel{\ev^{\bfU}_\xi}{\lra}  
\bigoplus_{i=1}^\ell H^0(\fz_i, \left(P\times_{\Gamma} V)|_{\fz_i}\right)  = \bigoplus_{i=1}^\ell V^{g(v_i)}
\end{equation} 
where $\iota_{\bfC,\xi}$ is injective and $\ev^{\bfU}_\xi$ is surjective.  As a consequence, 
\begin{equation} \label{eqn:evU} 
\ev^{\bfU}:= \prod_{i=1}^\ell \ev_i^{\bfU}: \bfU \to \sX_{\zeta,\vec{v}} 
\end{equation}
is a smooth map between smooth DM stacks.  

\subsubsection{Superpotentials and matrix factorizations}
Let $\bw_{v_i}:\fX_{v_i} = [V^{g(v_i)}/G] \to  \bC$ denote the restriction of $\bw: [V/G]\to \bC$.
Define a superpotential $\bw_{\bfA}$ on $\bfA$: 
\begin{equation}\label{eqn:wA} 
\bw_{\bfA}:= \sum_{i=1}^\ell  (\ev_i^{\bfA})^* \bw_{v_i} \in \Gamma(\bfA, \cO_{\bfA})
\end{equation}
which restricts to a superpotential on $\bfU$:
\begin{equation}\label{eqn:wU}
\bw_{\bfU}:= j_{\bfU}^*\bw_{\bfA}= \sum_{i=1}^\ell (\ev_i^{\bfU})^*\bw_{\zeta,v_i}  \in \Gamma(\bfU,\cO_{\bfU}). 
\end{equation}
The sum of residues of a meromorphic 1-form on a curve is zero, so 
$$
\iota_{\bfC}^*\bw_{\bfA}=0, \quad \iota_{\bfX}^*\bw_{\bfU}  =  0. 
$$

When the GLSM $\ufX$ is a convex hypbrid model, it is shown in \cite{CFGKS} that
\begin{enumerate}
\item[(a)] $\bfU$ can be chosen to be separated over $\Spec\bC$.
\item[(b)] There exists a cosection $\alpha^\vee_{\bfA}: \pi_{\bfA/\fB^\circ}^*B\to \cO_{\bfA}$, or equivalently a section 
$\alpha_{\bfA} : \cO_{\bfA}\to \pi_{\bfA/\fB^\circ}^*B^\vee$, such that 
$$
\langle \alpha_{\bfA}, \beta_{\bfA} \rangle = - \bw_{\bfA}.
$$
\item[(c)] Let $\alpha_{\bfU} := j_{\bfU}^*\alpha_{\bfA} \in \Gamma(\bfU, B_{\bfU}^\vee)$. Then 
$$
\bfZ = Z(\alpha_{\bfU})\cap Z(\beta_{\bfU}) \subset \bfX = Z(\beta_{\bfU})\subset \bfU. 
$$
\end{enumerate} 
 In Section \ref{sec:singular}, we will provide explicit construction of $\bfU$ and $\alpha_{\bfA}^\vee$ satisfying (a)-(c) for the genus-zero one-pointed moduli spaces used to define the GLSM $I$-functions for all abelian GLSMs. 

When $\alpha_{\bfA}$ exists,  one obtains a Koszul matrix factorization $\{ \alpha_{\bfA},\beta_{\bfA} \}$ of $(\bfA, -\bw_{\bfA})$ defined by 
$$
\{ \alpha_{\bfA}, \beta_{\bfA} \} = \Big[
\xymatrix{
\bigoplus_i \Lambda^{2i} \pi_{\bfA/\fB^\circ}^*B^\vee \ar@/{ }^{1pc}/[r]^\partial &   \bigoplus_i \Lambda^{2i+1} \pi_{\bfA/\fB^\circ}^*B^\vee \ar@/{ }^{1pc}/[l]^\partial 
}
\Big],\quad \text{where }  \partial = i_{\beta_{\bfA}} + \alpha_{\bfA}\wedge.
$$
Then
$$
\bK_{\bfU}:= \{ \alpha_{\bfU},\beta_{\bfU}\} = j_{\bfU}^*\{\alpha_{\bfA}, \beta_{\bfA}\}
$$
is a Koszul matrix factorization of $(\bfU, -\bw_{\bfU})$. It is called the {\em fundamental factorization} in \cite{PV16, CFGKS}
and called the {\em virtual factorization} in \cite{FK}.

\subsubsection{Localized Chern character and the virtual fundamental class} 
Let  $U$ be a smooth DM stack over $\bC$ and let $w:U\to \bC$ be a regular function.
In \cite[Section B.4]{FK}, Favero-Kim define the Atiyah class and the localized Chern character
of a matrix factorization for $(U,w)$, following the construction of Kim-Polishchuk \cite{KP22}
when $U$ is a smooth scheme. Applying the definition to $\bK_{\bfU}$, one obtains a {\em localized Chern character}
$$
\ch^{\bfU}_{\bfZ} \bK_{\bfU} \in \bH^{even}_{\bfZ}\left(\bfU, (\Omega^\bullet_{\bfU}, -d\bw_{\bfU})\right). 
$$
The virtual fundamental class of $\bfX$ is defined to be
$$
[\bfU]^\vir_{\bw} := \big(\prod_{i=1}^\ell r_i \big) \td\ch_{\bfZ}^{\bfU} \bK_{\bfU} \in \bH^{even}_{\bfZ}\left(\bfU, (\Omega^\bullet_{\bfU}, -d\bw_{\bfU})\right).
$$
where $r_i = \ord(g(v_i))$, or equivalently $\fz_i = B\mu_{r_i}$.

\subsection{Effective classes}
Given $\ufX=(V,G,\bC_R^*,W,\zeta)$, and $g,\ell\in \bZ_{\geq 0}$, define
$$
\bK^\eff(\ufX)_{g,\ell}  := \{ \beta_G \in \bL_\bQ:  LG_{g,\ell}^{\pre}(\ufX,\beta_G) \text{ is nonempty}\}. 
$$
We now give a more explicit description of  $K^\eff(\fX)_{g,\ell}$ when $\ell=1$.  Let $((\cC,\fz), P, \rho,u)$ be
an object in $LG_{g,1}^{\pre}(\ufX,\beta_G)$.  Then 
$$
u(\fz) \in \sX_\zeta = \bigcup_{I\in \cA_\zeta^{\min} }\sX_I.
$$
If $u(\fz) \in \sX_I$ then $u_i (x)\neq 0$ for $i\in I$. We observe that
$$
u_i(\fz)\neq 0 \Rightarrow  \deg \cL_i \geq 0 \text{ and } \age_{\fz}(\cL)=0
\Leftrightarrow \deg \cL_i \in \bZ_{\geq 0}
$$
since $\deg \cL_i -\age_{\fz}(\cL)\in \bZ$. Therefore, 
$$
\bK^\eff(\ufX)_{g,1}\subset    \bigcup_{I\in \cA_\zeta^{\min}} \bK^\eff_I(\ufX)_{g,1},
$$  
where
\begin{equation}  
\bK^\eff_I(\ufX)_{g,1}= \{\beta_G \in \bL_{\bQ}: \langle \D_i,\beta_G\rangle + \frac{q_i}{2} (2g-1) \in \bZ_{\geq 0} \textup{ for all } i\in I \}.
\end{equation} 
Indeed, it is not hard to see that
\begin{equation} 
\bK^\eff(\ufX)_{g,1} = \bigcup_{I\in \cA_\zeta^{\min} }\bK^\eff_I(\ufX)_{g,1}. 
\end{equation}

Let $\{ \D_i^{*I}: i\in I\}$ be the $\bQ$-basis of $\bL_{\bQ}\cong \bQ^{\boldsymbol{\kappa}}$ dual to the $\bQ$-basis
$\{ \D_i: i\in I\}$ of $\bL^\vee_\bQ$: for any $i,j\in I$, 
$$
\langle \D_i,  \D_j^{*I} \rangle =\delta_{ij} 
$$
Then
$$
\bK^\eff_I(\sX_\zeta ,\bw_\zeta)_{g,1} = \Big\{ \sum_{i\in I} (m_i -  \frac{q_i}{2} (2g-1)) \D_i^{*I}:  m_i \in \bZ_{\geq 0}  \Big \}  
$$

\subsection{Stacky loop spaces} In orbifold quasimap theory \cite{CKK}, the  $J^\epsilon$-function 
is defined via $\bC^*$ localization on genus-zero  $\epsilon$-stable quasimap graph spaces.  
 When the gauge group $G=(\bC^*)^{\boldsymbol{\kappa}}$
is an algebraic torus, the small $I$-function $I= J^{0+}|_{t=0}$ can be computed completely explicitly by $\bC^*$ localization on stacky loop spaces.  Similarly, we may define the small $J^\epsilon$-function of a GLSM  via $\bC^*$ localization on genus-zero $\epsilon$-stable LG quasimap graph spaces. When the gauge group $G=(\bC^*)^{\boldsymbol{\kappa}}$ is an algebraic torus, the small $I$-function $I=J^{0_+}|_{t=0}$ of the GLSM can be computed by 
$\bC^*$ localization on the LG version of stacky loop spaces.  In this paper, we define
and compute $K$-theoretic and cohomological $I$-functions via $\bC^*$ localization on stacky loop spaces defined in this subsection. 

\subsubsection{Classical version}
Given any subset $I$ of $\{1,\ldots, n+{\boldsymbol{\kappa}}\}$, define
$$
V_I = \{ (x_1,\ldots,x_{{\boldsymbol{\kappa}}+n}) \in V=\bC^{n+{\boldsymbol{\kappa}}}\mid x_i\neq 0 \textup{ if } i \in I\} = \bC^{I'} \times (\bC^*)^I
$$
where $I' =\{ 1,\ldots,n+{\boldsymbol{\kappa}}\}\setminus I$.  Then
$$
V^{ss}_G(\zeta) = \bigcup_{I\in \cA_\zeta^{\min}} V_I, 
$$ 
and 
$$
\sX_\zeta = [V^{ss}_G(\zeta)/G] = \bigcup_{I\in \cA_\zeta^{\min}} \sX_I
$$
where $\sX_I = [V_I/G] \subset \sX_\zeta$ is an open toric substack which is an affine toric orbifold
defined by the $n$-dimensional cone $\sigma_{I'}$ spanned by $\{ v_i: i \in I'\}$. $\cX_I$ contains a unique torus fixed (possibly stacky) point
$$
\fp_I = \left[ \left( \{0\}^{I'} \times (\bC^*)^I\right) \Big/ G \right] = \left[  \left( \{0\}^{\bar I} \times \{1\}^I \right) \Big/ G_{\sigma_{I'}}
\right] \cong B G_{\sigma_{I'}}
$$
where $G_{\sigma_{I'}}  \subset G$ is the stabilizer of the point
$\{0\}^{I'} \times \{1\}^I \in V$ (see Section \ref{sec:closed-substacks}).  Let 
$$
N_{\sigma_{I'}} = \bigoplus_{i\in I} \bZ v_i. 
$$
Then $G_{\sigma_{I'}}$ is a finite abelian group, and 
$$
G_{\sigma_{I'}}  \cong N/N_{\sigma_{I'}},\quad \Hom(G_{\sigma_{I'}} ,\bC^*)\cong \bL^\vee\Big/ \big(\bigoplus_{i\in I} \bZ \D_i  \big).
$$

\subsubsection{Quantum version}
We introduce the following convention. Given any rational number $m/a$, where $m\in \bZ$, $a\in \bZ_{>0}$, and 
$m,a$ are coprime, we define 
$$
H^i(\bP^1, \cO_{\bP^1}(\frac{m}{a}) ) := H^i(\bP[a,1], \cO_{\bP[a,1]}(m)),\quad i=0,1.
$$
Recall that the total space of $\cO_{\bP^1[a,1]}(m)$ is $[  \left( (\bC^2-\{0\})\times\bC  \right) /\bC^*]$, where $\bC^*$ acts by weights $(a,1,m)$. 

Given any   $\beta  \in \bK^{\eff}(\sX_\zeta,\bw_\zeta)_{0,1}$, we let
\begin{equation} \label{eqn:V} 
V_\beta = \bigoplus_{i=1}^{n+{\boldsymbol{\kappa}}} V_{\beta,i},  \quad \text{where } V_{\beta,i}= H^0\left(\bP^1, \cO_{\bP^1}(\langle \D_i, \beta\rangle -\frac{q_i}{2})\right),
\end{equation}
and let 
\begin{equation} \label{eqn:W}
W_\beta= \bigoplus_{i=1}^{n+{\boldsymbol{\kappa}}} W_{\beta,i}, \quad \text{where } W_{\beta,i} = H^1\left(\bP^1, \cO_{\bP^1}(\langle \D_i, \beta\rangle -\frac{q_i}{2})\right). 
\end{equation} 
Let $\chi^{\D_i}: G\to \bC^*$, where $1\leq i \leq n+{\boldsymbol{\kappa}}$, be defined as in Section \ref{sec:geometry}.
Let $G$ act on $V_{\beta,i}$ and $W_{\beta,i}$  by 
$g\cdot u = \chi^{\D_i}(g) u$ where $g\in G$ and $u\in V_{\beta,i}$ or $W_{\beta,i}$.

\begin{definition}[stacky loop space] \label{stacky-loop} 
We define the degree $\beta$ stacky loop space by
\begin{equation} 
\sX_{\zeta,\beta} := [V_\beta^{ss}(\zeta)/G].
\end{equation} 
\end{definition} 
The stacky loop space  in the above definition is the analogue of 
the stacky loop space in  orbifold quasimap theory \cite[Section 4.2]{CCK},
which can be viewed as the orbifold version of Givental's toric map space \cite[Section 5]{Gi98}. 

We have
$$
V_\beta^{ss}(\zeta) = \bigcup_{I\in\cA_\zeta^{\min} }  V_{\beta,I}
$$
where
$$
V_{\beta,I} = \{ (u_1,\ldots, u_{n+{\boldsymbol{\kappa}}})\in V_\beta\mid u_i \neq 0  \text{ if }i\in  I \}. 
$$

\begin{definition}[obstruction bundle and obstruction sheaf] 
We define the degree $\beta$ {\em obstruction bundle} by
$$
\Ob_{\zeta,\beta} = [ \left(V_\beta^{ss}(\zeta) \times W_\beta\right) /G].
$$
The degree $\beta$ {\em  obstruction sheaf} $\cOb_{\zeta,\beta}$ is the locally free sheaf of $\cO_{\sX_{\zeta,\beta}}$-modules on $\sX_{\zeta,\beta}$ associated to 
the vector bundle $\Ob_{\zeta,\beta}$. 
\end{definition} 
The obstruction bundle $\Ob_{\zeta,\beta}$ is a toric vector bundle over the smooth toric DM stack $\sX_{\zeta,\beta}$, and 
$$
\Ob_{\zeta,\beta} =\Spec\left(\Sym \cOb_{\zeta,\beta}^\vee\right).
$$
Let $\cT_{\sX_{\zeta,\beta}}$ be the tangent sheaf of $\sX_{\zeta,\beta}$.  The two-term complex 
\begin{equation}\label{eqn:tangent-obstruction}
\big[ \cT_{\sX_{\zeta,\beta}} \stackrel{0}{\lra} \cOb_{\zeta,\beta} \big]
\end{equation} 
is a perfect tangent-obstruction complex \cite{LT98} on $\sX_{\beta,\zeta}$. 
Taking the dual of \eqref{eqn:tangent-obstruction}, we obtain
\begin{equation} \label{eqn:prefect-obstruction}
\big[ \cOb^\vee_{\zeta,\beta} \stackrel{0}{\lra} \Omega_{\sX_{\zeta,\beta}} \big]
\end{equation} 
which is a perfect obstruction theory \cite{BF97} on $\sX_{\zeta,\beta}$.  In particular, the tangent-obstruction complex \eqref{eqn:tangent-obstruction}
and the perfect obstruction theory are objects in $D(\sX_{\zeta,\beta})$, the derived category of coherent sheaves on $\sX_{\zeta,\beta}$. The virtual tangent bundle is
$$
T^\vir_{\sX_{\zeta,\beta}} = T_{\sX_{\zeta,\beta}} - \Ob_{\sX_{\zeta,\beta}} \in K(\cX_{\zeta,\beta}). 
$$

Let $T_{\sX_{\zeta,\beta}/BG}$ be the relative tangent bundle
of the smooth map $\sX_{\zeta,\beta} \to BG$. We have the following short exact sequence of vector bundles on $\sX_{\zeta,\beta}$: 
\begin{equation}
0\to \cX_{\zeta,\beta} \times \bL_{\bC} \to T_{\sX_{\zeta,\beta}/BG} = [ \left(V_\beta^{ss}(\zeta)\times V_\beta\right) /G] \to T_{\sX_{\zeta,\beta}} \to 0. 
\end{equation} 

Given $I\in \cA_\zeta^{\min}$,  $\sigma_{I'}$ is an $n$-dimensional cone in $\Sigma$.  Define
\begin{eqnarray*} 
\bK^\eff_I(\sX,\zeta)_{0,1} & \longrightarrow  &  \mathrm{Box}(\sigma_{I'}) =\Big\{ \sum_{i\in I'} a_i v_i: a_i \in [0,1)\cap \bQ \Big\} \\
\beta  &\mapsto&  v(\beta):= \sum_{i \in I'}   \{ - \langle \D_i   ,\beta\rangle  + \frac{q_i}{2}  \} v_i.
\end{eqnarray*} 

The big torus  $\tT=(\bC^*)^{n+{\boldsymbol{\kappa}}}$ acts on $V_\beta$ by
$$
(\tit_1,\ldots, \tit_{n+{\boldsymbol{\kappa}}})\cdot (u_1,\ldots, u_{n+{\boldsymbol{\kappa}}}) = (\tit_1 u_1,\ldots, \tit_{n+{\boldsymbol{\kappa}}} u_{n+{\boldsymbol{\kappa}}}).
$$
This induces an action of $T=\tT/G \cong (\bC^*)^n$ (the flavor torus) on $\sX_{\zeta, \beta}$. The tangent sheaf
$\cT_{\cX_{\zeta,\beta}}$ and the obstruction sheaf $\cOb_{\zeta,\beta}$ are $T$-equivariant locally free sheaves $\sX_{\zeta,\beta}$, so the
the perfect obstruction theory is $T$-equivariant and is an object in $D_T(\sX_{\zeta,\beta})$, the derived category of $T$-equivariant
coherent sheaves on $\sX_{\zeta,\beta}$, and $T^\vir_{\sX_{\zeta,\beta}} \in K_T(\sX_{\zeta,\beta})$.





Define
$$
V_{\beta,I}^\circ =\{ (u_1,\ldots,u_{n+{\boldsymbol{\kappa}}}) \in V_{\beta} \mid u_i (1,0) \neq 0 \text{ if }i\in I\}.
$$
 Then $V_{\beta,I}^\circ$ is a Zariski open dense subset of $V_{\beta,I}$, and 
$$
V_\beta^{ss}(\zeta)^\circ :=\bigcup_{I\in \cA_\zeta^{\min}} V_{\beta,I}^\circ
$$
is a Zariski dense open subset of $V_\beta^{ss}(\zeta)$. Define
$$
\sX_{\zeta,\beta}^\circ := [V_\beta^{ss}(\zeta)^\circ/G]
$$
which is the open substack of $\sX_{\zeta,\beta}$. (Our notation $\sX_{\zeta,\beta}^\circ$ is motivated by Okounkov's notation in \cite{Ok20}, in which 
$\mathsf{QM}_\circ$ denotes the open subtack of $\mathsf{QM}$ where $\infty=[1,0]$ is not a base point.)    

Given $I\in \cA_\zeta^{\min}$, $\sigma_{I'}$ is a top-dimensional (i.e. $n$-dimensional) cone in $\Sigma$.  Recall that $g(v) \in G_I$ is the image of $v =\sum_{i\in I'} a_i v_i \in \mathrm{Box}(\sigma_{I'})$ under the
bijection $\mathrm{Box}(\sigma_{I'})  \to G_{\sigma_{I'}}$. The $g(v)$-fixed subspace of $V$ is
$$
V^{g(v)}  = \{ x= (x_1,\ldots, x_{n+{\boldsymbol{\kappa}}}) \in V \mid  x_i =0 \text{ if }  i\in I' \text{ and } a_i  \notin \bZ\}.
$$
The connected component $\sX_{\zeta,v}$ of the inertia stack $I\sX_{\zeta}$ associated to $v$ is
$\sX_{\zeta,v} =\left[ \big( V^{g(v)} \cap V^{ss}(\zeta) \big) /G \right]$ which is an open dense substack of the Artin stack $\fX_v= [V^{g(v)}/G]$.  
There is an evaluation map
$$
\ev_\infty: \cX_{\zeta, \beta}\lra \fX_{v(\beta)}  \quad [u_1,\ldots, u_{n+{\boldsymbol{\kappa}}}] \mapsto [u_1(1,0),\ldots, u_{n+{\boldsymbol{\kappa}}}(1,0)].
$$
Then
$$
\ev_\infty^{-1}(\sX_{\zeta,v(\beta)}) = \sX^\circ_{\zeta,\beta} 
$$

\subsection{Torus actions and $\bC_q^*$ fixed points}
In orbifold quasimap theory, the $I$-function is obtained by torus localization on the stacky loop space, using the
$\bC^*$-action on $\bP^1$, the coarse moduli space of $\bP[a,1]$ where $a$ is a positive integer.   We denote
this $\bC^*$ by $\bC^*_q$ since it corresponds to $\bC_q^\times$ in \cite{Ok20}. 
For $\tT$-equivariant parameters, we use notation similar to that in \cite{CIJ} and \cite{GiV}.
$$
K_{\bC_q^*}(\bullet) =  K(B\bC_q^*) = \bZ[q^{\pm 1} ],\quad K_{\tT}(\bullet) = K(B\tT) = \bZ[\Lambda_1^{\pm 1},\ldots, \Lambda_{n+{\boldsymbol{\kappa}}}^{\pm 1} ],
$$
Let
$$
z = c_1(q) \in  H^2_{\bC_q^*}(\bullet;\bZ),\quad \lambda_j = -c_1(\Lambda_j) \in H^2_{\tT}(\bullet;\bZ).
$$
Then 
$$
H^*_{\bC^*_q}(\bullet;\bZ) = H^*(B\bC_q^*;\bZ) =\bZ[z],\quad H^*_{\tT}(\bullet;\bZ) = H^*(B\tT;\bZ) =\bZ[\lambda_1,\ldots,\lambda_{n+{\boldsymbol{\kappa}}}],\quad
\tM =\bigoplus_{j=1}^{n+{\boldsymbol{\kappa}}} \bZ\lambda_j. 
$$

Let $\deg(x)=a$, $\deg(y)=1$. Then 
$$
\bC[x,y] =\bigoplus_{m=0}^\infty \bC[x,y]_m
$$
where $\bC[x,y]_m$ denote the degree $m$ part of the graded ring $\bC[x,y]$. If $m\in \bZ_{\geq 0}$ then
$$
H^0(\bP[a,1], \cO_{\bP^1[a,1]}(m)) = \bC[x,y]_m  = \bigoplus_{k=0}^{\lfloor \frac{m}{a}\rfloor} \bC x^k y^{m-ka},\quad
H^1(\bP[a,1],\cO_{\bP^1[a,1]}(m))=0.  
$$
Let $\bC^*_q$ act on $\bP[a,1]$ by $q\cdot[x,y] = [qx,y]= [x,q^{-1/a}y]$, and on  
$\bC[x,y]$ by $q\cdot x=x$, $q\cdot y = q^{-1/a}y$.  Given any number $r\in \bQ$, let
$\lfloor r \rfloor$ be the unique integer such that $\lfloor r \rfloor\leq r <  \lfloor r \rfloor +1$, and let
$\{ r\} := r - \lfloor r \rfloor \in [0,1)$.  As an element in $K_{\bC^*_q}(\bullet)$, 
\begin{eqnarray*}
&& H^0(\bP[a,1], \cO_{\bP^1[a,1]}(m))  - H^1(\bP[a,1], \cO_{\bP^1[a,1]}(m))   \\
&=&  \frac{q^{-\{  \frac{m}{a}\}}  }{1-q^{-1}} + \frac{q^{-\frac{m}{a}} }{1-q} 
=  \sum_{k=0}^\infty q^{-\{ \frac{m}{a}\} -k} -\sum_{k=0}^\infty q^{-\frac{m}{a}-1-k}\\
&=& 
\begin{cases}
\displaystyle{\sum_{k=0}^{\lfloor \frac{m}{a}\rfloor}  q^{-\{ \frac{m}{a}\}-k} }, & m\geq 0\\
0, &- a\leq m \leq -1 \\
\displaystyle{ -\sum_{k=1}^{-\lfloor \frac{m}{a}\rfloor -1} q^{-\{ \frac{m}{a}\} +k} }, & m \leq  -a-1
\end{cases}
\end{eqnarray*}

For $j=1,\ldots,n+{\boldsymbol{\kappa}}$,  we define
\begin{equation}\label{eqn:d-beta} 
d_j(\beta):= \langle \D_j,\beta \rangle-\frac{q_j}{2} \in \bQ
\end{equation}
Then 
 \begin{equation}\label{eqn;Vweights}
V^j_\beta = \begin{cases}
\displaystyle{   \sum_{k=0}^{\lfloor d_j(\beta) \rfloor}   \Lambda_j q^{-\{ d_j(\beta)  \} -k},  } & d_j (\beta) \geq 0;\\
  0, & d_j (\beta) <0
  \end{cases},\quad \quad
  W^j_\beta =\begin{cases}
  0, & d_j(\beta) \geq -1,\\
\displaystyle{  \sum_{k=1}^{-\lfloor d_j(\beta) \rfloor -1} \Lambda_j q^{-\{ d_j (\beta) \} +k},   } &  d_j (\beta) <-1.
  \end{cases}
\end{equation}
$$
(V_\beta^j)^{\bC^*_q} =\begin{cases}
\bC, &  d_j(\beta) \in \bZ_{\geq 0} ,\\
0, & \text{otherwise}.
\end{cases} 
$$

Write
$$
V =\bigoplus_{j=1}^{n+\boldsymbol{\kappa}} V_j= V_\beta^+ \oplus V_\beta^- \oplus V_\beta^\perp
$$
where $V_j =\Spec\bC[x_j]$, 
\begin{equation}
V_\beta^+ :=\bigoplus_{ \substack{j \in [1..n+\boldsymbol{\kappa}] \\ d_j(\beta)\in \bZ_{\geq 0} }} V_j, \quad
V_\beta^- :=\bigoplus_{ \substack{j \in [1..n+\boldsymbol{\kappa}] \\ d_j(\beta)\in \bZ_{< 0} }} V_j, \quad
V_\beta^\perp  :=\bigoplus_{ \substack{j \in [1..n+\boldsymbol{\kappa}] \\ d_j(\beta)\notin \bZ }}V_j. 
\end{equation} 
Then 
$$
V_\beta^+ = V_\beta^{\bC_q^*},\quad V_\beta^+ \oplus V_\beta^- = V^{g(v(\beta))}. 
$$
We define 
$$
\sF_\beta := \left[ \left. \left(V^+_\beta \cap V^{ss}_G(\zeta)\right) \right/G \right] .
$$
Then 
$\sF_\beta$ is a closed substack of $\sX_{\zeta,v(\beta)}$. 
$$
\ev_\infty: \sX_{\zeta, \beta}^\circ \to \sX_{\zeta,v(\beta)}
$$
restricts to an isomorphism 
$$
\ev_\infty: \left(\sX_{\zeta, \beta}^\circ \right)^{\bC^*_q} \to \sF_\beta.  
$$
We identify $\left(\sX_{\zeta, \beta}^\circ\right)^{\bC^*_q}$ with $\sF_\beta$ under the above isomorphism. 

If $I\in \cA_\zeta^{\min}$ and $\beta \in \bK^{\eff}(\sX_\zeta, \bw_\zeta)$  then 
the torus fixed point $\fp_I$  is contained in $\sF_\beta$ if and only of $\beta \in \bK^{\eff}_I$.

\subsection{Virtual tangent and normal bundles} 
Since $\bC_q^*$ acts trivially on $\sF_\beta$, it acts linearly on the fibers of any $\bC_q^*$-equivariant vector bundle  on $\sF_\beta$.  If $V$ is a $\bC^*_q$-equivariant vector bundle over $\sF_\beta$ then
$$
V = \bigoplus_{d\in \bZ}  V_d = V^f \oplus V^m,
$$
where $V_d$ is the subbundle on which $\bC_q^*$ acts by weight $d$, and $V^m =\displaystyle{ \bigoplus_{d\neq 0} V_d}$  (resp. $V^f=V_0$)  is the moving (resp. fixed) part of
$V$ under the $\bC^*_q$-action. 
Let
$$
T_\beta^1 :=  \left. T_{\sX_{\zeta,\beta}}\right|_{\sF_\beta},\quad T_\beta^2 :=  \left. \Ob_{\zeta,\beta}\right|_{\sF_\beta}. 
$$
Then $T_\beta^{1,f} = T_{\sF_\beta}$ is the tangent bundle of $\sF_\beta$, and $T_\beta^{1,m} =N_{\sF_\beta/\sX_{\zeta,\beta}^\circ} $ is the normal bundle
of $\sF_\beta$ in $\sX_{\zeta,\beta}$.

The virtual tangent bundle of $\sF_\beta$ is 
$$
T_\beta^{1,f} - T_\beta^{2,f} = T_{\sF_\beta} -0 = T_{\sF_\beta}.
$$
Therefore,
$$
[\sF_\beta]^\vir = [\sF_\beta]. 
$$
The virtual normal bundle of $\sF_\beta$ is defined to be
\begin{equation}\label{eqn:Nvir} 
N^\vir_\beta: = T_\beta^{1,m} - T_\beta^{2,m} \in K_{\bC_q^*\times \tT}(\sF_\beta). 
\end{equation}
Let 
$$
\iota_{\beta \to v(\beta)}: \sF_\beta \to  \sX_{\zeta,v(\beta)}, \quad \iota_v: \sX_{\zeta,v}\to  \sX_\zeta, 
\quad \iota_\beta = \iota_{v(\beta)}\circ \iota_{\beta\to v(\beta)} : \sF_\beta\to \sX_\zeta
$$
be inclusion maps. 
\begin{Proposition}
\begin{equation} \label{eqn:NtN} 
N_\beta^\vir = \iota_{\beta \to v(\beta)}^*\tN_\beta^\vir
\end{equation} 
where 
\begin{equation}\label{eqn:tNbeta} 
\tN^\vir_\beta =  \sum_{j=1}^{n+{\boldsymbol{\kappa}}} \iota_{v(\beta)}^* (U^{\tT}_j)^{-1} \Big( \sum_{k=0}^\infty q^{-d_j(\beta)+k}  -\sum_{k=0}^\infty q^{ \{-d_j(\beta)\} +k }\Big) 
+\sum_{d_j(\beta)\in\bZ_{<0}} \iota_{v(\beta)}^*(U_j^{\tT})^{-1} \quad \in K_{\bC^*_q\times \tT}\left(\sX_{\zeta,v(\beta)}\right).
\end{equation} 
\end{Proposition}
\begin{proof}
$$
T^{1,m}_\beta = \sum_{d_j(\beta)\geq 0} \iota_\beta^*(U^{\tT}_j)^{-1} \Big( \sum_{\substack{k\in \bZ\\ 0\leq k<d_j(\beta)} } q^{-d_j(\beta)+k} \Big)
=\sum_{d_j (\beta)\geq 0} \iota_\beta^* (U^{\tT}_j) ^{-1}\Big( \sum_{k=0}^\infty q^{-d_j(\beta)+k}  -\sum_{k=0}^\infty q^{ \{-d_j(\beta)\} +k }\Big). 
$$
\begin{eqnarray*}
T^{2,m}_\beta &=& \sum_{d_j(\beta)< 0}  \iota_\beta^* (U_j^{\tT})^{-1} \Big( \sum_{\substack{k\in \bZ\\  d_j(\beta)<  k<0} } q^{-d_j(\beta)+k} \Big) \\
&=& \sum_{d_j(\beta)<0} \iota_\beta^*(U_j^{\tT})^{-1} \Big( \sum_{k=0}^\infty q^{ \{-d_j(\beta)\} +k}  -\sum_{k=0}^\infty q^{ -d_j(\beta) +k }\Big)  
-\sum_{d_j(\beta) \in \bZ_{<0}} \iota_\beta^* (U_j^{\tT})^{-1} 
\end{eqnarray*}
Therefore,
\begin{equation}\label{eqn:Nbeta} 
N^\vir_\beta = T_\beta^{1,m}-T_\beta^{2,m} = \sum_{j=1}^{n+{\boldsymbol{\kappa}}} \iota_\beta^* (U^{\tT}_j)^{-1} \Big( \sum_{k=0}^\infty q^{-d_j(\beta)+k}  -\sum_{k=0}^\infty q^{ \{-d_j(\beta)\} +k }\Big) 
+ \sum_{d_j(\beta) \in \bZ_{<0}}  \iota_\beta^*(U_j^{\tT})^{-1}
\end{equation}
Note that $\iota_\beta^* = \iota_{\beta\to v(\beta)}^* \circ \iota_{v(\beta)}^*$, so 
\eqref{eqn:NtN} follows from \eqref{eqn:tNbeta} and \eqref{eqn:Nbeta}. 
 \end{proof} 
In the above proof, 
$$
\sum_{d_j(\beta) \in \bZ_{<0}}  \iota_\beta^*(U_j^{\tT})^{-1} = N_{\sF_\beta/\sX_{\zeta,v(\beta)}}. 
$$
We define
\begin{equation}
\tN_{\sF_\beta/\sX_{\zeta,v(\beta)} }  :=  \sum_{d_j(\beta) \in \bZ_{<0}}  \iota_{v(\beta)}^*(U_j^{\tT})^{-1}  \in K_{\tT}(\sX_{\zeta,v(\beta)}).
\end{equation} 
Then
\begin{equation}
\iota_{\beta\to v(\beta)}^* \tN_{\sF_\beta/\sX_{\zeta,v(\beta)}}  = N_{\sF_\beta/\sX_{\zeta,v(\beta)} }. 
\end{equation}
Recall that $V^{g(v(\beta))} = V_\beta^+\oplus V_\beta^-$. We have
$$
\tot\big( \tN_{\sF_\beta/\sX_{\zeta,v(\beta)}}\big) = \left[ \Big(  V^{g(v(\beta))}   \cap V_G^{ss}(\zeta)) \times V_\beta^-\Big)\big/G\right],\quad
\tot\big( N_{\sF_\beta/\sX_{\zeta,v(\beta)}}\big) = \left[ \Big( (V_\beta^+ \cap V_G^{ss}(\zeta)) \times V_\beta^-\Big)\big/G\right]. 
$$

\subsection{Virtual factorization and the singularity category} \label{sec:singular}
In this subsection, we fix an effective class $\beta$ and let $v=v(\beta)$.  Let $n_\pm = \dim V_\beta^\pm$. We introduce 
variables $(y_1,\ldots, y_{n_+})$ and $(p_1,\ldots, p_{n_-})$ such that 
$$
\{ y_1,\ldots, y_{n_+} \} =\{ x_i: d_i(\beta)\in \bZ_{\geq 0}\},\quad
\{ p_1,\ldots, p_{n_-} \} = \{ x_i: d_i(\beta) \in \bZ_{<0}\}. 
$$
Then
$$
V_\beta^+ = \Spec\bC[x_i: d_i(\beta)\in \bZ_{\geq 0}\} = \Spec\bC[y],\quad
V^{g(v)} = \Spec\bC[y,p]. 
$$
where
$$
\bC[y] =\bC[y_1,\ldots, y_{n_+}], \quad \bC[y,p] = \bC[y_1,\ldots, y_{n_+}, p_1,\ldots, p_{n_-}].
$$
Let
$$
W_v := W|_{V^{g(v)}} \in \bC[y,p]^G. 
$$

The inclusion $\iota_{\beta\to v}: \sF_\beta \to \sX_{\zeta,v}$ is a closed embedding, and can be identified with
the restriction of the evaluation map $\ev_\infty:\sX_{\zeta,\beta}^\circ \to \sX_{\zeta,v}$ to $(\sX_{\zeta,\beta}^\circ)^{\bC^*_q} \cong \sF_\beta$. 
Given any $(u_1,\ldots, u_{n+{\boldsymbol{\kappa}}})  \in V_\beta$,  $W(u_1,\ldots, u_{n+{\boldsymbol{\kappa}}}) \in H^0(\bP^1, \cO_{\bP^1}(-1)) =0$, so 
$\ev_\infty^*\bw_{\zeta,v}=0 \in H^0(\sX_{\zeta,\beta},\cO_{\sX_{\zeta,\beta}})$, which implies $\iota_{\beta\to v}^* \bw_{\zeta,v} = 0 \in  H^0(\sF_\beta, \cO_{\sF_\beta})$, which then implies
$W_v|_{V_\beta^+}=0$. Therefore,  $W_v$ is contained in the ideal $I= \langle p_1,\ldots, p_{n_-}\rangle \subset \bC[y,p]$, and can be written as
\begin{equation}\label{eqn:Wv}
W_v =\sum_{k=1}^{n_-} p_k W_k(y,p) 
\end{equation} 
for some $W_k(y,p) \in \bC[y,p]$ which are unique mod $I$.  We have
$$
\left.W_k(y,p)\right|_{p=0} = \left.\frac{\partial W_v}{\partial p_k}(y,p)\right|_{p=0}. 
$$
Recall that 
$$
V= \bigoplus_{i=1}^{n+\boldsymbol{\kappa}} V_i
$$
where $G$ acts on $V_i$ by character $\D_i$. Let $L_i$ denote the line bundle on $\fX_v =[V^{g(v)}/G]$ with total space
$$
\tot(L_i) = [ (V^{g(v)}\times V_i)/G].
$$
If $p_k=x_{i_k}$ then $p_k$ defines a section of $L_{i_k}$ and $W_k(y,p)$ defines a section of $L_{i_k}^{-1}$  (since $W_v$ is $G$-invariant).

We now apply the construction in Section \ref{sec:factorization} to 
$$
\bfX = \sF_\beta = [(V_\beta^+\cap V_G^{ss}(\zeta))/G], \quad
\bfC = [V_\beta^+/G],\quad  \fB^\circ = [\bullet/G]=BG.
$$
Let
$$
\bfU = \sX_{\zeta,v} = [V^{g(v)}\cap V_G^{ss}(\zeta)/G],\quad
\bfA = \fX_v =[V^{g(v)}/G]. 
$$
The closed embeddings $\iota_{\bfX}: \bfX \hookrightarrow \bfU$ and $\iota_{\bfC}:  \bfC\hookrightarrow \bfA$ are induced by the inclusion 
$V_\beta^+\subset V^{g(v)}$.  
$$
B_{\bfA}:= \pi_{\bfA/\fB^\circ}^*B = \bigoplus_{k=1}^{n_-} L_{i_k}, \quad B_{\bfU} = j_U^*B_{\bfA}=  \bigoplus_{k=1}^{n_-} j_U^* L_{i_k}
= \bigoplus_{k=1}^{n_-} \iota_v^* \cO_{\sX_\zeta}(\DD_{i_k}). 
$$
The surjective linear map 
$$
(p_1,\ldots, p_k): V^{g(v)}\to \bC^{n_-}
$$ 
descends to a regular section  $\beta_{\bfA} \in H^0(\bfA, B_{\bfA})$ which 
restricts to a regular section $\beta_{\bfU} \in H^0(\bfU, B_{\bfU})$, and 
$$
Z(\beta_{\bfA}) =\bfC,\quad Z(\beta_{\bfU})= \bfX. 
$$

The evaluation maps $\ev_{\bfA}:\bfA \to \fX_v$ and $\ev_{\bfU}:\bfU\to \sX_{\zeta,v}$ are identity maps, and 
$$
\bw_{\bfA} = \ev_{\bfA}^*\bw_v = \bw_v,\quad \bw_{\bfU} = \ev_{\bfU}^*\bw_{\zeta,v}= \bw_{\zeta,v}.
$$ 
The vector-valued polynomial function
$$
(-W_1,\ldots, -W_k): V^{g(v)}\to \bC^{n_-}
$$
descends to a section $\alpha_{\bfA} \in H^0(\bfA, B_{\bfA}^\vee)$
which restricts to  $\alpha_{\bfU} = j_{\bfU}^*\alpha_{\bfA} \in H^0(\bfU, B_{\bfU}^\vee)$.
$$
\langle \alpha_{\bfA}, \beta_{\bfA} \rangle = -\bw_{\bfA},\quad
\langle \alpha_{\bfU}, \beta_{\bfU}\rangle = -\bw_{\bfU}. 
$$ 
Then   $\{\alpha_{\bfA}, \beta_{\bfA}\}$ is a Koszul matrix factorization for $(\bfA, -\bw_{\bfA})$, and
\begin{equation}
\bK_\beta:= \{ \alpha_{\bfU},\beta_{\bfU}\} 
\end{equation}
is a Koszul matrix factorization for $(\bfU, -\bw_{\bfU}) = (\sX_{\zeta,v},  -\bw_{\zeta,v})$. 
Note that $\beta_{\bfU}$ is a regular section of $B_{\bfU}$, so $\{\alpha_{\bfU}, \beta_{\bfU}\}$ is a {\em regular} 
Koszul matrix factorization in the sense of \cite[Definition 1.6.1]{PV16} and (according to \cite[Section 1.6]{PV16}) should 
be viewed as a deformation of the Koszul complex $\{0,\beta_{\bfU}\}$. 

Let $\sX_{\zeta,v,0} = \bw_{\zeta,v}^{-1}(0)$, and let
$$
\fC : H\MF(\cX_{\zeta,v},\bw_{\zeta,v}) \to D_{\Sg}\left(\sX_{\zeta,v,0} \right)
$$
be the functor sending a matrix factorization $(E_\bullet,\delta)$ to the cokernel of $\delta_1: E_1\to E_0$.  There are two cases:
\begin{enumerate}[label=\roman*]
\item[(i)] If $\bw_{\zeta,v}$ is nonzero then it is not a zero divisor.  By \cite[Lemma 1.6.2 (i)]{PV16}, 
\begin{equation}
\fC(\{\alpha_{\bfU}, \beta_{\bfU}\}) \simeq \cO_{\sF_\beta}   \;  \textup{ in } D_{\Sg}\left(\sX_{\zeta,v,0}\right),
\end{equation} 
where $\cO_{\sF_\beta}$ is viewed as a coherent sheaf on $\sX_{\zeta,v,0}$.  By \cite[Theorem 3.14]{PV11}, 
the functor 
$$
\overline{\fC}:D\MF(\cX_{\zeta,v}, \bw_{\zeta,v}) \to D_{\Sg}\left(\sX_{\zeta,v,0}\right)
$$
induced by $\fC$ is an equivalence of triangulated categories. 
\item[(ii)]  If $\bw_{\zeta,v}=0$, then by \cite[Lemma 1.6.2 (ii)]{PV16}
\begin{equation} \label{eqn:KisG}
H^0\left(\{\alpha_{\bfU}, \beta_{\bfU}\}\right)\simeq i_{\beta\to v *}\cO_{\sF_\beta}, \quad
H^1\left(\{\alpha_{\bfU}, \beta_{\bfU}\}\right)=0. 
\end{equation} 
\end{enumerate}
 
\subsection{K-theoretic $I$-functions}  \label{sec:IK} 
For each $\beta\in \bK^{\eff}:= \bK^{\eff}(\sX_\eta,\bw_\zeta)_{0,1}$, the  derived pushforward
$$
R\iota_{\beta\to v(\beta) *} \cO_{\sF_\beta}  
$$
defines an element $G^{\tT}_\beta \in K_{\tT}(\sX_{\zeta, v(\beta)}) $ and $G_\beta \in K(\sX_{\zeta,v(\beta)})$. 
We have
\begin{equation} \label{eqn:GtT} 
G^{\tT}_\beta = \Lambda_{-1}(\tN_{\sF_\beta/\sX_{\zeta,v(\beta)}})^\vee  = \prod_{d_j(\beta)\in \bZ_{<0}}\left(1- \iota_v^*U_j^{\tT}\right) 
\;  \in  K_{\tT}\left(\sX_{\zeta, v(\beta)}\right)
\end{equation}
\begin{equation}
G_\beta = \prod_{d_j(\beta)\in \bZ_{<0}}\left(1- \iota_v^*U_j\right) \;  \in K(\sX_{\zeta,{v(\beta)}}). 
\end{equation}
\begin{remark} If $\sF_\beta$ is proper
then $R\iota_{\beta \to v(\beta)*}\cO_{\sF_\beta}$
defines an element $G_\beta^c \in K_c(\sX_{\zeta,v(\beta)})$ and $G_\beta$ lies in the ideal $K_{\ct}(\sX_{\zeta,v(\beta)}) \subset K(\sX_{\zeta,v(\beta)})$. 
\end{remark}

We define the {\em K-theoretic $\tT$-equivariant $I$-function} of the GLSM $(V,G,\bC_R^*, 0, \zeta)$  to be
$$
I^K_{\tT}= \sum_{v\in \mathrm{Box}(\bSi)} I^K_{\tT,v}
$$
where
\begin{equation}\label{eqn:IK-T} 
\begin{aligned}
I^K_{\tT,v}  :=&   \sum_{\substack {\beta \in \bK^\eff \\ v(\beta)=v} } y^\beta (R \iota_{\beta\to v(\beta)})_* \left( \frac{\cO_{\sF_\beta}}{\Lambda_{-1}(N^\vir_\beta)^\vee }  \right) 
  =   \sum_{\substack {\beta \in \bK^\eff \\ v(\beta)=v} } y^\beta (R \iota_{\beta\to v(\beta)})_* \left( \frac{\cO_{\sF_\beta}}{ \iota_{\beta\to v(\beta)}^*\Lambda_{-1}(\tN^\vir_\beta)^\vee }  \right) \\
 =&   \sum_{\substack {\beta \in \bK^\eff \\ v(\beta)=v} } y^\beta \frac{G^{\tT}_\beta}{\Lambda_{-1}(\tN^\vir_\beta)^\vee } 
 =  \sum_{\substack {\beta \in \bK^\eff \\ v(\beta)=v} } y^\beta \frac{\prod_{k=0}^\infty(1-\iota_v^*U_j^{\tT} q^{k +\{ -d_j(\beta) \} } ) }{\prod_{k=0}^\infty(1-\iota_v^*U_j^{\tT}q^{k -d_j(\beta)})}
\end{aligned}
\end{equation}
In Equation \eqref{eqn:IK-T} above, the second, third, and fourth equalities follow from \eqref{eqn:NtN}, \eqref{eqn:GtT}, and \eqref{eqn:tNbeta}, respectively. 

The non-equivariant limit of $I^K_{\tT}$ is 
$$
I^K =\sum_{v\in \mathrm{Box}(\bSi)} I^K_v
$$
where
$$
I^K_v  = \sum_{\substack {\beta \in \bK^\eff \\ v(\beta)=v} } y^\beta \frac{G_\beta}{\Lambda_{-1}(\tN^\vir_\beta)^\vee } 
= \sum_{\substack {\beta \in \bK^\eff \\ v(\beta)=v} } y^\beta \prod_{j=1}^{n+{\boldsymbol{\kappa}}} \frac{\prod_{k=0}^\infty(1-\iota_v^*U_j q^{k +\{ -d_j(\beta) \} } ) }{\prod_{k=0}^\infty(1-\iota_v^*U_j q^{k -d_j(\beta)})}.
$$

\begin{remark} \label{IK-compact}
If $\sF_\beta$ is proper for all $\beta\in \bK^\eff$ then we may define $I_c^K$ which takes values in
$\displaystyle{\bigoplus_{v\in \mathrm{Box}(\bSi)} K_c(\cX_{\zeta,v})} :$
$$
I^K_c =\sum_{v\in \mathrm{Box}(\bSi)} I^K_{c,v}
\quad  \text{where} \quad
I^K_{c,v}  := \sum_{\substack {\beta \in \bK^\eff \\ v(\beta)=v} } y^\beta \frac{G^c_\beta}{\Lambda_{-1}(\tN^\vir_\beta)^\vee }. 
$$
\end{remark}

 For every $\beta\in K^\eff$, let
 $$
 [\bK_\beta] \in K\left(\MF(\sX_{\zeta,v},\bw_{\zeta,v})\right)
$$
be the K-theory class of the Koszul matrix factorization $\bK_\beta$ defined in Section \ref{sec:singular}. 
Then $G_\beta$ is the image of $[\bK_\beta]$ under $K\left(\MF(\sX_{\zeta,v},\bw_{\zeta,v})\right) \lra K(\sX_{\zeta,v})$. 
We define the {\em K-theoretic GLSM $I$-function} of the GLSM $(V,G, \bC_R^*,W,\zeta)$ to be
$$
I^K_{\bw} =\sum_{v\in \mathrm{Box}(\bSi)} I^K_{\bw,v}
\quad \text{where} \quad
I^K_{\bw,v}  : = \sum_{\substack {\beta \in \bK^\eff \\ v(\beta)=v} } y^\beta \frac{[\bK_\beta]}{\Lambda_{-1}(\tN^\vir_\beta)^\vee }.
$$

\subsection{Cohomological $I$-functions} \label{sec:IH} 
We have a proper pushforward 
\begin{equation}\label{eqn:pushforward-HT} 
(\iota_{\beta\to v(\beta)})_* : H^*_{\tT}(\sF_\beta) \to  H^*_{\tT}(\sX_{\zeta,v(\beta)})
\end{equation}
which is a morphism of $H_{\tT}^*\left(\sX_{\zeta,v(\beta)} \right)$-modules:  for any $a\in H^*_{\tT}\left(\sX_{\zeta,v(\beta)}\right)$ and 
$b\in H^*_{\tT}(\sF_\beta)$, we have
$$
(\iota_{\beta\to v(\beta)})_* \left( \iota_{\beta\to v(\beta)}^*a \cup b\right) =  a \cup (\iota_{\beta\to v(\beta)})_* b. 
$$
The image is the principal ideal generated by the $\tT$-equivariant Poincar\'{e} dual $F_\beta^{\tT}$ of $\sF_\beta$ in $\sX_{\zeta,v(\beta)}$, where 
\begin{equation}\label{eqn:FtT}
F^{\tT}_\beta = (\iota_{\beta \to v(\beta)})_* 1 = e_{\tT}(\tN_{\sF_\beta/\sX_{\zeta,v(\beta)}})  = \prod_{d_j(\beta) \in \bZ_{<0}}  \iota_{v(\beta)}^*u_j^{\tT}  \in H^*_{\tT} (\sX_{\zeta,v(\beta)}). 
\end{equation} 
The non-equivariant Poincar\'{e} dual of $\sF_\beta$ is 
\begin{equation}
F_\beta = \prod_{d_j(\beta) \in \bZ_{<0}} \iota_{v(\beta)}^*u_j  \in H^*\left(\sX_{\zeta, v(\beta)}\right).  
\end{equation} 
\begin{remark}
If $\sF_\beta$ is proper then it also has a non-equivariant Poincar\'{e} dual $F_\beta^c \in H^*_c(\sX_{\zeta,v(\beta)})$, and
$F_\beta$ is the image of $F_\beta^c$ under $H^*_c(\sX_{\zeta, \beta}) \to H^*(\sX_{\zeta,v(\beta)})$; in particular, $F_\beta$ lies
in the ideal  $H^*_{\ct}(\sX_{\zeta, v(\beta)})\subset H^*(\sX_{\zeta,v(\beta)})$. 
\end{remark}

We define the (cohomological) {\em $\tT$-equivariant $I$-function} of the GLSM $(V,G,\bC^*_R, W,\zeta)$ to be
$$
I_{\tT}(y,z)= \sum_{v\in \mathrm{Box}(\bSi)} I_{\tT,v}(y,z) \one_v
$$
where
\begin{equation} \label{eqn:I-T} 
\begin{aligned}
I_{\tT,v}(y,z) :=&  e^{(\sum_{a=1}^{\boldsymbol{\kappa}} (\log y_a) \iota_v^* p_a^{\tT})/z}\sum_{\substack {\beta \in \bK^\eff \\ v(\beta)=v} } y^\beta 
\left(\iota_{\beta\to v(\beta)}\right)_* \left( \frac{1}{e_{\tT\times \bC_q^*}(N_\beta^\vir)}  \right)\\
=&   e^{(\sum_{a=1}^{\boldsymbol{\kappa}} (\log y_a) \iota_v^* p_a^{\tT})/z}\sum_{\substack {\beta \in \bK^\eff \\ v(\beta)=v} } y^\beta 
\left(\iota_{\beta\to v(\beta)}\right)_* \left( \frac{1}{ \iota_{\beta\to v(\beta)}^*e_{\tT\times \bC_q^*}(\tN_\beta^\vir)}  \right)\\
=& e^{(\sum_{a=1}^{\boldsymbol{\kappa}} (\log y_a) \iota_v^*p_a^{\tT})/z}\sum_{\substack {\beta \in \bK^\eff \\ v(\beta)=v} } y^\beta 
\frac{F_\beta^{\tT}}{e_{\tT\times \bC_q^*}(\tN_\beta^\vir)} \\ 
=&  e^{(\sum_{a=1}^{\boldsymbol{\kappa}} (\log y_a) \iota_v^*p_a^{\tT})/z}\sum_{\substack {\beta \in \bK^\eff \\ v(\beta)=v} } y^\beta  \prod_{j=1}^{n+{\boldsymbol{\kappa}}} 
\frac{\prod_{k=0}^\infty(\iota_v^*u_j^{\tT}  + (- \{-\langle \D_j, \beta\rangle +q_j/2\} - k)z ) }{\prod_{k=0}^\infty(\iota_v^* u_j^{\tT} + (\langle \D_j,\beta)-q_j/2 -k) z)} \\
=& \frac{e^{\left(\sum_{a=1}^{\boldsymbol{\kappa}} (\log y_a) \iota_v^*p_a^{\tT}\right)/z} }{z^{\age(v)- \hat{q}}} \sum_{\substack {\beta \in \bK^\eff \\ v(\beta)=v} } 
\frac{y^\beta}{z^{\langle \sum_{j=1}^{n+{\boldsymbol{\kappa}}} \D_j, \beta\rangle}}  \prod_{j=1}^{n+{\boldsymbol{\kappa}}}
\frac{\Gamma\left(1+\frac{\iota_v^*u_j^{\tT}}{z} - \{ -\langle \D_j,\beta\rangle +\frac{q_j}{2} \} \right)}{ \Gamma\left(1+\frac{\iota_v^*u_j^{\tT}}{z} + \langle \D_j,\beta\rangle -\frac{q_j}{2}\right) }
\end{aligned} 
\end{equation} 
In Equation \eqref{eqn:I-T} above, the second, third, and fourth equalities follow from \eqref{eqn:NtN}, \eqref{eqn:FtT}, and \eqref{eqn:tNbeta}, respectively.  

Recall that
$$
\deg u_j^{\tT} = \deg p_a^{\tT} = \deg z=2. 
$$
We define
$$
\deg y^\beta = 2 \Big\langle \sum_{j=1}^{n+{\boldsymbol{\kappa}}} \D_j, \beta \Big\rangle, \quad \deg (\one_v)  = 2 (\age(v) - \hat{q}).
$$
Then $I_{\tT}(y,z)$ is homogeneous of degree zero.  Setting $q_j=0$ recovers the formula of (extended) $I$-functions in 
\cite[Section 5]{CCIT} and \cite[Section 6.2.1]{CIJ}.  

The non-equivariant limit of $I_{\tT}(y,z)$ is 
$$
I(y,z)= \sum_{v\in \mathrm{Box}(\bSi)} I_v(y,z) \one_v
$$
where
\begin{equation}\label{eqn:I}
\begin{aligned}
I_v(y,z)= & e^{\left(\sum_{a=1}^{\boldsymbol{\kappa}} (\log y_a) \iota_v^* p_a\right) /z}\sum_{\substack {\beta \in \bK^\eff \\ v(\beta)=v} } y^\beta 
\left(\iota_{\beta\to v(\beta)}\right)_*\frac{1}{e_{\bC_q^*}(N_\beta^\vir)}  \\
=& e^{\left(\sum_{a=1}^{\boldsymbol{\kappa}} (\log y_a) \iota_v^*p_a \right) /z}\sum_{\substack {\beta \in \bK^\eff \\ v(\beta)=v} } y^\beta 
\frac{F_\beta}{e_{\bC_q^*}(\tN_\beta^\vir)} \\
=&  \frac{e^{\left(\sum_{a=1}^{\boldsymbol{\kappa}} (\log y_a) \iota_v^*p_a \right)/z} }{z^{\age(v)- \hat{q}}} \sum_{\substack {\beta \in \bK^\eff \\ v(\beta)=v} } 
\frac{y^\beta}{z^{\langle \sum_{j=1}^{n+{\boldsymbol{\kappa}}} \D_j, \beta\rangle}} 
 \prod_{j=1}^{n+{\boldsymbol{\kappa}}}
\frac{\Gamma\left(1+\frac{\iota_v^*u_j}{z} - \{ -\langle \D_j,\beta\rangle +\frac{q_j}{2} \} \right)}{ \Gamma\left(1+\frac{\iota_v^*u_j}{z} + \langle \D_j,\beta\rangle -\frac{q_j}{2}\right) }
\end{aligned} 
\end{equation} 

For every $\beta\in \bK^\eff$, 
$$
\td\ch^{\sX_{\zeta,v}}_{\sZ_{\zeta,v}}(\bK_\beta) \in 
\bH^{even}_{\sZ_{\zeta,v}}\left(\sX_{\zeta,v}, (\Omega^\bullet_{\sX_{\zeta,v}}, -d\bw_{\zeta,v}) \right)
$$
where $\sZ_{\zeta,v} = \Crit(\bw_{\zeta,v})$.   We define the {\em GLSM $I$-function} of the GLSM $(V,G,\bC_R^*, W, \zeta)$ to be
$$
I_{\bw}(y,z)= \sum_{v\in \mathrm{Box}(\bSi)} I_{\bw,v}(y,z)
$$
where
\begin{equation}
I_{\bw,v}(y,z) :=  e^{\left(\sum_{a=1}^{\boldsymbol{\kappa}} (\log y_a) \iota_v^*p_a\right) /z}\sum_{\substack {\beta \in \bK^\eff \\ v(\beta)=v} } y^\beta 
\frac{1}{e_{\bC_q^*}(\tN_\beta^\vir)}\cdot  \td\ch^{\sX_{\zeta,v} }_{ \sZ_{\zeta,v} }(\bK_\beta). 
\end{equation}

\begin{lemma}
If $\bw_{\zeta,v}=0$ (which is true when $v$ is narrow) then 
$$
I_{\bw,v}(y,z) = I_v(y,z). 
$$
\end{lemma}
\begin{proof}
If $\bw_{\zeta,v}=0$  then  $\sZ_{\zeta,v}=\sX_{\zeta,v}$ and 
$$
\bH^{even}_{\sZ_{\zeta,v}}\left(\sX_{\zeta,v}, (\Omega^\bullet_{\sX_{\zeta,v}}, -d\bw_{\zeta,v})\right) = 
\bH^{even}_{\sX_{\zeta,v}}\left(\sX_{\zeta,v}, (\Omega^\bullet_{\sX_{\zeta,v}}, 0)\right) = 
H^*(\sX_{\zeta,v}),
\quad  \td\ch^{\sX_{\zeta,v}}_{\sZ_{\zeta,v}}(\bK_\beta) = F_\beta. 
$$
\end{proof} 

\begin{remark}\label{I-compact} 
If $\sF_\beta$ is proper for all $\beta\in \bK^\eff$ then we may define $I_c(y,z)$ which takes values in 
$\displaystyle{ \bigoplus_{v\in \mathrm{Box}(\bSi)}H_c^*(\sX_{\zeta,v}) }:$
$$
I_c(y,z)= \sum_{v\in \mathrm{Box}(\bSi)} I_{c,v}(y,z) \one_v
$$
where
\begin{equation}\label{eqn:I}
I_{c,v}(y,z) :=  e^{\left(\sum_{a=1}^{\boldsymbol{\kappa}} (\log y_a) \iota_v^*p_a\right) /z}\sum_{\substack {\beta \in \bK^\eff \\ v(\beta)=v} } y^\beta 
\frac{F^c_\beta}{e_{\bC_q^*}(\tN_\beta^\vir)}.
\end{equation} 
\end{remark}

\subsection{Central charges}
Recall that $H^*_{\tT}(\bullet;\bC) = H^*(B\tT;\bC) =\bC[\lambda_1,\ldots, \lambda_{n+\boldsymbol{\kappa}} ] =: \bC[\lambda]$. 
Following Fang \cite[Section 4.2]{Fa20}, we define a map 
\begin{eqnarray*}
\tch_z: K_{\tT}(\sX_\zeta)   &\lra &  \bigoplus_{v\in \mathrm{Box}(\bSi_\zeta)}  H^*_{\tT}(\sX_{\zeta,v};\bC) \otimes_{\bC[\lambda]} \bC[\lambda] (( z^{-1})) \one_v, \\
\cE & \mapsto& \sum_{v\in \mathrm{Box}(\bSi_\zeta)} \tch_z(\cE)_v \one_v,  
\end{eqnarray*}
by the following two properties which uniquely characterizes it. 
\begin{enumerate}
\item If $\cE_1$, $\cE_2$ are $\tT$-equivariant vector bundles on $\sX_\zeta$ then
$$
\tch^{\tT}_z(\cE_1\oplus \cE_2)_v = \tch^{\tT}_z(\cE_1)_v +\tch^{\tT}_z(\cE_2)_v,\quad
\tch^{\tT}_z(\cE_1\otimes \cE_2)_v = \tch^{\tT}_z(\cE_1)_v \tch^{\tT}_z(\cE_2)_v. 
$$
\item If $\cL$ is a $\tT$-equivariant line bundle on $\sX_\zeta$ then 
$$
\tch^{\tT}_z(\cL)_v = e^{2\pi\sqrt{-1}(\age_v(\cL)- \frac{(c_1)_{\tT}( \iota_v^*\cL)}{z})} 
$$
\end{enumerate}
\begin{remark} Recall that $\deg(z)=2$, so $\tch^{\tT}_z(\cL)_v$ is homogeneous of degree 0 and is determined
by $\tch^{\tT}_1(\cL)_v = \tch^{\tT}_z(\cL)_v\Big|_{z=1}$: 
$$
\tch^{\tT}_z(\cL)_v =  z^{-\deg/2} \tch^{\tT}_1(\cL). 
$$
\end{remark}

We define
$$
\tGa^{\tT}_z = \sum_{v\in \mathrm{Box}(\bSi)}  \prod_{j=1}^{n+{\boldsymbol{\kappa}}}  \exp\Big( \log z(1+\frac{ \iota_v^*u_j^{\tT}}{z} -\age_v(U_j^{-1})  \Big)
\Gamma\Big( 1+ \frac{ \iota_v^*u_j^{\tT} }{z}   - \age_v (U_j^{-1}) \Big) \one_v
$$

Recall that the  $\tT$-equivariant state space $\cH_{\tT}$ of $(V,G,\bC^*_R,0,\zeta)$ is
$$
\cH_{\tT} =\bigoplus_{v\in \mathrm{Box}(\bSi_\zeta)} H^*_{\tT}(\sX_{\zeta,v};\bC)[2 \left(\age(v)-\hq \right) ]
$$
as a graded vector space over $\bC$, where each $H^*_{\tT}(\sX_{\zeta,v};\bC)$ is a $\bC[\lambda]$-module. 
Let $( \ , \  )_{\tT}$ denote the non-degenerate pairing
$$
\cH_{\tT}\otimes_{\bC[\lambda]} \bC(\lambda)  \times \cH_{\tT}\otimes_{\bC[\lambda]} \bC(\lambda) \lra \bC(\lambda). 
$$

\begin{definition}[$\tT$-equivariant central charge]
Given $\fB\in D^b_{\tT}(\sX_\zeta)$, we define its  $\tT$-equivariant central charge by 
\begin{equation}
Z_{\tT}(\fB) := \left(   I_{\tT}(y,-z), \tGa_z^{\tT} \tch_z^{\tT}([\fB])\right)_{\tT} 
\end{equation}   
where $[\fB]\in K_{\tT}(\sX_\zeta)$ is the K-theory class of $\fB$. 
\end{definition} 

By definition, $Z_{\tT}(\fB)$ depends only on $[\fB]\in K_{\tT}(\sX_\zeta)$. We use the notation in Section \ref{sec:line-bundles}  and  make the following observations:
\begin{itemize} 
\item Any class in $K_{\tT}(\sX_\zeta)$ can be written as
$$
\sum_{ \St \in \bL^\vee}  c_{\St} \bfL_{\St}^{\tT}
$$
where $c_{\St} \in K_{\tT}(\bullet) = \bZ[\Lambda_1^{\pm 1},\ldots, \Lambda_{n+\boldsymbol{\kappa}}^{\pm 1} ]$, and 
$$
\bfL_{\St}^{\tT} =\prod_{a=1}^{{\boldsymbol{\kappa}}} (P^{\tT}_a)^{-\langle \St,  \xi_a\rangle}  \in \Pic_{\tT}(\cX_{\zeta})
$$
is characterized by 
$$
(c_1)_{\tT}\big(\bfL_{\St}^{\tT}\big) = \sum_{a=1}^{{\boldsymbol{\kappa}}} \langle \St, \xi_a \rangle p_a^{\tT}. 
$$
\item The map $\fB \lra Z_{\tT}(\fB)$ is $K_{\tT}(\bullet)$-linear. 
\end{itemize} 
Therefore, it suffices to compute $Z_{\tT}(\bfL^{\tT}_{\St})$ for all $\St\in \bL^\vee$. For $j\in \{1,\ldots, n+\boldsymbol{\kappa} \}$, we let 
\begin{equation}\label{eqn:alpha-lambda}
\alpha_j =  \frac{\lambda_j}{z}+ \frac{q_j}{2}. 
\end{equation}
\begin{Theorem}
For any $\St \in \bL^\vee$, we have
$$
Z_{\tT}(\bfL^{\tT}_{\St} ) = \sum_{I\in\cA_\zeta^{\min}} Z_{\tT}(\bfL^{\tT}_{\St} )_I,
$$ 
where
\begin{eqnarray*}
Z_{\tT}(\bfL^{\tT}_{\St} )_I &=& \frac{1}{ |G_{\sigma_{I'}}|} \sum_{m\in (\bZ_{\geq 0})^I}  \prod_{i' \in I' }\frac{\Gamma\left( \langle \D_{i'}, - \sum_{i\in I} (m_i +\alpha_i)\D_i^{*I}\rangle 
+ \alpha_{i'}   \right) } {z^{\langle \D_{ i'},  \sum_{i\in I} (m_i +\alpha_i)\D_i^{*I}\rangle - \alpha_{i'}  }}  \prod_{i\in I}\frac{1}{(-z)^{m_i} m_i!} \\
&& \cdot \exp(\langle  (\sum_{a=1}^{\boldsymbol{\kappa}} \log y_a \xi_a)  -2\pi\sqrt{-1} \St, \sum_{i\in I} (m_i +\alpha_i) D_i^{* I}\rangle) 
\end{eqnarray*} 
\end{Theorem}
\begin{proof} Let
$$
\Xi =\{ (I,v): I\in \cA^{\min}_\zeta, v\in \mathrm{Box}(\sigma_{I'})\} =\{ (I,v): I\in \cA^{\min}_\zeta, \fp_I\subset \sX_{\zeta,v}\}.
$$
Given a  pair $(I,v)\in \Xi$, let $\iota_{I,v}: \fp_I\hookrightarrow \sX_{\zeta,v}$ be the inclusion.  
Then 
$$ 
Z_{\tT}(\bfL^{\tT}_{\St}) =\sum_{I\in \cA_\zeta^{\min}} A_I
$$
where 
\begin{eqnarray*}
A_I &=& \frac{1}{|G_{\sigma_I}|} \sum_{v\in \mathrm{Box}(\sigma_{I'}) }
\frac{  \iota_{I,v}^*I_{\tT,v} \iota_{I,\inv(v)}^*\iota^*_{\inv(v)} \left(\tGa_z \iota_{I,\inv(v)}\tch_z(\bfL^{\tT}_{\St})\right)  }
{e_{\tT}(T_{\fp_I}\cX_v)} \\
&=&  \frac{1}{ |G_{\sigma_{I'}}|} \sum_{m\in (\bZ_{\geq 0})^I}  \prod_{i' \in I' }
\frac{\Gamma\left( \langle \D_{i'}, - \sum_{i\in I} (m_i +\alpha_i)\D_i^{*I}\rangle 
+ \alpha_{i'}   \right) } {z^{\langle \D_{ i'},  \sum_{i\in I} (m_i +\alpha_i)\D_i^{*I}\rangle - \alpha_{i'}  }}  \prod_{i\in I}\frac{1}{(-z)^{m_i} m_i!} \\
&& \cdot \exp(\langle  (\sum_{a=1}^{\boldsymbol{\kappa}} \log y_a \xi_a)  -2\pi\sqrt{-1} \St, \sum_{i\in I} (m_i +\alpha_i) D_i^{* I}\rangle) 
\end{eqnarray*}

\end{proof}

Suppose that $\sum_{i=1}^{n+\boldsymbol{\kappa}} \D_i =0$ (Calabi-Yau condition), and let $C$ be the maximal
cone in the secondary fan which contains the stability condition $\zeta$. Then 
$$
Z^{\tT}(\bfL_{\St}^{\tT})\Big|_{z=1} = Z_{D^2}(\bfL_{\St})_C \Big|_{\substack{ \theta =-\sum_{a=1}^{\boldsymbol{\kappa}}\log y_a \xi_a\\ 
\alpha_j = \lambda_j+q_j/2}  }
$$
where $Z_{D^2}(\bfL_{\St})_C$ is the chamber hemisphere partition function  in Definition  \ref{chamber-hemisphere}.


Recall that the GLSM state space $\cH_{\bw}$ of $(V,G, \bC_R^*, W, \zeta)$ is 
$$
\cH_{\bw} = \bigoplus_{v\in \mathrm{Box}(\bSi_\zeta)} \bH^*\left(\sX_{\zeta,v}, (\Omega^\bullet_{\sX_{\zeta,v}}, d\bw_{\zeta,v})\right)[2(\age(v)-\hq)]
$$
as a graded vector space over $\bC$, where
$$
\bH^*\left(\sX_{\zeta,v}, (\Omega^\bullet_{\sX_{\zeta,v}}, d\bw_{\zeta,v})\right) \simeq
H^*(\sX_{\zeta,v}, \bw_{\zeta,v}^\infty;\bC).
$$
Let $( \ , \  )_{\bw}$ denote the non-degenerate pairing
$\cH_{\bw}\otimes \cH_{\bw} \to \bC$ in Section \ref{sec:state}
 
 Let $\tGa$ be the non-equivariant limit of $\tGa^{\tT}_z\big|_{z=1}$, i.e., 
 \begin{equation}
 \tGa = \sum_{v\in \mathrm{Box}(\bSi)}  \prod_{j=1}^{n+{\boldsymbol{\kappa}}} 
\Gamma\Big( 1+  \iota_v^*u_j     - \age_v (U_j^{-1}) \Big) \one_v
 \end{equation} 

\begin{definition}[GLSM central charge]
Given $\fB\in D\MF(\sX_\zeta,\bw_\zeta)$, we define its GLSM central charge by 
\begin{equation}
Z_{\bw} (\fB) := \left(   I_{\bw}(y,-1), \tGa \tch_{\bw}([\fB])\right)_{\bw} 
\end{equation}   
where $[\fB]\in K(D\MF(\sX_\eta, \bw_\zeta))$ is the K-theory class of $\fB$, and $\tch_{\bw}([\fB])$ is defined
in \cite{CKS}. 

\end{definition}

\section{The Coulomb Branch} \label{sec:Coulomb} 

\subsection{Disk partition function}

In this section we introduce the disk partition function following~\cite{HR} as an
explicit contour integral associated to a GLSM and a brane (element of $K(\MF([V/G],\bw))$) in
a given phase. Physically this integral is a Coulomb branch representation of a disk partition
function. We show that this integral coincides with the D-brane central charge introduced in the
sections earlier establishing the Higgs-Coulomb correspondence.

\paragraph{\bf Integral density}

We work in the notation of Section~\ref{sec:abelian}. Recall that we have the exact sequence of tori
\begin{equation}
  \label{eq:1}
  1 \to G \to \tT \to T \to 1,
\end{equation}
where $\tilde{T} \simeq (\bC^{*})^{n+{\boldsymbol{\kappa}}} \subset V \simeq \bC^{n+{\boldsymbol{\kappa}}}$ and the action of $G \simeq (\bC^{*})^{{\boldsymbol{\kappa}}}$ is given by the charge vectors
$\D_{i} \in \bL^{\vee}, \; 1 \le i \le n+{\boldsymbol{\kappa}}$. Indeed, it suffices to assume that the charge vectors 
span $\bL^{\vee}_{\bQ}$ over $\bQ$ which is the same as
saying that the image of $\rho_V: G\to \tT$ is ${\boldsymbol{\kappa}}$-dimensional, or equivalently,
the kernel of $\rho_V$ is finite. 

Let $\fg$ be the Lie algebra of $G$ and let $\fg^\vee$ be the dual of $\fg$. Then
$\fg=\bL_\bC$ and $\fg^\vee=\bL^\vee_\bC$. . 
Let $\theta \in \fg^\vee$ be a complexified stability parameter (complexified K\"{a}hler variable) and 
$\sigma =\sum_{a=1}^{{\boldsymbol{\kappa}}}\sigma_a \xi_a \in\fg$. 
Then $H^*_G(V)$ can be identified with the ring of algebraic functions on $\fg$: 
$H^*_G(V)\simeq H^*_G(\bullet) = \bC[\sigma_1,\ldots,\sigma_{\boldsymbol{\kappa}}]$  where $\sigma_1,\ldots, \sigma_{\boldsymbol{\kappa}} \in H^2_G(V)$ are equivariant
variables of the torus $G$. We have
$$
\fg=\Spec\,\bC[\sigma_1,\ldots, \sigma_{\boldsymbol{\kappa}}],\quad
G =\Spec\,\bC[s_1^{\pm 1},\ldots, s_{\boldsymbol{\kappa}}^{\pm 1} ], \quad
\text{where  }s_a = e^{2\pi\sqrt{-1}\sigma_a}.
$$
Let $H^*_{G,an}(V)$ denote the ring of meromorphic functions on $\fg$. 
We define $\Gamma \in H^{*}_{G, an}(V)$ by the following formula
\begin{equation}
  \Gamma  : = \prod_{i=1}^{n+{\boldsymbol{\kappa}}}\Gamma(\langle \D_{i}, \sigma \rangle + \alpha_{i})
  =\prod_{i=1}^{n+{\boldsymbol{\kappa}}} \Gamma(\sum_{a=1}^{\boldsymbol{\kappa}} Q_i^a \sigma_a +\alpha_i)
\end{equation}
where $Q_i^a =\langle \D_i, \xi_a\rangle\in \bZ$ is defined as in Section \ref{sec:abelian}. 
For fixed $\alpha_i\in \bC$,  $\Gamma$ is a meromorphic function in $\sigma_{1}, \ldots, \sigma_{{\boldsymbol{\kappa}}}$ that has poles along 
$\langle \D_{i}, \sigma \rangle + \alpha_{i} \in \ZZ_{\le 0}$. Each pole divisor is a hyperplane in $\fg$. Whenever
$\{ \D_{i} \}_{i \in I}$ form a basis of $\fg^\vee$ the corresponding polar hyperplanes form simple normal
crossing divisors. As we show below, these simple normal crossing divisors correspond to effective classes of curves
for $[V \sslash_{\zeta}G]$ for the particular choice of the stability parameter $\zeta$.

The variable $\alpha = (\alpha_{1}, \ldots, \alpha_{n+{\boldsymbol{\kappa}}})$ are related to the equivariant variable
$\lambda= (\lambda_1,\ldots, \lambda_{n+{\boldsymbol{\kappa}}})$ 
of the big torus $\tilde{T} \simeq (\bC^{*})^{n+{\boldsymbol{\kappa}}}$ by $\alpha_i = \lambda_i + q_i/2$. So far there is no superpotential in the story; moreover,
only monomials in $x_1,\ldots,x_{n+\boldsymbol{\kappa} }$ are  equivariant with respect to the  $\tilde{T}$-action on $V=\Spec\bC[x_1,\ldots, x_{n+\boldsymbol{\kappa}}]$. 
In order to account for the superpotential we first need to take the limit $\alpha_{i} \to q_{i}/2$
and then pair the result with the elements of the GLSM branes $K(\MF([V/G], \bw))$.
We note that the other limit $\alpha_{i} \to 0$ corresponds to the non-equivariant theory without superpotential (quasimaps to GIT quotients
$[V \sslash_{\zeta}G]$).

Let $\fB \in K(\MF([V/G], \bw))$. We can compose the forgetful morphism $K(\MF([V/G], \bw)) \to
K([V/G]) \simeq K_{G}(V)$ with the equivariant Chern character map 
$$
\ch \; : \; K_{G}(V) \to H^{*}_{G,an}(V), \quad \sum_{\St\in \bL^\vee} c_{\St}\bfL_{\St}
\mapsto \sum_{\St\in \bL^\vee} c_{\St}  e^{2\pi\sqrt{-1}\langle \St,\sigma\rangle}
$$
where $c_{\St}\in \bZ$ and the sum is finite. We denote the resulting map as just $\ch$.
Note that $e^{2\pi\sqrt{-1}\langle \St,\sigma\rangle} = \prod_{a=1}^{\boldsymbol{\kappa}} s_a^{\langle \St,\xi_a\rangle}$, 
where $\langle \St, \xi_a\rangle\in \bZ$, so the image of $\ch$ is the subring $\bZ[s_1^{\pm},\ldots, s_{\boldsymbol{\kappa}}^{\pm} ] \subset H^*_{G,an}(V)$.

\begin{definition}
  Integral density for the disk partition function is the following class in $H_{G,an}(V)$:
  \begin{equation} \label{eq:diskPartitionDensity}
    F_{\fB}(\sigma) := \Gamma \cdot \ch(\fB).
  \end{equation}
\end{definition}

\subsection{Component of the disk partition function}

\paragraph{\bf The integral definition}


Disk partition function is additive in $\fB$, so we first define it on the characters of $G$. 
The definition will depend on the complexified stability parameter $\theta$ which is also a (complexified) character of $G$.

Let us pick an ordered $\bZ$-basis  $\{\xi_a \}_{a=1}^{{\boldsymbol{\kappa}}}$ of $\bL$, and 
let $\{\xi_a^*\}_{a=1}^{{\boldsymbol{\kappa}}}$ be the dual $\bZ$-basis of $\bL^\vee$, as in Section \ref{sec:abelian}. 
This defines a volume form $\dd \sigma = \Lambda_{a=1}^{{\boldsymbol{\kappa}}}\xi^*_a$ and an orientation on $\fg_{\bR}=\bL_\bR$. Let further $\delta \in \fg_{\bR}=\bL_\bR$ and $\cC_{\dd \sigma}(\delta)$ be a cycle
$\delta + \sqrt{-1} \fg_{\bR}$ oriented by the form $\dd \sigma$.

\begin{definition}
  Disk partition function  of  the character $\St$ 
  is an analytic function of  the variable $\theta = \zeta + 2\pi\sqrt{-1} B \in \fg^\vee$ given by the formula:
  \begin{equation} \label{eq:diskPartitionComponent}
    Z_{D^{2}}(\bfL_\St) := \frac{1}{(-2\pi \sqrt{-1})^{{\boldsymbol{\kappa}}}}\int_{\cC_{\dd \sigma}(\delta)} \dd \sigma \; \Gamma \cdot
    e^{\langle \theta + 2\pi \sqrt{-1} \St, \sigma \rangle},
  \end{equation}
 where $\delta \in \fg_{\bR}$ satisfies the condition~\footnote{The condition specifies a
    chamber in the hyperplane arrangement of pole hyperplanes of the integrand. The fact that
    such a chamber is nonempty is not immediate. For the mirror quintic we can choose $\delta = (\sum_{a=1}^{100}s_{1a})c\xi_{101}-c\sum_{a=1}^{100} \xi_{a}$, where c is a small positive number, $0<500c < 1$.} $\forall i \in [1..n+{\boldsymbol{\kappa}}] \; : \;
  \langle \D_{i}, \delta \rangle + \alpha_{i} > 0$. 
  \end{definition}

\begin{remark}
  The formula~\eqref{eq:diskPartitionComponent} is a multidimensional inverse Mellin transform
  (see e.g.~\cite{TsikhBook}) of
  $\Gamma e^{2\pi\sqrt{-1} \langle \St, \sigma\rangle } $, where the $G$-character $\St\in \bL^\vee$ is identified with 
the line bundle $\bfL_{\St} \in K_{G}(V)$ through the natural isomorphism.
\end{remark}

Due to Proposition~\ref{lem:convImaginary}
\begin{remark}
  The right hand side of~\eqref{eq:diskPartitionComponent} converges in the domain
  \begin{equation}
    U_{\St} := \Big\{ B \in \fg_{\bR} \; | \; |\langle B+\St, \nu \rangle| < \sum_{i} |\langle \D_{i}, \nu \rangle|/4 \text{ for all } \nu \in \fg_{\bR} \backslash
    \{0\} \Big\}.
  \end{equation}
  In particular, this holds if $B = -\St$, so that $U_\St$ is non-empty for any character $\St$.
  This condition is quite restrictive, in particular none of the branes in $K(\MF([V/G],\bw))$ in the the mirror quintic example satisfies this condition.
\end{remark}

\paragraph{\bf Expansion in phases}

Now we turn to the description of the disk partition function component in the phases of the GLSM.

Let $C$ be a maximal cone of the secondary fan.
The set of minimal anticones is defined as:
\begin{equation}
  \cA^{\min}_C:= \{I \subset [1 ..  n+\boldsymbol{\kappa} ] \; | \; C \subset \angle_{I}, \; |I| \text{ is minimal}\}
\end{equation}
where $\angle_I$ is defined as in Section \ref{sec:anticones}. The latter condition is equivalent to $|I| = {\boldsymbol{\kappa}}$. 
To formulate our version of Higgs-Coulomb Correspondence (Theorem \ref{the:expansion}), we recall some notations
from Definition \ref{D-dual}. For a fixed minimal anticone $I \in \cA_C^{\min}$, the $G$-characters $\{\D_i : i\in I\}$ form a $\bQ$-basis 
of $\bL^\vee_\bQ$. We define $\{\D_i^{*,I}: i\in I \}$ to be the dual $\bQ$-basis of $\bL_\bQ$. 
\begin{definition}[chamber hemisphere partition function] \label{chamber-hemisphere} 
	Let $C$ be a maximal cone of the secondary fan. Define the chamber hemisphere
	partition function 
	\begin{equation} \label{eq:diskPartitionComponent2}
		Z_{D^{2}}(\bfL_{\St})_{C} := (-1)^{{\boldsymbol{\kappa}}}\sum_{I \in \cA^{min}_C} \left|\frac{\Lambda_{a=1}^{{\boldsymbol{\kappa}}} \xi^*_{a}}{\Lambda_{a=1}^{{\boldsymbol{\kappa}}} \D_{i_{a}}} \right| \sum_{m \in (\bZ_{\ge 0})^{{\boldsymbol{\kappa}}}}
		\prod_{i' \in I'} \Gamma(\langle \D_{i'}, \sigma_{m} \rangle + \alpha_{i'})
		\prod_{i \in I} \frac{(-1)^{m_{i}}}{m_{i}!} \exp(\langle \theta + 2\pi \sqrt{-1}\St,
		\sigma_{m} \rangle),
	\end{equation}
  where $m = (m_{i_{1}}, \ldots, m_{i_{{\boldsymbol{\kappa}}}})$, $I=\{ i_1,\ldots, i_{{\boldsymbol{\kappa}}}\}$, and
  $\sigma_{m} = -\sum_{i \in I} (m_{i} + \alpha_{i} )\D_{i}^{*,I}$.
\end{definition}

\begin{Theorem}[Higgs-Coulomb correspondence] \label{the:expansion}
  Let $U_{C} := \bigcap_{I \in \cA^{\min}_C} U_{I} \subset C$, where $U_{I}$ is defined in Proposition~\ref{prop:apConv}.
  $U_{C}$ is open and nonempty and if $\zeta \in U_{C}$ and $B \in U_{\St}$ then we have the equality
  \begin{equation} \label{eq:diskPartitionComponent2}
    Z_{D^{2}}(\bfL_\St) = Z_{D^2}(\bfL_\St)_{C}
  \end{equation}
\end{Theorem}

\begin{proof}

  Under the assumptions in the theorem all the integrals and series in~\eqref{eq:diskPartitionComponent2} converge due to 
  Proposition~\ref{lem:convImaginary} and Proposition~\ref{prop:apConv}.

  Now we turn to the proof of the equality of the right hand side and left hand side.
  Our strategy is to deform the integration contour in~\eqref{eq:diskPartitionComponent} while
  separating the contributions from different  torus fixed points.

  First we need to introduce some notations to work with contours. Let $\pi^{i}_{m_{i}}$ denote a hyperplane
  $\langle \D_{i}, \sigma \rangle + \alpha_{i} = -m_{i}$ and
  $$ 
  \pi^{\{i_{1}, \ldots, i_{l}\}}_{\{m_{i_{1}}, \ldots, m_{i_{l}}\}} := \bigcap_{k=1}^{l} \left\{ \langle \D_{i_{k}}, \sigma \rangle + \alpha_{i_{k}}  = -m_{i_{k}} \right\}. 
  $$

  Below we will use the following notations: $I_{k} = \{i_{1}, \ldots, i_{k}\} \subset \{1,\ldots, n+{\boldsymbol{\kappa}}\}$ such that 
  $\D_{i_{1}}, \ldots, \D_{i_{k}}$
  are linearly independent if not stated otherwise. $I^{\mathrm{or}}_{k}$ stands for an ordering $(i_{1}, \ldots, i_{k})$ of $I_{k}$. In addition
  $m_{I_{k}} = (m_{i_{1}}, \ldots, m_{i_{k}}) \in (\bZ_{\ge 0})^{k}$.
  
  Given such $I_k$ consider the following exact sequence:
  $$
  1 \to G_{I_{k}} \to G \stackrel{\D_{i_{1}}, \ldots, \D_{i_{k}}}{\lra} H_{I_{k}} \simeq (\bC^{*})^{k} \to  1,
  $$
  where $H_{I_{k}} := G/G_{I_{k}}$. We also denote the corresponding Lie algebras by $\fg_{I_{k}}$ (of dimension ${\boldsymbol{\kappa}}-k$) and $\fh_{I_{k}}$
  (of dimension $k$). Canonically
  $\fh_{I_{k}}^{\vee} = \langle \D_{i_{1}}, \ldots, \D_{i_{k}}\rangle$.
  Note that $\pi^{I_{k}}_{m_{I_{k}}}$ is an affine space parallel to $\fg_{I_{k}}$ so that $\pi_{\fg \to \fh_{I_{k}}} (\pi^{i}_{m_{i}})$ is a hyperplane
  in $\fh_{I_{k}}$ whenever $i \in I_{k}$. We will denote these hyperplanes by the same symbol. If $k = {\boldsymbol{\kappa}}$, then $\pi^{I_{k}}_{m_{I_{k}}}$ is
  just a point and $G_{I_k} = G_{\sigma_{I'_k}}$ is finite.  \\

  Given a point $p \in \pi^{I_{k}}_{m_{I_{k}}}$ that is not contained in the other hyperplanes consider an analytic
  neighbourhood $U$ that does not intersect other hyperplanes and let $[p]$ and $U_{0} = \pi_{\fg \to \fh_{I_{k}}} (U)$ be their projections to
  the quotient $\fh_{I_{k}}$.
  
  The linear independence condition implies that $\bigcup_{l=1}^{k} \pi^{i_{l}}_{m_{i_{l}}}$ is a simple normal crossing divisor for any $m_{I_{k}}$
  with the center at $[p]$, so $U_{0} \backslash \bigcup_{l=1}^{k} \pi^{i_{l}}_{m_{i_{l}}}$ is homotopically equivalent to $(S^{1})^{k}$.
  Let $\cC^{I^{\mathrm{or}}_{k}}_{0}([p])$ denote a generator in $H_{k}(U_{0} \backslash \bigcup_{l=1}^{k} \pi^{i_{l}}_{m_{i_{l}}})$ oriented by the
  form $(2\pi \sqrt{-1})^{-k} \dd \D_{i_{1}}/\D_{i_{1}} \wedge \ldots \wedge \dd \D_{i_{k}}/\D_{i_{k}} $, that is in the basis of (complex) linear functions
  on $\fh_{I_{k}}$ given by $z_{l} = \D_{i_{l}}$ we have
  \begin{equation}
    \frac{1}{(2\pi\sqrt{-1})^{k}}\int_{\cC^{I^{\mathrm{or}}_{k}}_{0}([p])} \Lambda_{l=1}^{k} \frac{\dd z_{l}}{z_{l}}  = 1.
  \end{equation}
  Let also $\cC^{I^{\mathrm{or}}_{k}}_{0}(p)$ denote a generator of $$H_{k}(U \backslash \bigcup_{l=1}^{k} \pi^{i_{l}}_{m_{i_{l}}})$$ projecting to
  $\cC^{I^{\mathrm{or}}_{k}}_{0}([p])$. Let $\Omega \in \Lambda^{{\boldsymbol{\kappa}}-k} (\fg_{I_{k}})^\vee_{\bR}$ be a volume form. Such a form defines an orientation
  of $\sqrt{-1} (\fg_{I_{k}})_{\bR}$. This is canonically the Lie algebra of $(G_{I_k})_\text{comp}$, the maximal compact subgroup
  of $G_{I_k}$. \\

  We define a cycle $$\cC^{I^{\mathrm{or}}_{k}}_{\Omega}(p) \in H_{{\boldsymbol{\kappa}}}(\fg \backslash \mathrm{Polar}, |\Im(\sigma)| \gg 0)$$ as a class of
  $\cC^{I^{\mathrm{or}}_{k}}_{0}(p) + \sqrt{-1} (\fg_{I_{k}})_{\bR}$ oriented by the form 
  $(2\pi\sqrt{-1})^{-k}\Lambda_{l=1}^{k} \dd \D_{i_{l}}/\D_{i_{l}}  \wedge \Omega $.
  In the definition above $\mathrm{Polar}$ stands for the polar divisor of the integrand, that is
  $\bigcup_{i=1}^{n+{\boldsymbol{\kappa}}}\bigcup_{m_{i} \ge 0} \pi^{i}_{m_{i}}$. If $k = {\boldsymbol{\kappa}}$, then $\Omega \in \bR$. If $\Omega = 1$ we usually omit it from
  the notation.
  \\
  
  \begin{remark}
    We will be considering integrals of the form $\int_{\cC_{\Omega}^{I_{k}^{\mathrm{or}}}(p)} \dd \sigma \, \Gamma  \cdot e^{\langle \theta, \sigma \rangle}$.
    Such integrals depend only on the homology class of the cycle in the group
    $$H_{{\boldsymbol{\kappa}}}(\fg \backslash \bigcup_{i =1}^{n+{\boldsymbol{\kappa}}} \bigcup_{m \ge 0} \pi^{i}_{m_{i}}, |\Im(\sigma)| \gg 0).$$
    In particular, two cycles are homologous if they are related by a smooth homotopy that leaves the set $|\Im(\sigma)| \gg 0$ invariant and
    does not cross polar hyperplanes.
  \end{remark}
  \begin{lemma} \label{lem:convImaginary}
    Let $\cC^{I^{or}_{k}}_{\Omega}(p)$ be as above, $k<{\boldsymbol{\kappa}}$ and
    \begin{equation} \label{eq:convCondition1}
      |\langle B, \nu \rangle| < \frac14 \sum_{i=1}^{n+{\boldsymbol{\kappa}}}\langle \D_{i}, \nu \rangle \text{ for all }\nu \in \fg_{I_{k}} \backslash \{0\}.
    \end{equation}
    Then the integral
    \begin{equation}
      \int_{\cC^{I^{\mathrm{or}}_{k}}_{\Omega}(p)} \dd \sigma \, \Gamma \cdot e^{\langle \theta, \sigma \rangle}
    \end{equation}
    is absolutely convergent. In particular it is always convergent if $B \in \fg_{I_{k}}^{\perp}$. We show that $\Gamma \cdot e^{\langle \theta,
    \sigma \rangle}$ uniformly exponentially decays as $|\sigma| \to \infty$ on the integration cycle.
  \end{lemma}
  \begin{proof}
    We present the proof of this statement in Appendix \ref{appendix-A}. 
  \end{proof}

  \begin{lemma}[One plane crossing]
    \label{lem:contDef}
    Let $I_{k}^{\mathrm{or}} = (I_{k-1}^{\mathrm{or}}, i)$ such that $\fg_{I_{k-1}} \ne \fg_{I_{k-1} \cup \{i\}}$, $\Omega$ be a volume form on $(\fg_{I_{k-1}})_{\bR}$ and $p \in \pi^{I_{k}}_{m_{I_k}}$ be generic (does not intersect other polar hyperplanes).
    Pick $f \in \fg_{I_{k-1}} \backslash \fg_{I_{k}}$.    
    \begin{equation}
      \cC^{I^{\mathrm{or}}_{k-1}}_{\Omega}(p + \varepsilon f) = \cC^{I^{\mathrm{or}}_{k}}_{\iota_{f} \Omega}(p) +
      \cC^{I^{\mathrm{or}}_{k-1}}_{\Omega}(p - \varepsilon f),
    \end{equation}
    where $\iota_{f} \Omega$ is a ${\boldsymbol{\kappa}}-k$ form on $(\fg_{I_{k}})_{\bR}$ and $\varepsilon$ small enough.
    
  \end{lemma}
  \begin{proof}
    Let $D(z,r)$ denote a disk with center at $z \in \bC$ and radius $r$.
    Consider a linear coordinate system $\sigma_{1}, \ldots, \sigma_{{\boldsymbol{\kappa}}}$ on $\fg$ with the center at $p$ such that $\cC^{I^{\mathrm{or}}_{k}}_{\Omega}(r + \varepsilon f)$ can be represented by the cycle
    $$ (\pd D(0, \alpha))^{k-1} \times (\varepsilon + \sqrt{-1}\bR) \times (\sqrt{-1} \bR^{{\boldsymbol{\kappa}}-k}) $$
    for for some $\alpha \in \bR \backslash \{0\}$ such that $f = \pd/\pd \sigma_{k}$. Then the form $\Omega$ is equal to
    $c \, \dd \sigma_k \wedge \cdots \wedge \dd \sigma_{{\boldsymbol{\kappa}}}$, where $c \in \bR \backslash \{0\}$.
    This cycle is homotopic to the union     
    $$ (\pd D(0, \alpha))^{k-1} \times (-\varepsilon + \sqrt{-1}\bR) \times (\sqrt{-1} \bR^{{\boldsymbol{\kappa}}-k}) $$
    oriented by the same form and    
    $$ (\pd D(0, \alpha))^{k} \times (\sqrt{-1} \bR^{{\boldsymbol{\kappa}}-k}) $$
    oriented by the form $c \, \dd \sigma_{k+1} \wedge \ldots \wedge \dd \sigma_{{\boldsymbol{\kappa}}} = \iota_{f} \Omega$.
    The homotopy is nontrivial only in the $k$-th component, where it is shown in Figure~\ref{fig:contour1}.

    \begin{figure}[h!]
    \centering
    \def\svgwidth{6cm}
\begingroup%
  \makeatletter%
  \providecommand\color[2][]{%
    \errmessage{(Inkscape) Color is used for the text in Inkscape, but the package 'color.sty' is not loaded}%
    \renewcommand\color[2][]{}%
  }%
  \providecommand\transparent[1]{%
    \errmessage{(Inkscape) Transparency is used (non-zero) for the text in Inkscape, but the package 'transparent.sty' is not loaded}%
    \renewcommand\transparent[1]{}%
  }%
  \providecommand\rotatebox[2]{#2}%
  \newcommand*\fsize{\dimexpr\f@size pt\relax}%
  \newcommand*\lineheight[1]{\fontsize{\fsize}{#1\fsize}\selectfont}%
  \ifx\svgwidth\undefined%
    \setlength{\unitlength}{238.45066101bp}%
    \ifx\svgscale\undefined%
      \relax%
    \else%
      \setlength{\unitlength}{\unitlength * \real{\svgscale}}%
    \fi%
  \else%
    \setlength{\unitlength}{\svgwidth}%
  \fi%
  \global\let\svgwidth\undefined%
  \global\let\svgscale\undefined%
  \makeatother%
  \begin{picture}(1,0.77662923)%
    \lineheight{1}%
    \setlength\tabcolsep{0pt}%
    \put(0,0){\includegraphics[width=\unitlength,page=1]{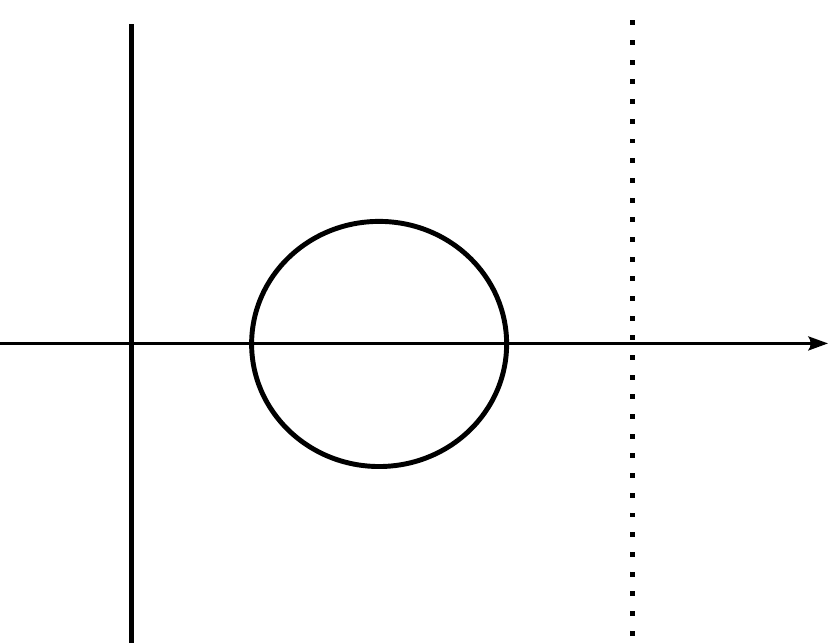}}%
    \put(0.77177622,0.24526862){\makebox(0,0)[lt]{\lineheight{1.25}\smash{\begin{tabular}[t]{l}$\varepsilon$\end{tabular}}}}%
    \put(0.16665635,0.25278187){\makebox(0,0)[lt]{\lineheight{1.25}\smash{\begin{tabular}[t]{l}$-\varepsilon$\end{tabular}}}}%
    \put(0,0){\includegraphics[width=\unitlength,page=2]{Deformation1.pdf}}%
    \put(0.20324017,0.70015897){\makebox(0,0)[lt]{\lineheight{1.25}\smash{\begin{tabular}[t]{l}$\sigma_k$\end{tabular}}}}%
  \end{picture}%
\endgroup%

    \caption{Basic contour deformation}
    \label{fig:contour1}
    \end{figure}

  \end{proof}

  Applying the lemma several times we get the following statement:
  \begin{lemma}
    \label{lem:contDef}
    Consider $I_{k-1}^{\mathrm{or}}$, a volume form $\Omega$ on $(\fg_{I_{k}})_{\bR}$ and $p \in \pi^{I_{k-1}}_{m_{I_{k-1}}}$ generic (does not intersect other polar hyperplanes) and $f \in \fg_{I_{k-1}} \backslash \bigcup_{i, \; \fg_{I_{k-1}} \ne \fg_{I_{k-1} \cup \{i\}}} \fg_{I_{k-1} \cup \{i\}}$.
    \begin{equation}
      \cC^{I^{\mathrm{or}}_{k-1}}_{\Omega}(p) = -\sum_{i \notin I_{k-1}}\sum_{m_{i} \ge 0} \cC^{I^{\mathrm{or}}_{k}}_{\iota_{f} \Omega}(\{p + [0,N]f\} \cap \pi^{i}_{m_{i}}) +
      \cC^{I^{\mathrm{or}}_{k-1}}_{\Omega}(p + N f),
    \end{equation}
    where $\iota_{f} \Omega$ is a ${\boldsymbol{\kappa}}-k$ form on $(\fg_{I_{k}})_{\bR}$ and the cycle is empty if the intersection is. We note that only $i$ such
    that $\fg_{I_{k-1}} \ne \fg_{I_{k-1} \cup \{i\}}$ enter the right hand side.
  \end{lemma}
  
  \begin{remark} \label{rem:contourConvergence}
    Later we consider integrals of the form
     $$
     \int_{\cC} \dd \sigma \, \Gamma \cdot e^{\langle \theta, \sigma \rangle},
     $$
    $\cC \in H_{{\boldsymbol{\kappa}}}(\fg \backslash Polar, |\Im(\sigma)| \gg 0)$. Let $\{\cC_{\alpha} \}_{\alpha \in A} \subset H_{{\boldsymbol{\kappa}}}(\fg \backslash Polar, |\Im(\sigma)| \gg 0)$. If $A$ is an infinite set, then by abuse of notations when we write an expression of the form
    \begin{equation}
      \sum_{\alpha \in A} \cC_{\alpha} = \cC
    \end{equation}
    we actually mean
    \begin{equation}
      \sum_{\alpha \in A} \int_{\cC_{\alpha}} \dd \sigma \, \Gamma \cdot e^{\langle \theta, \sigma \rangle} = \int_{\cC} \dd \sigma \,
      \Gamma \cdot e^{\langle \theta, \sigma \rangle},
    \end{equation}
    when the left hand side converges. We use this notation to simplify already bulky notations.
  \end{remark}
  
  \begin{Corollary} \label{cor:contDef}
    Let $p, \Omega, N$ be as in Lemma~\ref{lem:contDef} then there exist a constant $const$ that does not depend on $p$ and $f$ such that if $\langle \zeta, f \rangle < const$ then
    \begin{equation}
      \cC^{I^{\mathrm{or}}_{k-1}}_{\Omega}(p) = -\sum_{i \notin I_{k-1}}\sum_{m_{i} \ge 0} \cC^{I^{\mathrm{or}}_{k}}_{\iota_{f} \Omega}(\{p + \bR_{\ge 0}f\} \cap \pi^{i}_{m_{i}}).
    \end{equation}
  \end{Corollary}
  \begin{proof}
    Lemma~\ref{lem:contDef} implies
    \begin{equation}
      \int_{\cC^{I^{\mathrm{or}}_{k-1}}_{\Omega}(p)} \dd \sigma \, \Gamma \cdot e^{\langle
        \theta, \sigma \rangle} = -\sum_{i \notin I_{k-1}}\sum_{m_{i \ge 0}} \int_{\cC^{I^{\mathrm{or}}_{k}}_{\iota_{f}\Omega}(\{p+[0,N]f\} \cap \pi^{i}_{m_{i}})} \dd \sigma \, \Gamma \cdot e^{\langle
        \theta, \sigma \rangle} + \int_{\cC^{I^{\mathrm{or}}_{k-1}}_{\Omega}(p+Nf)} \dd \sigma \, \Gamma \cdot e^{\langle
        \theta, \sigma \rangle}.
    \end{equation}
    Using Proposition~\ref{prop:integralBound} with $p \to p+Nf$ we can estimate the last term on the right hand side:
    \begin{equation}
      \int_{\cC^{I^{\mathrm{or}}_{k-1}}_{\Omega}(p+Nf)} \dd \sigma \, \Gamma \cdot e^{\langle
        \theta, \sigma \rangle} \le e^{-const |p+Nf|},
    \end{equation}
    In particular, the integral approaches 0 as $N \to \infty$.
    Now the corollary follows from taking the limit $N \to \infty$.

  \end{proof}

\begin{figure}[h!]
  \centering
  \def\svgwidth{10cm}
\begingroup%
  \makeatletter%
  \providecommand\color[2][]{%
    \errmessage{(Inkscape) Color is used for the text in Inkscape, but the package 'color.sty' is not loaded}%
    \renewcommand\color[2][]{}%
  }%
  \providecommand\transparent[1]{%
    \errmessage{(Inkscape) Transparency is used (non-zero) for the text in Inkscape, but the package 'transparent.sty' is not loaded}%
    \renewcommand\transparent[1]{}%
  }%
  \providecommand\rotatebox[2]{#2}%
  \newcommand*\fsize{\dimexpr\f@size pt\relax}%
  \newcommand*\lineheight[1]{\fontsize{\fsize}{#1\fsize}\selectfont}%
  \ifx\svgwidth\undefined%
    \setlength{\unitlength}{118.94634697bp}%
    \ifx\svgscale\undefined%
      \relax%
    \else%
      \setlength{\unitlength}{\unitlength * \real{\svgscale}}%
    \fi%
  \else%
    \setlength{\unitlength}{\svgwidth}%
  \fi%
  \global\let\svgwidth\undefined%
  \global\let\svgscale\undefined%
  \makeatother%
  \begin{picture}(1,0.83253538)%
    \lineheight{1}%
    \setlength\tabcolsep{0pt}%
    \put(0,0){\includegraphics[width=\unitlength,page=1]{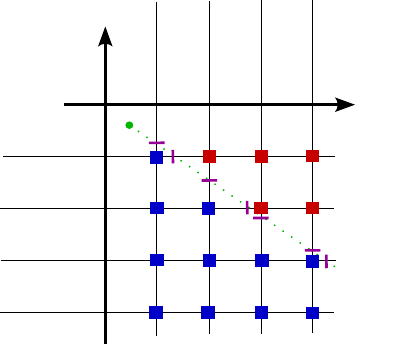}}%
    \put(0.27727123,0.60446351){\makebox(0,0)[lt]{\lineheight{1.25}\smash{\begin{tabular}[t]{l}$\cC_{\dd x\wedge \dd y}(p)$\end{tabular}}}}%
    \put(0.80508067,0.22674377){\makebox(0,0)[lt]{\lineheight{1.25}\smash{\begin{tabular}[t]{l}$\cC_{f^{\perp}}^{(1)}(q)$\end{tabular}}}}%
    \put(0.7909362,0.091069){\makebox(0,0)[lt]{\lineheight{1.25}\smash{\begin{tabular}[t]{l}$\cC^{(2,1)}_{1}((-4,-4))$\end{tabular}}}}%
    \put(0.79091311,0.48766576){\makebox(0,0)[lt]{\lineheight{1.25}\smash{\begin{tabular}[t]{l}$\cC^{(1,2)}_{1}((-4,-1))$\end{tabular}}}}%
    \put(0.11037202,0.73389801){\makebox(0,0)[lt]{\lineheight{1.25}\smash{\begin{tabular}[t]{l}$\fg_{\bR}$\end{tabular}}}}%
  \end{picture}%
\endgroup%

  \caption{Contour deformations in projection to $\fg_{\bR}$.}
  \label{fig:contour2}
\end{figure}
Figure~\ref{fig:contour2} shows two consecutive applications of Corollary~\ref{cor:contDef} in a simple example.
  The plane is a real part $\fg_{\bR}$ of the two-dimensional Lie algebra $\fg \simeq \bC^{2}$. The picture shows only 2 series of polar divisors
  located at $x = -1, -2, -3, \ldots$ and $y = -1, -2, -3, \ldots$. The initial contour $\cC_{\dd x \wedge \dd y}(p)$ is a purely imaginary contour
  depicted by a green dot. Its projection to the real Lie algebra is just a point. Direction of deformation $f$ is a dashed green line.
  After first application of Corollary~\ref{cor:contDef} the contour transforms into two series of copies of $\bR \times S^{1}$ denoted by
  purple intervals. Indeed, projections of these contours to $\fg_{\bR}$ are intervals intersecting the corresponding polar hyperplanes.
  Each of this contour is deformed parallel to the respective polar hyperplane applying Corollary~\ref{cor:contDef} again. The contours deform to a
  number of $T^{2} \simeq (S^{1})^{2}$ winding around each crossing of polar hyperplanes.\\

  Let us introduce some more notations for cones of different dimensions in $\fg$.
  Given $I$ such that $\{\D_{i}\}_{i \in I}$ form a basis of $\fg^\vee$ consider the cone
  \begin{equation} \label{eq:correctedDualCone}
    \angle_{I}^{*,\zeta} = \{ \sum_{i \in I}\mathrm{sign}(\langle \zeta, \D^{*, I}_{i} \rangle) c_{i} \D^{*, I}_{i} \; | \; c_{i} \ge 0 \},
  \end{equation}
  and its subcones:
  \begin{equation} \label{eq:correctedDualSubcones}
    \angle_{J \subset I}^{*,\zeta} = \angle^{*,\zeta}_{I} \cap \fg_{J}.
  \end{equation}
  The signs are chosen so that $\langle \zeta, \sigma \rangle \ge 0$ for any $\sigma \in \angle_{J \subset I}^{*,\zeta}$.
  We also define
  $$
  \mathrm{sign}_{J \subset I}^{\zeta} = (-1)^{\#\{i \in I \backslash J \; | \;  \langle \zeta, \D^{*,I}_{i} \rangle < 0 \}}. 
  $$
  In particular, if $I$ is a minimal anticone (which is the main case of interest) then $\angle^{*,\zeta}_{I} = \angle^{*}_{I}$, 
  where $\angle^*_I \subset \fg_\bR$ is  the dual cone of $\angle_I \subset \fg^\vee_\bR$,  and $\mathrm{sign}_{J \subset I}^{\zeta} = 1$ for all $J \subset I$. 
  We also have $\angle^{*,\zeta}_{\mempty \subset I} = \angle^{*,\zeta}_{I}$ and  $\angle^{*,\zeta}_{I \subset I} = 0$.
  
  Further recall that
  $$
  \bK^{\eff}_I =\Big\{  \sum_{i\in I} m_i \D^{*,I}_{i} \;\Big|\: m_i\in \bZ_{\geq 0} \Big\} 
  $$
  and define
  $$
  \bK_{I}^{\eff}(\alpha) := - \sum_{i\in I}\alpha_i \D^{*,I}_i -\bK^{\eff}_{I} = \Big\{  - \sum_{i\in I} (\alpha_i + m_i) \D^{*,I}_i \;\Big|\: m_i\in \bZ_{\geq 0}\Big\}. 
  $$
  If $\alpha_i = q_i/2$ then $\bK_{I}^{\eff}(\alpha) = - \bK^{\eff}_{I}(\sX_{\zeta}, \bw_{\zeta})$. 
 If $p\in \pi^{I_k}_{m_{I_k}}$  then
 $$
 p -\angle^{*\zeta}_{I_k\subset I} = \Big\{  -\sum_{i\in I_k} (\alpha_i+m_i)\D_i^{*, I} + \sum_{i\in I\setminus I_k} (\langle \D_i, p\rangle - c_i) \D_i^{*, I} 
 \: \Big| \;  c_i\in \bR,  c_i\geq 0  \text{ if } 
 \langle \zeta, \D_i^{*,I} \rangle > 0 \text{ and } c_i\leq 0 \text{ if } \langle \zeta, \D_i^{*,I}\rangle < 0 \Big\}. 
 $$

  The following lemma is the main technical lemma of the proof. It relates integrals over $(S^1)^{k} \times \bR^{{\boldsymbol{\kappa}}-k}$ to
  sums of cycles over $(S^1)^{{\boldsymbol{\kappa}}}$ in some ${\boldsymbol{\kappa}}-k$-dimensional
  cone by recursively applying Corollary~\ref{cor:contDef} and checking that only the contributions
  in the corresponding cone survive.
  \begin{lemma} \label{lem:residues}

    Consider $I_{k}^{\mathrm{or}}$, a volume form $\Omega$ on $(\fg_{I_{k}})_{\bR}$ and $p \in \pi^{I_{k}}_{m_{I_{k}}}$ generic (does not intersect other polar hyperplanes). Then
    
    \begin{equation} \label{eq:recursion1}
      \begin{aligned}
        \cC^{I^{\mathrm{or}}_{k}}_{\Omega}(p) = (-1)^{{\boldsymbol{\kappa}}-k}\sum_{I, I_{k} \subset I} \mathrm{sign}_{I_{k} \subset I}^{\zeta}  \sum_{q \in (p - \angle^{*,\zeta}_{I_{k} \subset I}) \cap \bK^{\eff}_{I}(\alpha)}
        \mathrm{sign}\left( \frac{\Lambda_{l=1}^{k} \D_{i_{l}} \wedge \Omega}{\Lambda_{l=1}^{{\boldsymbol{\kappa}}} \D_{i_{l}}} \right) \cC^{I^{\mathrm{or}}}(q).
      \end{aligned}
    \end{equation}
    We note that we can describe the set  $(p-\angle^{*, \zeta}_{I_{k} \subset I}) \cap \bK_{I}^{\eff}(\alpha)$ explicitly as follows.  
    Let $-n_{i'} := \lfloor \langle \D_{i'}, p \rangle + \alpha_{i'}\rfloor \in \bZ$ for $i' \in I \backslash I_{k}$. Then
    $$ 
    (p-\angle^{*, \zeta}_{I_{k} \subset I}) \cap \bK_{I}^{\eff}(\alpha) = \Big\{ -\sum_{i \in I}(\alpha_{i}+m_{i})\D^{*,I}_{i} 
    \; \Big| \:  i \in I \backslash I_{k} \implies m_i\in \bZ\text{ and }
      \begin{cases}
         n_{i} \le m_{i} & \text{if } \langle \zeta, \D^{*,I}_{i}\rangle > 0, \\
         0 \le m_{i} < n_{i} & \text{if } \langle \zeta, \D^{*,I}_{i} \rangle < 0.
      \end{cases}
    \Big\}.  
    $$
  \end{lemma}

  \begin{proof}
    The right hand side is convergent due to Proposition \ref{prop:apConv} since the series in the right hand side
    is a subseries of the one considered in the proposition.

    We prove the lemma by recursion in $k$. First, let $k = {\boldsymbol{\kappa}}-1$. We apply Corollary~\ref{cor:contDef} where $f \in (\fg_{I_{{\boldsymbol{\kappa}}-1}})_\bR \simeq \bR$
    such that $\langle \zeta, f \rangle < 0$. $\Omega \in (\fg_{I_{{\boldsymbol{\kappa}}-1}})_{\bR}^{\vee} \simeq \fg_\bR^\vee/ \left(\oplus_{i\in I_{{\boldsymbol{\kappa}}-1}}\bR \D_i \right)$, so that
    for each $i \notin I_{{\boldsymbol{\kappa}}-1}$ one can write $\Omega = c_{i} \D_{i}$ (where right hand side is a representative). In particular,
    $\iota_{f} \Omega = c_{i} \langle \D_{i}, f \rangle$, where
    $$
    c_{i} = \frac{\Lambda_{l=1}^{{\boldsymbol{\kappa}}-1}\D_{i_{l}} \wedge \Omega}{\Lambda_{l=1}^{{\boldsymbol{\kappa}}-1}\D_{i_{l}} \wedge \D_{i}}. 
    $$
    Furthermore, $\mathrm{sign}(\langle \D_{i}, f \rangle) = \mathrm{sign}(\langle \zeta, \D^{*, I}_{i} \rangle) = 
    \mathrm{sign}_{I_{{\boldsymbol{\kappa}}-1} \subset I}^{\zeta}$. So we have
    \begin{equation}
      \cC^{I^{\mathrm{or}}_{{\boldsymbol{\kappa}}-1}}_{\Omega}(p) = -\sum_{i \notin I_{{\boldsymbol{\kappa}}-1}}\sum_{m_{i} \ge 0} \cC^{(I^{\mathrm{or}}, i)}_{c_i\langle \D_{i}, f \rangle}(\{p + \bR_{\ge 0}f\} \cap
      \pi^{i}_{m_{i}}).
    \end{equation}
    Thus the theorem in this case follows.

    We remark that the conditions in Corollary~\ref{cor:contDef} are stricter than the convergence condition of the theorem, but the statement
    is about equality of the analytic functions is some domain, so it extends to the maximal domain of the mutual convergence.\\

    Consider the general case $0 \le k < {\boldsymbol{\kappa}}-1$. We choose a generic $f \in (\fg_{I_{k}})_{\bR}\simeq \bR^{\boldsymbol{\kappa}-k}$ such that $\langle \zeta, f \rangle < 0$ and
    apply Corollary~\ref{cor:contDef}

    \begin{equation} \label{eq:residueExpansion1}
      \cC^{I^{\mathrm{or}}_{k}}_{\Omega}(p) = -\sum_{i \notin I_{k}}\sum_{m_{i} \ge 0} \cC^{(I^{\mathrm{or}}_{k}, i)}_{\iota_{f} \Omega}(\{p + [0,N]f\} \cap \pi^{i}_{m_{i}}) +
      \cC^{I^{\mathrm{or}}_{k}}_{\Omega}(p + N f).
    \end{equation}
    The sum over $m_{i} \in \bZ$ restricts to a set
    \begin{equation}
      \begin{cases}
        n_{i} \le m_{i}, \; \langle \D_{i}, f \rangle > 0, \\
        0 \le m_{i} < n_{i}, \; \langle \D_{i}, f \rangle < 0,
      \end{cases}
    \end{equation}
    where again $-n_{i} = \lfloor \langle \D_{i}, p \rangle + \alpha_{i}\rfloor.$
    In particular, if $I$ is a minimal anticone then the second case does not appear.

    Now we can use the recursion step for each term in the double sum on the right hand side.
    \begin{equation}\label{eq:recursion2}
      \begin{aligned}
        \cC^{(I^{\mathrm{or}}_{k}, i)}_{\iota_{f} \Omega}(\tilde p=\{p + [0,N]f\} \cap \pi^{i}_{m_{i}}) = \sum_{I, I_{k} \cup \{i\} \subset I}
        \mathrm{sign}_{I_{k} \cup \{i\} \subset I}^{\zeta}  \sum_{q \in (\tilde{p} - \angle^{*,\zeta}_{I_{k} \cup \{i\} \subset I}) \cap \bK^{\eff}_{I}(\alpha)}
        \mathrm{sign}\left( \frac{\Lambda_{l=1}^{k} \D_{i_{l}}\wedge \D_{i} \wedge \iota_{f}\Omega}{\Lambda_{l=1}^{{\boldsymbol{\kappa}}} \D_{i_{l}}} \right)
        \cC^{I^{\mathrm{or}}}(q)
      \end{aligned}
    \end{equation}

    Let us analyze when $q \in \bK^{\eff}_{I}(\alpha)$ appears in the expansion of the formula~\eqref{eq:residueExpansion1} after applying
    \eqref{eq:recursion2}.
    Such a term appears when there exist $i \in I'_{k}, \; m_{i} \in \bZ$ such that
    \begin{equation} \label{eq:lemResidues1}
      q \in \{p + [0,N] f\} \cap \pi^{i}_{m_{i}} -\angle^{*,\zeta}_{I_{k} \cup \{i\} \subset I}.
    \end{equation}

    But $\bigcup_{i \in I'_{k}} \angle^{*,\zeta}_{I_{k} \cup \{i\} \subset I} = \pd \angle^{*,\zeta}_{I_{k} \subset I}$. So we can rewrite
    equation~\eqref{eq:lemResidues1} as
    $$
    q \in p+cf - \pd \angle^{*,\zeta}_{I_{k} \subset I}
    $$
    for some $c > 0$, or
    \begin{equation}
      q -cf \in p - \pd \angle^{*,\zeta}_{I_{k} \subset I}.
    \end{equation}
    In other words, for each intersection of the ray $q - \bR_{\ge 0} f$ with the boundary of the cone $p - \angle^{*,\zeta}_{I_{k} \subset I}$ we have
    a term with the cycle centered at $q$ in the expansion of~\eqref{eq:residueExpansion1}. Number of such intersections depends on whether
    $q \in p - \angle^{*,\zeta}_{I_{k} \subset I}$ or not.
    \begin{enumerate}
      \item $q \in p - \angle^{*,\zeta}_{I_{k} \subset I}$. Scalar product of $\zeta$ with all the inward normals to 
      $\pd \angle^{*,\zeta}_{I_k}$
            is by definition positive, $f$ is a linear combination of these normals, and $\langle \zeta, f \rangle < 0$. Then at least one
            coefficient of the linear combination is negative, and $q - \bR_{\ge 0} f$ has one intersection point with
        $p-\pd \angle^{*,\zeta}_{I_{k} \subset I}$. The total contribution to the sum is
        \begin{equation}
          \label{eq:onePointContribution}
           \mathrm{sign}_{I_{k} \cup \{i\} \subset I}^{\zeta}\mathrm{sign}\left( \frac{\Lambda_{l=1}^{k} \D_{i_{l}}\wedge \D_{i} \wedge \iota_{f}\Omega}{\Lambda_{l=1}^{{\boldsymbol{\kappa}}} \D_{i_{l}}} \right) \cC^{I^{\mathrm{or}}}(q),
         \end{equation}
        where
        $$ 
        \Omega = c \D_{i} \wedge \Lambda_{l=k+2}^{{\boldsymbol{\kappa}}} \D_{i_{l}} + \cdots 
        $$
        and
        $$
        \iota_{f} \Omega = c \langle \D_{i}, f \rangle \Lambda_{l=k+2}^{{\boldsymbol{\kappa}}} \D_{i_{l}} + \cdots,
        $$
        where dots denote the possible terms that vanish in the formula~\eqref{eq:onePointContribution}. Plugging this expression for the sign, we get
        the equality
        \begin{equation}
          \mathrm{sign}_{I_{k} \cup \{i\} \subset I}^{\zeta}\mathrm{sign}\left( \frac{\Lambda_{l=1}^{k} \D_{i_{l}}\wedge \D_{i} \wedge \iota_{f}\Omega}{\Lambda_{l=1}^{{\boldsymbol{\kappa}}} \D_{i_{l}}} \right) = \mathrm{sign}_{I_{k} \subset I}^{\zeta}\mathrm{sign}\left( \frac{\Lambda_{l=1}^{k} \D_{i_{l}}\wedge \Omega}{\Lambda_{l=1}^{{\boldsymbol{\kappa}}} \D_{i_{l}}} \right).
        \end{equation}
      \item $q \notin p - \angle^{*, \zeta}_{I_{k} \subset I}$. There are either 0 or two intersection points because the scalar product of
            $f$ with at least one inward normal to $\pd \angle^{*, \zeta}_{I_{k} \subset I}$ is negative. The total contribution to the
        formula~\eqref{eq:residueExpansion1} is
        \begin{equation}
          \label{eq:twoPointContribution}
          \left(\mathrm{sign}_{I_{k} \cup \{i\} \subset I}^{\zeta}\mathrm{sign}\left( \frac{\Lambda_{l=1}^{k} \D_{i_{l}}\wedge \D_{i} \wedge
              \iota_{f}\Omega}{\Lambda_{l=1}^{{\boldsymbol{\kappa}}} \D_{i_{l}}} \right) + \mathrm{sign}_{I_{k} \cup \{j\} \subset I}^{\zeta}\mathrm{sign}\left(
            \frac{\Lambda_{l=1}^{k} \D_{i_{l}}\wedge \D_{j} \wedge \iota_{f}\Omega}{\Lambda_{l=1}^{{\boldsymbol{\kappa}}} \D_{i_{l}}} \right)\right)\cC^{I^{\mathrm{or}}}(q).
        \end{equation} 
        Here we can write
        $$ \Omega = c \D_{i} \wedge \D_{j} \wedge \Lambda_{l=k+3}^{{\boldsymbol{\kappa}}} \D_{i_{l}} + \cdots, $$
        where dots denote irrelevant terms again.
        In the first summand we have
        $$\iota_{f}\Omega = c \langle \D_{i}, f \rangle \D_{j} \wedge \Lambda_{l=k+3}^{{\boldsymbol{\kappa}}} \D_{i_{l}} + \cdots,$$
        and in the second one
        $$\iota_{f}\Omega = -c \langle \D_{j}, f \rangle \D_{i} \wedge \Lambda_{l=k+3}^{{\boldsymbol{\kappa}}} \D_{i_{l}} + \cdots.$$
        So the total coefficient is
        \begin{multline} \label{eq:normals1}
            \mathrm{sign}^{\zeta}_{I_{k} \cup \{i\}} \mathrm{sign}(c\langle \D_{j}, f \rangle) + \mathrm{sign}^{\zeta}_{I_{k} \cup \{j\}}
            \mathrm{sign}(c\langle \D_{i}, f \rangle) = \\ = \mathrm{sign}^{\zeta}_{I_{k}} \left(\mathrm{sign}(c\langle \zeta, \D^{*I}_{j} \rangle
            \langle \D_{j}, f \rangle) +
            \mathrm{sign}(c\langle \zeta, \D^{*, I}_{i} \rangle\langle \D_{i}, f \rangle) \right).
        \end{multline}
        In the last formula $\langle \zeta, \D^{*, I}_{i} \rangle \D_{i}$ and $\langle \zeta, \D^{*, I}_{j} \rangle \D_{j}$ are inward normals of the cone
        $\angle^{*,\zeta}_{I_{k} \subset I}$. However, we know that the ray $q - \bR_{\ge 0} f$ intersects the cone $p - \angle^{*,\zeta}_{I_{k} \subset I}$
        with inward pointing first intersection and outward pointing second one, so the signs in the second line of the formula~\eqref{eq:normals1}
        are different and the total contribution vanishes.
    \end{enumerate}

    Collecting contributrions from all the points $q \in \bK^{\eff}_{I}(\alpha)$ for all $I$ using the statement above we obtain the statement of the
    lemma.
The picture is shown in our toy example (we suppose that $\langle \zeta, \D^{*\{1,2\}}_{1} \rangle, \; \langle \zeta, \D^{*\{1,2\}}_{2} \rangle > 0$)
    in Figure~\ref{fig:contour2}.

  \end{proof}

    Using Lemma~\ref{lem:residues} in the case when $\cC^{I^{\mathrm{or}}}_{\Omega}(p) = \cC_{\dd \sigma}(\delta)$ (where we remind that our choice is
  $\dd \sigma = \Lambda_{a=1}^{{\boldsymbol{\kappa}}} \xi^*_{a}$) we write
  \begin{equation} \label{eq:residueExpansion1}
    \int_{\cC_{\dd \sigma}(\delta)} \dd \sigma \, \Gamma \cdot e^{\langle \theta, \sigma \rangle} = \\
    = (-1)^{{\boldsymbol{\kappa}}}\sum_{I \in \cA^{\min}_{C}} \sum_{q \in \bK^{\eff}_{I}(\alpha)} \mathrm{sign}\left(\frac{\Lambda_{a=1}^{{\boldsymbol{\kappa}}} \xi^*_{a}}{\Lambda_{a=1}^{{\boldsymbol{\kappa}}}
      \D_{i_{a}}} \right)\int_{\cC^{I^{\mathrm{or}}}(q)} \dd \sigma \, \Gamma \cdot e^{\langle \theta, \sigma \rangle}.
  \end{equation}
  
  Let us compute each term in the right hand side.
  \begin{lemma} \label{lem:residue}
    \begin{equation}
      \int_{\cC^{I^{\mathrm{or}}}(q)} \dd \sigma \, \Gamma \cdot e^{\langle \theta, \sigma \rangle} =
      (2\pi \sqrt{-1})^{{\boldsymbol{\kappa}}}\frac{\Lambda_{a=1}^{{\boldsymbol{\kappa}}} \xi^*_a}{\Lambda_{a=1}^{{\boldsymbol{\kappa}}} \D_{i_{a}}}
    \prod_{i' \in I'} \Gamma(\langle \D_{i'}, \sigma_{m} \rangle + \alpha_{i'}) \prod_{i \in I} \frac{(-1)^{m_{i}}}{m_{i}!} e^{\langle \theta, \sigma_{m} \rangle}.
    \end{equation}
  \end{lemma}

  \begin{proof}
    Each integral is a ${\boldsymbol{\kappa}}$-dimensional torus around a simple normal crossing divisor and
    $q = -\sum_{i \in I} (\alpha_{i} + m_{i})\D^{*, I}_{i}$. Let $\{\sigma^{I}_{i}\}_{i \in I}$ denote a set of coordinates on $\fg$ corresponding to
    linear functions $\{\D_{i}\}_{i \in I}$ on $\fg$. Then we have
    \begin{equation}
      \int_{\cC^{I^{\mathrm{or}}}(q)} \dd \sigma \, \Gamma \cdot e^{\langle \theta, \sigma \rangle} =
      (2\pi \sqrt{-1})^{{\boldsymbol{\kappa}}} \res_{\cC^{I^{\mathrm{or}}}(q)}\dd \sigma \, \Gamma \cdot e^{\langle \theta, \sigma
        \rangle}.
    \end{equation}
    Next we use
    \begin{equation}
      \dd \sigma = \frac{\Lambda_{a=1}^{{\boldsymbol{\kappa}}} \xi_a^*}{\Lambda_{a=1}^{{\boldsymbol{\kappa}}} \D_{i_{a}}} \dd \sigma^{I}
    \end{equation}
    so that
    \begin{multline}
      \res_{\cC^{I^{\mathrm{or}}}(q)} \dd \sigma \, \Gamma \cdot e^{\langle \theta, \sigma
        \rangle} = \frac{\Lambda_{a=1}^{{\boldsymbol{\kappa}}} \xi_a^*}{\Lambda_{a=1}^{{\boldsymbol{\kappa}}_a^*} \D_{i_{a}}}
      \prod_{a=1}^{{\boldsymbol{\kappa}}}\res_{\sigma^{I}_{i_{a}} \to -\alpha_{i_{a}} - m_{i_{a}}} \dd \sigma^{I}_{i_{a}} \, \Gamma \cdot e^{\langle \theta, \sigma \rangle} = \\
      = \frac{\Lambda_{a=1}^{{\boldsymbol{\kappa}}} \xi_a^*}{\Lambda_{a=1}^{{\boldsymbol{\kappa}}_a^*} \D_{i_{a}}}
      \prod_{i' \in I'} \Gamma(\langle \D_{i'}, \sigma_{m} \rangle + \alpha_{i'}) \prod_{i \in I} \frac{(-1)^{m_{i}}}{m_{i}!} e^{\langle \theta, \sigma_{m} \rangle},
    \end{multline}
    where $\sigma_{m} = q$.
  \end{proof}
  Substituting this expression into~\eqref{eq:residueExpansion1} we prove the proposition.
  
 \end{proof}

\begin{Corollary} \label{cor:analytCont}
   Power series in~\eqref{eq:diskPartitionComponent2} for different stability chambers are analytic continuations of each
   other in the parameter $\zeta = \Re(\theta)$. The analytic continuation is through the integral representation of the disk partition function.
   This analytic continuation can be understood as a particular case of the Crepant Transformation Conjecture~\cite{Ru99, Ru02, Ru06} for 
    toric Deligne-Mumford stacks proved in~\cite{CIJ}.
\end{Corollary}

\subsection{Disk partition functions and wall-crossing}

We defined the chamber disk partition function components corresponding to the phases of the GLSM
(cones of the secondary fan). The wall-crossing is given by the analytic continuation formula~\eqref{eq:diskPartitionComponent}.\\

We would like to extend it to wall-crossing for arbitrary branes. Let $\fB = \sum_{\St} c_{\St} \bfL_{\St} \in K([V/G])$. Then for
different characters $\St_{1}$ and $\St_{2}$ the Higgs-Coulomb correspondence~\eqref{eq:diskPartitionComponent2} hold for different domains $B \in U_{\St_{1}}$
and $B \in U_{\St_{2}}$ correspondingly. It could happen that $U_{\St_{1}} \cap U_{\St_{2}} = \mempty$. Then we need another description for analytic
continuation. This description is given in Definition~\ref{def:wall}  for general one wall crossing.\\

First of all, we need to explain the wall-crossing setup that we use. We follow the one stated, for example, in~\cite{BH2}.

Consider two maximal cones $C_+ , C_-$ of the secondary fan which are adjacent along a
codimension one wall $\fh^{\perp} \subset \fg^\vee$. Let $\fh := (\fh^{\perp})^{\perp} \subset \fg$
and pick its integral generator $h \in \bL, \; \fh=\bC h$. Also define
$h_i := \langle \D_i, h \rangle$. Then $\sum_{i=1}^{n+\boldsymbol{\kappa}} h_i  = 0$ since
$\sum_{i=1}^{n+\boldsymbol{\kappa}} \D_i=0$ (Calabi-Yau condition). 

In the terminology of the reference~\cite{BH2} $h = (h_{1}, \ldots, h_{n+{\boldsymbol{\kappa}}})$ defines
the circuit that is associated to the wall-crossing.  We define
$$I_{\pm} := \{ i \; | \; \pm h_{i} > 0  \}.$$
The circuit itself is $\{v_{i} \; : \; i \in I_{+} \cup I_{-}\}$.

Let $I_{0} \subset (I_{+} \cup I_{-})'$.
If $I = I_{0} \sqcup \{i_{0} \}$ is a minimal anticone of $C_{\pm}$ where
$i_{0} \in I_{\pm}$, then $I_{0} \sqcup \{i_{\pm}\}$ is an anticone of $C_{\pm}$ for each
$i_{\pm} \in I_{\pm}$. Modification along the circuit is obtained by replacing all the cones
of the form $I_{0} \sqcup \{i_{+}\}$ by the anticones $I_{0} \sqcup \{i_{-}\}$.
We denote by $\cA^{0}_{C_{0}}$ the set of all such $I_{0}$. 

Following~\cite{BH2} we call all the anticones of this form essential anticones. We have
$$ \cA^{\min}_{C_{\pm}} = \cA^{\ess}_{C_{\pm}} \sqcup \cA^{\noness}_{C_{\pm}}.$$
\begin{lemma}
  Nonessential anticones always contain at least one element from both $I_{+}$ and $I_{-}$.
\end{lemma}
\begin{proof}
  Any minimal anticone must contain at least one number from $I_{+} \cup I_{-}$ for dimensional
  reasons. If it does not contain numbers from $I_{-}$, then it is essential if
  it contains exactly one number from $I_{+}$ and if it contains more than 1, then the cones
  $C_{\pm}$ are adjacent along the wall of higher codimension.
\end{proof}

\begin{remark}
  We use anticones instead of the cones because they are better adjusted to partition functions
  even though they are completely interchangeable. Each minimal anticone corresponds to a torus fixed
  point of the corresponding toric variety.
\end{remark}

\begin{definition}[Wall hemisphere partition function] \label{def:wall}
	 Let  $C_{0} = C_{+} \cap C_{-}$ be a codimension one cone of the secondary fan and
	\begin{equation} \label{eq:GRR}
		|\langle B + \St, h \rangle| < \sum_{i}|\langle \D_{i}, h \rangle| /4.
	\end{equation}
	For each $I_{{\boldsymbol{\kappa}}-1} \in \cA^{0}_{C_{0}}$ and $m_{I_{{\boldsymbol{\kappa}}-1}} \in (\bZ_{\ge 0})^{{\boldsymbol{\kappa}}-1}$ we choose $p(m_{I_{{\boldsymbol{\kappa}}-1}}) \in \pi^{I_{{\boldsymbol{\kappa}}-1}}_{m_{I_{{\boldsymbol{\kappa}}-1}}}$ generic with the property that cardinality of
    $(p(m) \mp \bR_{\ge 0} h) \cap \bK^{\mathrm{eff}}_{I_{{\boldsymbol{\kappa}}-1} \cup \{i_{\pm}\}}(\alpha)$ is bounded by $c|m|$ for all $i_{\pm} \in I_{\pm}$
    (this property is satisfied if, for example, all $p(m)$ are on the same hyperplane)~\footnote{This property is rather technical and
      can be made even weaker. The meaning of this property is that $p(m)$ must not go to infinity in the direction of $\pm h$ too fast
    as $|m| \to \infty$}.

	Define the wall hemisphere partition function corresponding to the wall $C_0$:
	\begin{equation} \label{eq:wallPartitionFunction}
		Z_{D^{2}}(\bfL_\St)_{C_{0}} := \sum_{I_{{\boldsymbol{\kappa}}-1} \in \cA^{0}_{C_{0}}} Z_{I_{{\boldsymbol{\kappa}}-1}}^{\ess}(\bfL_{\St}) + \sum_{I \in \cA^{\noness}_{C_{\pm}}}Z_I(\bfL_{\St}),
	\end{equation}
	where
	\begin{equation}
		Z_{I}(\bfL_{\St}) = \sum_{q \in \bK^{\eff}_{I}(\alpha)} \mathrm{sign}\left(\frac{\dd \sigma}{\Lambda_{a=1}^{{\boldsymbol{\kappa}}}\D_{i_{a}}}\right) \int_{\cC^{I^{\mathrm{or}}}(q)} \dd \sigma \; \Gamma \cdot e^{\langle \theta + 2\pi\sqrt{-1} \St, \sigma \rangle}, \;\;\; I \in \cA^{\noness}_{C_{\pm}}.
	\end{equation}
	and
	\begin{multline} \label{eq:essential1}
		Z^{\ess}_{J}(\bfL_{\St}) = \sum_{m \in (\bZ_{\ge 0})^{{\boldsymbol{\kappa}}-1}}\biggl(-\mathrm{sign}\left(\frac{\dd \sigma}{\Lambda_{a=1}^{{\boldsymbol{\kappa}}-1}\D_{j_{a}}\wedge \Omega}\right) \int_{\cC^{J^{\mathrm{or}}}_{\Omega}(p(m))} \dd \sigma \; \Gamma \cdot e^{\langle
			\theta + 2\pi\sqrt{-1} t, \sigma \rangle} + \\ +  \sum_{i_{\pm} \in I_+ \cup I_-}\sum_{q \in (p(m) \mp \bR_{\ge 0} h) \cap \bK^{\mathrm{eff}}_{J \cup \{i_{\pm}\}}(\alpha)} \mathrm{sign}\left(\frac{\dd \sigma}{\Lambda_{a=1}^{{\boldsymbol{\kappa}}}\D_{i_{a}}}\right)\int_{\cC^{I^{\mathrm{or}}}(q)} \dd \sigma \; \Gamma \cdot e^{\langle\theta+2\pi\sqrt{-1}\St, \sigma
			\rangle} \biggr),
	\end{multline}
	where $\Omega$ is any volume form in $(\fg_{J})_{\bR}$. Note that in~\eqref{eq:essential1} the sum in the second sum is finite.
\end{definition}
\begin{remark}
  There seems to be an infinite number of choices of points $p(m)$
  in the definition of the wall hemisphere partition function. This is due to the fact that in the general case the
  cones in $\fg$ corresponding to essential anticones in the adjacent phases intersect and one cannot split the corresponding effective classes
  uniformly. As we will see below the definition is independent of these choices.
\end{remark}

\begin{remark}[Grade restriction rule] \label{cor:GRR}
  The formula~\eqref{eq:GRR} is called the {\em grade restriction rule} since it puts a restriction of
  the character $\St$ (grading of the brane).
  Disk partition function components can be analytically continued directly along the walls of the
  secondary fan if the grade restriction rule is satisfied as is stated in the following theorem.
\end{remark}

\begin{Theorem}[Wall-crossing] \label{th:wallCrossing}
  In the setting of Definition~\ref{def:wall} there exists a connected open set $U_{C_0} \subset \ga_{\bR}$ such that 
  $U_{C_0} \cap U_{C_{+}}$ and $U_{C_0} \cap U_{C_{-}}$ are nonempty and the wall hemisphere partition function converges for all $\zeta \in U_{C_0}$, $B$ satisfies~\eqref{eq:GRR} and for all $\zeta \in U_{C_0} \cap U_{C_{\pm}}$ it is equal to the chamber hemisphere partition function
  \begin{equation}
		Z_{D^{2}}(\bfL_{\St})_{C_{0}} = Z_{D^{2}}(\bfL_{\St})_{C_{\pm}}
  \end{equation}
  In particular, the wall hemisphere partition function does not depend on the choices in the definition.
\end{Theorem}
\begin{proof}

  The convergence of the non-essential part of the right hand side is established in Proposition~\ref{prop:apConv}.

  Consider the essential part. Let $\zeta \in C_{+} \cup C_{-}$. Then for each $J$ we can write $\zeta = \sum_{j \in J} \zeta^{J}_{j} \D_{j}$, where
  $\zeta^{J}_{j} > 0$.  Note that all $\D_{j}$ are in the wall hyperplane $(\bR h)^{\perp}$, so $\langle \D_{j}, h \rangle = 0$ and
  for any $c_{1}>0$ we can choose $\zeta$  far enough in the interior of $C_{+} \cup C_{-}$ such that
  $\langle \zeta, p(m)/|p(m)|\rangle < -c_{1}$ for all $m$. Moreover, the same inequality holds for all $q$ in~\eqref{eq:essential1}.
  Then we can use Proposition~\ref{prop:integralBound} to estimate the integral summands and the residue summands
  in~\eqref{eq:essential1} by $e^{-c_{2}|p(m)|}$ and $e^{-c_{2}|q|}$
  respectively. The number of summands for each $m$ is bounded by $c_{3} |m|$ because condition $q \in (p(m)\mp \bR_{\ge 0}h) \cap \bK^{\mathrm{eff}}_{J}
  \sqcup \{i_{\pm}\}$ is linear in $m$ if $p(m)$ are chosen as in the definition~\ref{def:wall}. 
 Therefore, The essential summand $Z^{\mathrm{ess}}_{J}(t)$ is bounded by 
 $\sum_{m_{J} \in (\bZ_{\ge 0})^{{\boldsymbol{\kappa}}-1} } c_{3} |m_J|e^{-c_{2}\mathrm{dist}(\pi^{J}_{m_{J}}, 0)} < \infty$. 
  Moreover, the convergence is an open condition on $\zeta$, so it must hold in an open neighbourhood
  This finishes the proof of convergence.

  Consider an essential integral term in~\eqref{eq:essential1}. If $\zeta \in C_{+}$ we can apply Corollary~\ref{cor:contDef}:
  \begin{equation}
    \cC^{J^{\mathrm{or}}}_{\Omega}(p(m)) = -\sum_{i_{\pm} \in I_{\pm}} \sum_{m_{i_{\pm}} \ge 0}
    \cC^{(J^{\mathrm{or}}, i_{\pm})}_{\iota_{h}\Omega}(\{p(m) + \bR_{\ge 0}h\} \cap \pi^{i_{\pm}}_{m_{i_{\pm}}}).
  \end{equation}
  By slight abuse of notation we can write $\Omega = c_{i_{\pm}} \D_{i_{\pm}} \in (\fg_{J})_{\bR}^{*}$, so $\iota_{h} \Omega = \pm c_{i_{\pm}}$. So
  \begin{equation}
    \mathrm{sign}\left(\frac{\dd \sigma}{\Lambda_{a=1}^{{\boldsymbol{\kappa}}-1}\D_{j_{a}}\wedge \Omega}\right) = \mathrm{sign}\left(c_{i_{\pm}}\frac{\dd \sigma}{\Lambda_{a=1}^{{\boldsymbol{\kappa}}-1}\D_{j_{a}}\wedge \D_{i_{\pm}}}\right),
  \end{equation}
  and therefore
  \begin{equation}
    -\mathrm{sign}\left(\frac{\dd \sigma}{\Lambda_{a=1}^{{\boldsymbol{\kappa}}-1}\D_{j_{a}}\wedge \Omega}\right)\cC^{J^{\mathrm{or}}}_{\Omega}(p(m)) = \pm \sum_{i_{\pm}} \sum_{q \in (p(m)+\bR_{\ge 0} h) \cap \bK^{\eff}_{J \cup \{i_{\pm}\}}(\alpha)}\mathrm{sign}\left(\frac{\dd \sigma}{\Lambda_{a=1}^{{\boldsymbol{\kappa}}-1}
   \D_{j_{a}}\wedge \D_{i_{\pm}}}\right)\cC^{(J^{\mathrm{or}}, i_{\pm})}(q).
  \end{equation}

  Combining this with the second term from~\eqref{eq:essential1} we obtain:
  \begin{equation}
    Z^{\ess}_{J}(\bfL_{\St}) = \sum_{i_{+}}Z_{J \cup \{i_{+}\}}(\bfL_{\St}).
  \end{equation}
  In the same way if $\zeta \in C_{-}$ we compute

  \begin{equation}
    Z^{\ess}_{J}(\bfL_{\St}) = \sum_{i_{-}}Z_{J \cup \{i_{-}\}}(\bfL_{\St}).
  \end{equation}
  The theorem is proved.~\footnote{We remark about the convergence again, the condition in Corollary~\ref{cor:contDef} is stricter than in the
  theorem, but equality of analytic functions extends to the maximal domain where both converge.}
\end{proof}

\appendix
\section{Convergence of multivariate hypergeometric functions}\label{appendix-A} 

This is mostly known due to many people:~\cite{Horn} or, in the more systematic exposition \cite{GKS},~\cite{SadykovTsikh} and others. The authors were not able to find some of the results,
so we provide a short overview of the subject here. We use the notations of the main part of the
paper, particularly Section~\ref{sec:Coulomb}.\\

Let $\theta = \zeta+2\pi\sqrt{-1}B$ and $\sigma = \tau+\sqrt{-1}\nu$ represent complex variables
on $\fg^\vee$ and $\fg$ correspondingly
and $\D_{i} \in \bL^{\vee}, \; i \le n+{\boldsymbol{\kappa}}$ be a collection of the vectors that spans $\fg^\vee$
over $\bC$ such that $\sum_{i=1}^{n+\boldsymbol{\kappa}} \D_{i} = 0$ (Calabi-Yau condition).
Let $\alpha \in \bC^{n+{\boldsymbol{\kappa}}}$ be a generic vector and
\begin{equation}
  \Gamma = \Gamma(\sigma) = \prod_{i=1}^{n+{\boldsymbol{\kappa}}} \Gamma(\langle \D_{i}, \sigma \rangle
  + \alpha_{i}).
\end{equation}

We remind that
the secondary fan $\Sigma^{2}$ is defined as a fan with $\Sigma^{2}(1) =
\{\D_{i}\}_{i=1}^{n+{\boldsymbol{\kappa}}}$ and whose maximal cones are all possible intersections of
$\angle_{I}$ of dimension ${\boldsymbol{\kappa}}$. 

First of all we need some basic results about the gamma function. 
The Stirling approximation:
\begin{equation}
  \Gamma(x+iy) = (2\pi)^{1/2} z^{z-1/2} e^{-z}(1+O(\frac{1}{|z|})), \;\;\; |\mathrm{Arg}(z) | < \pi -\alpha
\end{equation}
where $\alpha>0$, and the Landau notation is
$$ 
f(z) = O(\frac{1}{|z|}) \text{ is equivalent to } \limsup_{|z| \to \infty} |z| |f(z)| < \infty
$$
In particular,
\begin{equation}\label{eqn:stirling} 
  |\Gamma(x+iy)| = (2\pi)^{1/2} |z|^{x-1/2} e^{-x} e^{-y\mathrm{Arg}(z)}(1+O(\frac{1}{|z|})), \;\;\; |\mathrm{Arg}(z)| < \pi -\alpha.
\end{equation}
Now let $S_{\delta} = \bR \backslash \bigcup_{n=0}^{\infty} (-n-\delta, -n+\delta)$ be a domain in $\bR$ separated
from poles of the gamma function by a small positive constant $\delta$.

\begin{lemma} \label{lem:gammaBound}
  Let $z = x+\sqrt{-1}y$ and $x \in S_{\delta}$. Then
  \begin{equation}
    \label{eq:gammaBound}
    \Gamma(x) < const \cdot |x|^{x-\frac12} e^{-x}.
  \end{equation}
  \begin{equation}
  	\label{eq:gammaBoundCplx}
  	|\Gamma(z)| < const \cdot |z|^{x-\frac12} e^{-\mathrm{min}(x,0)} e^{-\frac{\pi|y|}{2}}.
  \end{equation}
  In addition, if $x \in S_\delta \cap K$, where $K$ is a compact set, then
  \begin{equation} \label{eq:imAs} 
  |\Gamma(z)| < const (|y|+1)^{x-1/2} e^{-\frac{\pi |y|}2}. 
  \end{equation}

\end{lemma}
\begin{proof} By Equation \eqref{eqn:stirling}, there exist positive constants $c_1$, $c_2$ such that 
	\begin{equation} \label{eq:positiveArgumentBounds}
	c_{1}|z|^{x-\frac12} e^{-x} e^{-y\mathrm{Arg}(z)} < |\Gamma(z)| < c_{2} |z|^{x-\frac12} e^{-x}e^{-y\mathrm{Arg}(z)}, \;\;\; x > \delta.
	\end{equation}
	If $x > 0$, then $x>\delta$ since $x\in S_\delta$. We can write $\mathrm{Arg}(z) = \mathrm{Arctan}(y/x) = \mathrm{sign}(y)\pi/2-\mathrm{Arctan}(x/y)$, so we have
	$$
	-y\mathrm{Arg}(z) = -y(\mathrm{sign}(y)\pi/2 - \mathrm{Arctan}(x/y)) < \frac{\pi |y|}{2} + x, 
	$$
	since $\mathrm{Arctan}(x/y) < x/y$.
	Therefore,
	\begin{equation} \label{eq:positiveArgumentBoundCplx}
		|\Gamma(z)| < c_{2} |z|^{x-\frac12} e^{-\frac{\pi|y|}{2}}, \;\;\; x > \delta.
	\end{equation}

  If $x < 0$, then $|\sin(\pi x)|> \sin \delta$ since $x\in S_\delta$. We use the reflection formula and~\eqref{eq:positiveArgumentBounds}:
  \begin{equation} 
    |\Gamma(z)| = \frac{\pi}{|\Gamma(1-z)\sin(\pi z)|} < \frac{\pi}{c_1}  |1-z|^{-(1-x)+1/2}e^{1-x} e^{y\mathrm{Arg}(1-z)} |\sin(\pi z)|^{-1}, \;\;\; x < 0.
  \end{equation}
  If $x < 0$ then $|1-z| > |z|$, $|1-z|^{x-1/2} < |z|^{x-1/2}$, and $|\mathrm{Arg}(1-z)| < \frac{\pi}{2}$, so  
  \begin{equation} \label{eq:negativeArgumentBounds}
  	|\Gamma(z)| < \frac{\pi e}{c_1} |z|^{x-1/2}e^{-x} e^{\frac{\pi |y|}{2}}|\sin(\pi z)|^{-1}.
  \end{equation}
  The first formula~\eqref{eq:gammaBound} follows from~\eqref{eq:positiveArgumentBounds} and~\eqref{eq:negativeArgumentBounds}
  for $y=0$. 

We have $\sin(\pi z) = \sin(\pi(x+\sqrt{-1}y)) = \sin(\pi x)\cosh(\pi y) + \sqrt{-1}\cos(\pi x)\sinh(\pi y)$, 
$$
|\sin(\pi z)| \geq |\sin(\pi x) \cosh(\pi y)| >  \sin(\pi\delta)  \frac{e^{\pi y}+ e^{-\pi y}}{2} > \frac{1}{2} \sin(\pi\delta) e^{\pi|y|}. 
$$
  Finally, we use this to simplify~\eqref{eq:negativeArgumentBounds}:
  \begin{equation} \label{eq:negativeArgumentBounds1}
  |\Gamma(z)| < const \cdot |z|^{x-1/2}e^{-x} e^{\frac{-\pi |y|}{2}}.
  \end{equation}
  
  Collecting the formulas~\eqref{eq:positiveArgumentBoundCplx} and~\eqref{eq:negativeArgumentBounds1} we get the inequality~\eqref{eq:gammaBoundCplx}:
  \begin{equation} \label{eq:upperBdCplx}
    |\Gamma(z)| \le const \cdot |z|^{x-1/2}e^{-\min(x,0)}e^{-\pi |y|/2}, \;\;\; x \in S_{\delta},
  \end{equation}
  where the constant depends on $\delta$.
  
  The third formula~\eqref{eq:imAs} follows from~\eqref{eq:upperBdCplx} and $(|y|+const)^x \sim |y|^x, \; y \to \infty$.
  
\end{proof}

We also recall the multivariate Cauchy-Hadamard theorem.
\begin{Theorem}[Cauchy-Hadamard]
  Let
  \begin{equation}
    \sum_{m \in (\bZ_{\ge 0})^{{\boldsymbol{\kappa}}}} c_{m} z^{m}.
  \end{equation}
  The series (absolutely) converges in the polydisk with the multiradii $r = (r_{1},
  \ldots, r_{{\boldsymbol{\kappa}}})$ if
  \begin{equation}
    \lim_{N \to \infty}\mathrm{sup}_{m, |m| = N}(c_{m}r^{m})^{1/|m|} \le 1
  \end{equation}
  and such polydisk is maximal if the left hand side is equal to the right hand side.
\end{Theorem}

\begin{Proposition}\label{prop:apConv}
  \begin{enumerate}
    \item The domain of convergence of

        \begin{equation} \label{eq:conv2}
          Z_{I} = \sum_{m \in (\bZ_{\ge 0})^{{\boldsymbol{\kappa}}}}
          \prod_{i' \in I'} \Gamma(\langle \D_{i'}, \sigma_{m} \rangle + \alpha_{i'})
          \prod_{i \in I} \frac{(-1)^{m_{i}}}{m_{i}!} \exp(\langle \theta,
          \sigma_{m} \rangle),
        \end{equation}

      is non-empty. Moreover, if the series converges at $\zeta_{0}$, then it also converges if $\zeta \in \{\zeta_{0}\} - \angle_{I}$.
    \item The domain of convergence of the series is $U_{I} +\sqrt{-1} \bR^{{\boldsymbol{\kappa}}}$, where  $U_{I} \subset \bR^{{\boldsymbol{\kappa}}}$
      is defined by the constraints
          \begin{equation}
            \langle \zeta + \log\Psi(\sigma),  \sigma \rangle = \sum_{i \in I} (\langle \zeta, \D^{*,I}_{i}\rangle + \log \Psi_{i}(\sigma))\sigma_{i} > 0, \; \sigma \in \angle^{*}_{I},
          \end{equation}
      where $\Psi(\sigma) = (\Psi_{1}(\sigma), \ldots, \Psi_{{\boldsymbol{\kappa}}}(\sigma))$ is the Horn vector defined below in the proof. 

  \end{enumerate}
\end{Proposition} 
\begin{proof}
  In this proof we work in the basis $\{\D^{*,I}_{i}\}_{i \in I}$ on $\fg$ and the corresponding coordinate system on $\fg$ and $\fg^{\vee}$. That is
  if $f \in \fg$ and $f^{\vee} \in \fg^{\vee}$, we write $f = \sum_{i} f_{i} \D^{*,I}_{i}, \; f^{*} = \sum_{i} f^{*}_{i_{1}} \D_{i}$, where
  $f_{i} := \langle \D_{i}, f \rangle$ and $f^{*}_{i} := \langle f^{*}, \D^{*I}_{i} \rangle$.

  Let $s^{I}_{i' i} := \langle \D_{i'}, \D^{*I}_{i} \rangle$. Consider the asymptotics of the series
  \eqref{eq:conv2}. We write the series as
  \begin{equation} \label{eq:conv3}
          Z_{I} = \sum_{m \in (\bZ_{\ge 0})^{{\boldsymbol{\kappa}}}}
          \prod_{i' \in I'} \Gamma(-\sum_{i} s^I_{i' i}(m_{i}+\alpha_{i}) + \alpha_{i'})
          \prod_{i \in I} \frac{(-1)^{m_{i}}}{m_{i}!} \exp(\langle \theta,
          \sigma_{m} \rangle) = \exp(-\sum_{i}  \theta_{i} \alpha_{i})  \sum_{m \in (\bZ_{\ge0})^{{\boldsymbol{\kappa}}}} c_{m} z^{m},
  \end{equation}
  where
  $z^{m} = \prod_{i=1}^{{\boldsymbol{\kappa}}} z_{i}^{m_{i}} =
  \prod_{i} \exp(-m_{i}\theta_{i})$.
  For generic $\alpha$ arguments of the gamma functions are uniformly separated from $\bZ_{\le 0}$ by a number $\delta>0$ because $s^I_{i'i}$ are
  rational numbers.

  Therefore we can apply Lemma~\ref{lem:gammaBound} and write the upper bound on the series expansion coefficients:
  \begin{equation}
    \begin{aligned}
      |c_{m}| \le const \cdot \frac{\prod_{i' \in I'} (s^{I}_{i'}m+c_{i'})^{(s^{I}_{i'}m+c_{i'}-1/2)}}{\prod_{i\in I} m_{i}^{m_{i}-1/2}}, \\ \;\;\; \text{ where }
      s^{I}_{i'}m = \sum_{i\in I} s^{I}_{i'i}m_{i} \text{ and } c_{i'} = \sum_{i\in I}s^{I}_{i'i}\alpha_{i}+\alpha_{i'}.
    \end{aligned}
  \end{equation}

  Thus we can write
  \begin{multline} \label{eq:CH1}
    \lim_{N \to \infty} \sup_{m, |m| \ge N} (c_{m} r^{m})^{1/|m|} \le \\ \le \lim_{N \to \infty} \sup_{m,|m|\ge N} 
    \left(const\prod_{i'\in I'}(s_{i'}^I m+c_{i'})^{c_{i'}-1/2} \prod_{i}m_{i}^{-1/2}\cdot
      \exp\left(- \langle \zeta, m \rangle - \sum_{i'\in I'}s^{I}_{i'}m \log(s^{I}_{i'}m + c_{i'}) - \sum_{i} m_{i}\log    m_{i}\right)\right)^{1/|m|},
  \end{multline}
  where $|m|$ denotes any norm on $\bR^{n}$.

    We define the Horn vector $\Psi(\sigma) := \sum_{i} \Psi_{i}(\sigma) \D_{i}$, where
  \begin{equation}
    \Psi_{i}(\sigma) := \frac{\sigma_{i}}{\prod_{i' \in I'}
    (s^{I}_{i'}\sigma)^{s^{I}_{i'i}}}.
  \end{equation}
  In particular, $\log \Psi(\sigma) = \sum_{i} \log \Psi_{i}(\sigma) \D_{i}$, where $\log \Psi_{i}(\sigma) = \log(\sigma_{i})-\sum_{i'\in I'} s^{I}_{i'i} \log(s^{I}_{i'}\sigma)$. We note that even though $\log \Psi_{i}(\sigma)$ is not defined if any of $\sigma_{i}, s_{i'}^I\sigma
   = 0$, the sum $\sum_{i} \sigma_{i} \log \Psi_{i}(\sigma) = \langle \log \Psi(\sigma), \sigma \rangle$ can be defined as a limit.

  Under the assumption of the proposition $\zeta = \Re(\theta)$ satisfies the equation
  \begin{equation} \label{eq:radius2}
    \langle \zeta + \log \Psi(\sigma), \sigma \rangle > 0
  \end{equation}
  for all $\sigma \in \bR^{{\boldsymbol{\kappa}}} \backslash \{0\}$,
  so we have
  \begin{multline}
    \label{eq:exponential}
-\langle \zeta, m \rangle - \sum_{i'\in I'}s^{I}_{i'}m \log(s^{I}_{i'}m + c_{i'}) - \sum_{i\in I} m_{i}\log m_{i} \le \\ \le
    -\langle \zeta, m \rangle - \sum_{i'\in I'}s^{I}_{i'}m \log(s^{I}_{i'}m + c_{i'}) - \sum_{i\in I} m_{i}\log m_{i} + \langle \zeta + \log \Psi(m), m\rangle = -\sum_{i'\in I'} s_{i'}^{I}m \log(1+c_{i'}/(s_{i'}^{I}m)),
  \end{multline}
  where by slight abuse of notation we identify $m = \sum_{i\in I} m_{i} \D^{*,I}_{i}$. Consider the function $x \log(1+c/x)$. As $x \to 0$ we have
  $x \log(1+c/x) \sim x \log(c/x) \to 0$, and when $x \to \infty$ then $\log(1+c/x) = c/x + O(1/x^{2})$, so $x \log(1+c/x) \to c$.
  Therefore the last expression in~\eqref{eq:exponential} is bounded from above and below, so we can write
  \begin{equation} \label{eq:CH2}
     \lim_{N \to \infty} \sup_{m, |m| \ge N} (c_{m} r^{m})^{1/|m|} \le \lim_{N \to \infty} \mathrm{sup}_{m, |m| \ge N} \exp(const \prod_{i'\in I'}(s^I_{i'}m+c_{i'})^{c_{i'}-1/2} \prod_{i\in I}m_{i}^{-1/2})^{1/|m|} \le 1.
  \end{equation}
  so the series~\eqref{eq:conv2} converges by the multivariate Cauchy-Hadamard theorem. If for some $\sigma$ the inequality~\eqref{eq:radius2}
  is not satisfied, then it is not satisfied for all $m$ proportional to $\sigma$ and the limit in~\eqref{eq:CH2} is greater than 1. This proves the second claim of the proposition.

  $\prod_{i} \Psi_{i}(\sigma)^{\sigma_{i}}$ is bounded and separated from $0$ on the unit sphere, so $\sum_{i} \log \Psi_{i}(\sigma)$ is bounded on the
  same domain below by a constant $-N$. First claim of the proposition is proved by choosing $\zeta = N\sum_{i} \D^{*,I}_{i}$.

\end{proof}

\begin{notation}
  Below we work with estimates including many constants whose exact values are of no importance and can be rather cumbersome. Notations $const, const_{i}$ denote various such constants.
\end{notation}

Below we present the proof of Lemma~\ref{lem:convImaginary} and Corollary~\ref{cor:contDef}.

\begin{proof}[\bf Proof of Lemma~\ref{lem:convImaginary}]
   
 We use the formula~\eqref{eq:imAs} from Lemma~\ref{lem:gammaBound} to estimate the integrand:
  \begin{equation} \label{eq:imAs1}
    |\prod_{i=1}^{n+{\boldsymbol{\kappa}}} \Gamma(\langle \D_{i}, \sigma \rangle +\alpha_i) e^{\langle \theta, \sigma \rangle}| < const \cdot \prod_{i \notin I_{k}}
    \left( |\langle \D_{i}, \Im(\sigma) \rangle|+1\right)^{\langle \D_{i}, \Re(\sigma) \rangle + \alpha_{i}} \exp\left( -2\pi \langle B, \Im(\sigma) \rangle - \pi/2 \sum_{i \notin I_{k}}\langle \D_{i}, \Im(\sigma) \rangle \right),
  \end{equation}
  for $\sigma$ on the integration contour. 

  We notice that condition~\eqref{eq:convCondition1} implies that
  $$ |\langle B, \nu \rangle| - \frac14 \sum_{i=1}^{n+{\boldsymbol{\kappa}}}\langle \D_{i}, \nu \rangle < -const \cdot |\nu|,$$
for some positive $const$ since the expression is homogeneous of degree 1 in $|\nu|$.

  Therefore, expression in the exponential in~\eqref{eq:imAs1} is bounded above by  $-c |\Im(\sigma)|$ for some constant $c>0$ (in any norm on $\fg$) under the assumption of the lemma, so the
  absolute value of the integrand is bounded by an exponentially decaying function.
\end{proof}

\begin{Proposition} \label{prop:integralBound}
Let $k < {\boldsymbol{\kappa}}, I^{\mathrm{or}}_{k} = (i_{1}, \ldots, i_{k})$ be such that $\{\D_{i}\}_{i \in I_{k}}$ are linearly independent.
Consider a cycle $\cC^{I^{\mathrm{or}}}_{\Omega}(p)$ where $p \in \pi^{I_{k}}_{m_{I_{k}}}$ and is separated from all other polar hyperplanes by positive number $\delta$. Let $B$ be such that the integral
\begin{equation} \label{eq:integral1}
  \mathrm{Int}(p) := \int_{\cC^{I^{\mathrm{or}}_{I_{k}}}_{\Omega}(p)} \dd \sigma \, \Gamma \cdot e^{\langle \theta, \sigma \rangle}
\end{equation}
converges absolutely. There exist constants $c_{1}, c_{2} > 0$ such that if $\langle \zeta, p/|p|
\rangle < -c_{1}$ then $|\mathrm{Int}(p)| < e^{-c_{2}|p|}, \; |p| \gg 0$.
\end{Proposition}
\begin{proof}
    Let us fix $I$ and work in the basis given by $\{\D_{i}\}_{i \in I}$.
    Using Fubini theorem we write
    \begin{multline}
      \mathrm{Int}(p) = \pm \int_{\bR^{{\boldsymbol{\kappa}}-k}} \int_{(S^{1})^{k}}\dd \sigma \, \prod_{i=1}^{n+{\boldsymbol{\kappa}}}\Gamma \cdot e^{\langle
      \theta, \sigma \rangle} = \\ = const \prod_{i \in I_{k}}\frac{(-1)^{m_{i}}}{m_{i}!} \int_{\bR^{{\boldsymbol{\kappa}}-k}} \Omega \, \prod_{i \notin I_{k}}
      \Gamma(\langle \D_{i}, p + \sqrt{-1}\Im(\sigma) \rangle  + \alpha_{i})  e^{\langle
      \theta, \sigma \rangle}.
  \end{multline}

  Let $p_{i} := \langle \D_{i}, p\rangle$ and $y_{i} := \Im(\langle \D_{i}, \sigma \rangle)$, $|y| := |\Im(\sigma)|$.
  Consider the main asymptotics of the Gamma functions without argument shift by $\alpha$:
  \begin{equation} \label{eq:mainAsympt}
    \prod_{i=1}^{n+\boldsymbol{\kappa} }|p_{i}|^{-p_{i}} := \mathrm{exp}\left(-\sum_{i \notin I_{k}} p_{i}\log |p_{i}| + \sum_{i \in I_{k}} (m_{i}+\alpha_{i}) \log (m_{i}+\alpha_{i})\right),
  \end{equation}
  where the left hand side is defined as a limit if any $p_{i} = 0$.
  Calabi-Yau condition $\sum_{i=1}^{n+\boldsymbol{\kappa}} \D_{i} = 0$ implies that~\eqref{eq:mainAsympt} is a homogeneous function of degree 0 in $p$.
  We use it to simplify the asymptotics of the integrand:
  \begin{multline} \label{eq:intHornAsymptotics}
    \prod_{i \in I_{k}} m_{i}^{-m_{i}} \prod_{i \notin I_{k}}|p_{i} + \alpha_{i} + \sqrt{-1}y_{i}|^{p_{i}} \prod_{i} |p_{i}|^{-p_{i}} = \\ =
    \prod_{i \in I_{k}}(m_{i}+\alpha_{i})^{\alpha_{i}}\mathrm{exp}\left( \sum_{i \notin I_{k}} p_{i} \log |1 + (\alpha_{i} + \sqrt{-1} y_{i})/p_{i}| +\sum_{i \in I_{k}} m_{i}\log(1+\alpha_{i}/m_{k})\right) \le \\
    \le e^{const \cdot |p|}\mathrm{exp}\left( \sum_{i \notin I_{k}} p_{i} \log |1 + (\alpha_{i} + \sqrt{-1} y_{i})/p_{i}| \right), \; |p| \to \infty.
  \end{multline}
  In the last inequality we used the fact that $(m_{i}+\alpha_{i})^{\alpha_{i}}$ is bounded by a polynomial in $|p|$ and $m_{i} \log(1+\alpha_{i}/m_{i}) < \alpha_{i}$.

  Analogously to the proof of Lemma~\ref{lem:convImaginary} we have
  \begin{equation} \label{eq:imaginaryAsympt1}
    e^{-2\pi\langle B, \Im(\sigma) \rangle - \sum_{i \notin I_{k}}\pi |y_{i}|/2} \le e^{-const_{2}|y|}.
  \end{equation} 
  Now we apply Lemma~\ref{lem:gammaBound} to~\eqref{eq:integral1}: 
    \begin{multline}
    \label{eq:intEstimate1}
    \prod |p_{i}|^{-p_{i}} |\mathrm{Int}(p)| \le \prod_{i} |p_{i}|^{-p_{i}} \prod_{i \in I_{k}} m_{i}^{-m_{i}}e^{-\sum_{i \in I_{k}} m_{i}} e^{\langle \zeta, p \rangle} \times \\ \times \int_{\bR^{{\boldsymbol{\kappa}}-k}} \Omega \;
    \prod_{i \notin I_{k}}|p_{i} + \alpha_{i} + \sqrt{-1}y_{i}|^{p_{i}+\alpha_{i}-1/2} e^{-\sum_{i \notin I_{k}}\min(p_{i}+\alpha_{i},0)}
    e^{-2\pi\langle B, \Im(\sigma) \rangle-\sum_{i \notin I_{k}} \pi |y_{i}|/2} \le \\ \le e^{const_{4} |p|  + \langle \zeta, p\rangle} e^{\sum_{i} \alpha_{i}} \int_{\bR^{{\boldsymbol{\kappa}}-k}} \Omega \;  \prod_{i \notin I_{k}}|p_{i} + \alpha_{i} + \sqrt{-1}y_{i}|^{\alpha_{i}-1/2}\mathrm{exp}\left( \sum_{i \notin I_{k}} p_{i} \log |1 + (\alpha_{i} + \sqrt{-1} y_{i})/p_{i}| \right) e^{-const_{2}|y|},
  \end{multline}
  where in the last line we used~\eqref{eq:intHornAsymptotics} and~\eqref{eq:imaginaryAsympt1}.

  We use the obvious inequality $(a+b)^{x} \le (2a)^{x}+(2b)^{x}$ for $a,b > 0, \; x \in \bR \backslash \{0\}$ to write $|a+\sqrt{-1}b|^{x} = (a^{2}+b^{2})^{x/2} \le
  2^{x/2}(|a|^x + |b|^{x})$. Furthermore, if $p_{i} < 0$,
  \begin{equation}
    |1 + (\alpha_{i} + \sqrt{-1} y_{i})/p_{i}|^{p_{i}} < |1+\alpha_{i}/p_{i}|^{p_{i}}.
  \end{equation}
  This estimate does not depend on $y$ and is subexponential in $p$ if $p_{i}$ is
  separated from $-\alpha_{i}$, so we can ignore these terms in the estimate.
  Then
\begin{multline} \label{eq:intAsympt2}
    \prod_{i \notin I_{k}, \alpha_{i}-1/2 > 0}|p_{i} + \alpha_{i} + \sqrt{-1}y_{i}|^{\alpha_{i}-1/2}\mathrm{exp}\left( \sum_{i \notin I_{k}, p_{i} >0} p_{i} \log |1 + (\alpha_{i} + \sqrt{-1} y_{i})/p_{i}| \right) < \\
    < 2^{\sum_{i \notin I_{k}} (p_{i}+\alpha_{i}-1/2)/2}\prod_{i \notin I_{k}, \alpha_{i}-1/2>0} \left(|p_{i}+\alpha_{i}|^{\alpha_{i}-1/2}+ |y_{i}|^{\alpha_{i}-1/2}\right) \prod_{i \notin I_{k}, p_{i} > 0}((1+|\alpha_{i}/p_{i}|)^{p_{i}} +
    |y_{i}|^{p_{i}}/|p_{i}|^{p_{i}}) \le \\
    \le 2^{\sum_{i \notin I_{k}} (p_{i}+\alpha_{i}-1/2)/2}\prod_{i \notin I_{k}, \alpha_{i}-1/2>0} \left(|p_{i}+\alpha_{i}|^{\alpha_{i}-1/2}+ |y_{i}|^{\alpha_{i}-1/2}\right) \prod_{i \notin I_{k}, p_{i} > 0}(e^{|\alpha_{i}|} +
    |y_{i}|^{p_{i}}/|p_{i}|^{p_{i}}),
  \end{multline}

  where in the last inequality we use $(1+c/x)^{x} = e^{x\log(1+c/x)} < e^{c}, \; c,x > 0$. The first product in the last line has polynomial behaviour
  in $p_{i}$ so it is bounded by a polynomial in $|p|$. Also, factors $2^{p_{i}}$ are bounded by $e^{const |p|}$. Let us focus on the
  only non-trivial terms containing $|y_{i}|^{p_{i}}/p_{i}^{p_{i}}$.
  We drop powers of $|y_{i}|^{\alpha_{i}-1/2}$ since they will produce the same exponential
  asymptotics.
  In order to prove the claim of the proposition we prove that each integral
  \begin{equation}
    \label{eq:intAsympt3}
    \int_{\bR^{{\boldsymbol{\kappa}}-k}} \Omega \;  \prod_{i \in J \subset I'_{k}} |y_{i}|^{p_{i}}/p_{i}^{p_{i}} e^{-const_{2}|y|}
  \end{equation}
  is bounded above by $e^{const |p|}$. Let $\rho = |y|$.   Then in the polar coordinates on $\bR^{{\boldsymbol{\kappa}}-k}$ we write
  \begin{multline}
    \label{eq:intAsympt4}
    \int_{\bR^{{\boldsymbol{\kappa}}-k}} \Omega \;  \prod_{i \in J \subset I'_k } |y|^{p_{i}}/p_{i}^{p_{i}} e^{-const_{2}|y|} = const
    \prod_{i} p_{i}^{-p_{i}}\int_{0}^{\infty} \dd \rho \; \rho^{{\boldsymbol{\kappa}}-k-1 + \sum_{i} p_{i}} e^{-const_{2} \rho} = \\
    = const \cdot const_{2}^{k-{\boldsymbol{\kappa}}-\sum_{i} p_{i}} \prod_{i} p_{i}^{-p_{i}} \Gamma(\sum_{i} p_{i} + {\boldsymbol{\kappa}}-k).
  \end{multline}   If $\sum_{i} p_{i}$ is small, then so is the integral. Otherwise we can use the Stirling approximation.
   $\Gamma(\sum_{i} p_{i}+ {\boldsymbol{\kappa}}-k)/\Gamma(\sum_ip_{i})$ is bounded by a polynomial in $|p|$, so we need to prove that
  $$ \Gamma(\sum_{i} p_{i}) \prod_{i}p_{i}^{-p_{i}} < e^{const |p|}. $$ 
  Stirling approximation implies:
  \begin{equation}
    \Gamma(\sum_{i} p_{i}) < const (\sum_{i} p_{i})^{\sum_{i} p_{i}-1/2}e^{-\sum_{i} p_{i}}.
  \end{equation}

  \begin{lemma}\label{lemma-a} 
      Let $a_{1}, \ldots, a_{k} > 0$ and $a = \sum_{i=1}^k a_{i} >0$. Then
    \begin{equation} \label{eq:leqka}
    a^{a}\prod_{i} a_{i}^{-a_{i}} \leq k^a
    \end{equation}
  \end{lemma}
  \begin{proof} Equation \eqref{eq:leqka} is equivalent to
  \begin{equation}\label{eq:leqk}
  \prod_{i=1}^k \left(\frac{a_i}{a}\right)^{-\frac{a_i}{a}} \leq k
    \end{equation}
 Define $f: [0,\infty)^k \to \bR$ by
 $$
 f(b_1,\ldots,b_k) = \begin{cases}
 \prod_{i=1}^k b_i^{-b_i} & \text{if } b_1,\ldots, b_k>0,\\
 0 & \text{if } b_i=0 \text{ for some }i.
 \end{cases}
 $$
 Then $f$ is continuous, and is smooth on $(0,\infty)^k$.  Let 
 $\Delta_k:=\{ (b_1,\ldots,b_k)\in\bR: b_i\geq 0, \sum_{i=1}^k b_i  =1\}$   
 be the $(k-1)$-simplex. Then $f$ is positive in the interior of $\Delta_k$ and is zero on  $\partial\Delta_k$. Using Lagrange multiplier we compute that
 $$
 \max_{(b_1,\ldots, b_k)\in \Delta_k} f(b_1,\ldots,b_k) = f(\frac{1}{k},\ldots, \frac{1}{k}) = k.
 $$ 
 Therefore, \eqref{eq:leqk} holds whenever $a_1,\ldots, a_k>0$ and $a=\sum_{i=1}^k a_i$. 
  \end{proof}

  Applying Lemma \ref{lemma-a} we find that
  \begin{equation}
    \Gamma(\sum_{i} p_{i}) \prod_{i} p_{i}^{-p_{i}} \le e^{const (\sum_{i} p_{i})} \le e^{const |p|}.
  \end{equation}
  Returning to~\eqref{eq:intEstimate1} we find
  \begin{equation}
    \prod_{i} |p_{i}|^{-p_{i}} |\mathrm{Int}(p)| \le e^{\langle \zeta, p \rangle} e^{const |p|}, \; |p| \to \infty,
  \end{equation}
  or
  \begin{equation}
    |\mathrm{Int}(p)| < e^{\langle \zeta, p \rangle + \sum_{i}p_{i}\log|p_{i}|+const|p|} = e^{(\langle \zeta, p/|p| \rangle  +\sum_{i} p_{i}/|p|\log |p_{i}|/|p|)+const)|p|},
  \end{equation}
  where we used that $\prod_{i}|p_{i}|^{-p_{i}}$ is homogeneous of degree 0.
  Therefore if $\langle \zeta , p/|p|\rangle + \sum_{i}p_{i}/|p|\log (|p_{i}|/|p|) < -const$ then the integral is bounded by $e^{-c_{2} |p|}$ for
  large $|p|$ and $c_{2}$ that does not depend on $p$. We note that $\sum_{i}p_{i}\log |p_{i}|$ is a continuous function, so it is bounded and the estimate is satisified if
  $\langle \zeta, p/|p| \rangle < -c_{1}$ for a real number $c_{1}$ that does not depend on $p$.

\end{proof}

\end{document}